\documentclass[12pt,psamsfonts,leqno,oneside,letterpaper]{amsart}
\usepackage[dvips,text={6.5truein,9truein},left=1truein,top=1truein]{geometry}
\usepackage{amssymb,amsmath,amscd,enumerate}
\usepackage[pdftex]{graphicx}
\usepackage{url}

\usepackage[colorlinks,linkcolor=blue,citecolor=blue,pdfstartview=FitH]{hyperref}
\input xy
\xyoption{all}
\SelectTips{cm}{12}

\usepackage{tikz}

\usepackage{color}
\definecolor{bettergreen}{rgb}{0,0.6,0.4}

\newcommand\x{\begingroup\color{red}}

\DeclareFontFamily{U}{wncy}{}
\DeclareFontShape{U}{wncy}{m}{n}{<->wncyr10}{}
\DeclareSymbolFont{mcy}{U}{wncy}{m}{n}
\DeclareMathSymbol{\El}{\mathord}{mcy}{"4C}

\theoremstyle{definition}
\newtheorem{para}{}[subsection]

\newtheorem{remark}[para]{Remark}

\newtheorem{remarks}[para]{Remarks}
\newtheorem{notation}[para]{Notation}
\newtheorem{convention}[para]{Convention}
\newtheorem{definition}[para]{Definition}
\newtheorem{definitions}[para]{Definitions}
\newtheorem{notationremarks}[para]{Notation and Remarks}
\newtheorem{definitionsremarks}[para]{Definitions and Remarks}
\newtheorem{claim}[equation]{Claim}
\newtheorem{plainclaim}[equation]{}
\newtheorem{numberedfact}[para]{Fact}

\newcommand\Alternatives{\begin{enumerate}[(i)]}
\newcommand\EndAlternatives{\end{enumerate}}
\newcommand\Conditions{\begin{enumerate}[(1)]}
\newcommand\EndConditions{\end{enumerate}}

\theoremstyle{plain}
\newtheorem{theorem}[para]{Theorem}
\newtheorem{lemma}[para]{Lemma}
\newtheorem{proposition}[para]{Proposition}
\newtheorem{corollary}[para]{Corollary}
\newtheorem{conjecture}[para]{Conjecture}

\numberwithin{equation}{para}
\numberwithin{figure}{section}

\newcommand\Number{\begin{para}}
\newcommand\EndNumber{\end{para}}
\newcommand\Definition{\begin{definition}}
\newcommand\EndDefinition{\end{definition}}
\newcommand\Definitions{\begin{definitions}}
\newcommand\EndDefinitions{\end{definitions}}
\newcommand\Theorem{\begin{theorem}}
\newcommand\EndTheorem{\end{theorem}}
\newcommand\Conjecture{\begin{conjecture}}
\newcommand\EndConjecture{\end{conjecture}}
\newcommand\Remark{\begin{remark}}
\newcommand\EndRemark{\end{remark}}
\newcommand\Remarks{\begin{remarks}}
\newcommand\EndRemarks{\end{remarks}}
\newcommand\Convention{\begin{convention}}
\newcommand\EndConvention{\end{convention}}
\newcommand\Notation{\begin{notation}}
\newcommand\EndNotation{\end{notation}}
\newcommand\Lemma{\begin{lemma}}
\newcommand\EndLemma{\end{lemma}}
\newcommand\Proposition{\begin{proposition}}
\newcommand\EndProposition{\end{proposition}}
\newcommand\Corollary{\begin{corollary}}
\newcommand\EndCorollary{\end{corollary}}
\newcommand\Claim{\begin{claim}}
\newcommand\EndClaim{\end{claim}}
\newcommand\Proof{\begin{proof}}
\newcommand\EndProof{\end{proof}}
\newcommand\Equation{\begin{equation}}
\newcommand\EndEquation{\end{equation}}

\newcommand\Bullets{\begin{itemize}}
\newcommand\EndBullets{\end{itemize}}

\newcommand{\redden}[1]{\textcolor{red}{#1}}

\newcommand\OneSide{Lemma 3.1}
\newcommand\TrunchTeth{\cite[Definition 3.2]{DeB_TrigTruncTriTet}}
\newcommand\TetraTrans{\cite[Lemma 3.4]{DeB_TrigTruncTriTet}}
\newcommand\ShrinkingTrans{\cite[Proposition 3.5]{DeB_TrigTruncTriTet}}

\newcommand\simple{simple}
\newcommand\discup{\mathbin{\rotatebox[origin=c]{90}{$\vDash$}}}
\newcommand\ie{i.e.\ }

\newcommand\inter{\mathop{\rm int}}

\newcommand{\cut}{\,\backslash\backslash\,}
\newcommand\Hg{{\rm Hg}}
\newcommand\redtopback{\redden{\ref{top back}}}

\newcommand\redcollar{open collar}
\newcommand\redD{D}
\newcommand\redI{I}
\newcommand\redII{II}

\newcommand\kish{\mathop{\rm kish}}
\newcommand\area{\mathop{\rm area}}
\newcommand\chibar{\bar\chi}
\newcommand\HH{{\mathbb H}}
\newcommand\ZZ{{\mathbb Z}}

\newcommand\tDN{\widetilde{DN}}

\newcommand\calp{{\mathcal P}}
\newcommand\calb{{\mathcal B}}
\newcommand\calw{{\mathcal W}}

\newcommand\cals{\mathcal{S}}

\newcommand\co{\colon\thinspace}

\newcommand\be{\mathbf{e}}
\newcommand\bx{\mathbf{x}}
\newcommand\bu{\mathbf{u}}
\newcommand\FF{{\mathbb F}}
\newcommand\vol{\mathop{\rm vol}}
\newcommand\voct{V_{\rm oct}}

\newtheorem{definitionsconventionsremarks}[para]{Definitions,
  Conventions,  and Remarks}
\newcommand\Fr{\mathop{\rm Fr}}
\newcommand\rank{\mathop{{\rm rank}}}
\newcommand\QQ{{\mathbb Q }}
\newcommand\tM{\widetilde M}
\newcommand\tN{\widetilde N}
\newcommand\ti{\widetilde i}
\newcommand\tS{\widetilde S}
\newcommand\tf{\widetilde f}
\newcommand\tF{\widetilde F}
\newcommand\calB{{\mathcal B}}
\newcommand\geodvol{\mathop{\rm geodvol}}
\newcommand\tr{{\widetilde r}}
\newcommand\diffeo{{diffeo}}

\begin{document}

\title{
Volume and topology of bounded and closed hyperbolic $3$-manifolds, II
}

\author{Jason DeBlois}
\address{Department of Mathematics\\
University of Pittsburgh\\
301 Thackeray Hall\\
Pittsburgh, PA 15260}
\email{jdeblois@pitt.edu}

\author{Peter B.~Shalen}
\address{Department of Mathematics, Statistics, and Computer Science
(M/C 249)\\
University of Illinois at Chicago\\
851 S. Morgan St.\\
Chicago, IL 60607-7045}
\email{petershalen@gmail.com}

\begin{abstract}
Let $N$ be a compact, orientable hyperbolic 3-manifold whose boundary is a connected totally geodesic surface of genus $2$. If $N$ has Heegaard genus at least $5$, then its volume is greater than $2\voct$, where $\voct=3.66\ldots$ denotes the volume of a regular ideal hyperbolic octahedron in $\HH^3$. This  improves the lower bound given in our earlier paper ``Volume and topology of bounded and closed hyperbolic $3$-manifolds.'' One ingredient in the improved bound is that in a crucial case, instead of using a single ``muffin'' in $N$ in the sense of Kojima and Miyamoto, we use two disjoint muffins. By combining the result about manifolds with geodesic boundary with the $\log(2k-1)$ theorem and results due to Agol-Culler-Shalen and Shalen-Wagreich, we show that if
$M$ is a closed, orientable hyperbolic $3$-manifold with $\vol M\le\voct/2$, then $\dim H_1(M;\FF_2)\le4$. We also provide new lower bounds for the volumes of closed hyperbolic $3$-manifolds whose cohomology ring over $\FF_2$ satisfies certain restrictions; these improve results that were proved in ``Volume and topology\ldots.''
\end{abstract}

\maketitle

\section{ Introduction
}

It follows from the Mostow rigidity theorem that the volume of a finite-volume hyperbolic $3$-manifold $M$ is a topological invariant of $M$. We may regard $\vol M$ as a measure of the topological complexity of $M$, and it is natural to try to relate this measure of topological complexity with more classical ones, such as the ranks of homology groups of $M$ with prescribed coefficients.

It has long been known that for any prime $p$, the dimension of
$H_1(M;\FF_p)$ (where $\FF_p$ denotes the field of order $p$) is
linearly bounded in terms of $\vol M$. According to \cite[Theorem
5.4]{ratioI}, 
we have 
$\dim H_1(M;\FF_p)<168.602\cdot\vol M$  for every prime $p$. It is
expected that in the forthcoming paper \cite{ratioIII} this result
will be improved in the case $p=2$, by replacing the coefficient
$168.602$ by one that is a bit less than $158$.

For small values of $\vol M$, 
these results were improved by a couple of orders of magnitude in
\cite{acs-surgery}, \cite{CS_vol},  
\cite{CS_vol3.44}, and \cite{kfree-volume}:

\begin{itemize}
\item  Theorem 1.1 of \cite{acs-surgery} asserts that if $\vol M\le1.22$ then $\dim H_1(M;\FF_p)\le 2$ for  $p\ne 2,7$, while $\dim H_1(M;\FF_p)\le 3$ if $ p$ is $2$ or $7$;
\item Theorem 1.2 of \cite{CS_vol} asserts that if $\vol M\le3.08$ then $\dim H_1(M;\FF_2)\le 5$. 
\item Theorem 14.5 of  \cite{kfree-volume} 
(which improves on Theorem 1.7 of \cite
{CS_vol3.44})
asserts that if $\vol M\le3.69$ then $\dim H_1(M;\FF_2)\le 7$.
\end{itemize}
Among these results, only Theorem 1.1 of \cite{acs-surgery} is known
to be sharp, and only for $p=5$: we have $\dim H_1(M;\FF_5)= 2$ when
$M$ is the Weeks manifold. Among the 
closed, orientable hyperbolic
$3$-manifolds $M$ of volume at most $3.69$
known from a census compiled by Hodgson--Weeks 
(see \cite{HodWee}) 
and shipped with SnapPy \cite{SnapPy}, the largest value of
$\dim H_1(M;\FF_2)$ is $3$; it is achieved by the manifold denoted by
${\rm m135}(-1,3)$ in the census,

In this paper, we will denote by $\voct$ the volume of a regular ideal hyperbolic
octahedron. 
It is known that $\voct = 8 \El(\pi/4)=3.6638\ldots$, where $\El$ is the Lobachevsky function; 
see the Example at the end of Section 7.2 of Thurston's notes \cite{Th1}. By the Fourier series for the Lobachevsky 
function given in \cite[Lemma 7.1.2]{Th1}, $\voct$ 
is equal to 
four times Catalan's constant.

We will prove:

\newcommand\MainIngredThrm{Let $M$ be a closed, orientable hyperbolic $3$-manifold with
\linebreak $\vol M\le\voct/2$. Then $\dim H_1(M;\FF_2)\le4$.}
\newtheorem*{main ingredient theorem}{Theorem \ref{main ingredient}}
\begin{main ingredient theorem}
Let $M$ be a closed, orientable hyperbolic $3$-manifold with $\vol M\le\voct/2$. Then $\dim H_1(M;\FF_2)\le4$.
\end{main ingredient theorem}

This result improves on Theorem 1.2 of \cite{CS_vol} in the
range in which it applies, but it is probably not sharp. Among known closed, orientable hyperbolic
$3$-manifolds $M$ of volume at most $\voct/2$, the largest value of
$\dim H_1(M;\FF_2)$ is $2$; it is achieved by the manifold denoted by
${\rm m}009(5,1)$ in the census
cited above,
which is arithmetic and has volume  exactly $\voct/2$.

Contrapositively, the results quoted above from \cite{acs-surgery}, \cite{CS_vol}, and \cite
{CS_vol3.44} may be interpreted as saying that lower bounds on the
dimension of $H_1(M;\FF_2)$, where $M$ is a closed, orientable, hyperbolic
manifold, give lower bounds on $\vol M$.
In \cite{DeSh} it was shown that these lower bounds become stronger if
one places restrictions on the cup product pairing from
$H^1(M;\FF_2)\otimes H^1(M;\FF_2)$ to $H^2(M;\FF_2)$. Specifically,
Theorem 1.2 of \cite{DeSh} asserts that if $\dim H^1(M;\FF_2)$ is at least $5$, and the
dimension of the cup product pairing is at most $1$, then $\vol
M>3.44$. In this paper we will prove the following stronger result:

\newcommand\CupStuffThrm{Let $M$ be a closed, orientable hyperbolic $3$-manifold. Set $r=\dim
H_1(M;\FF_2)$, and let $t$ denote the dimension of 
 the image of the
cup product pairing
$H^1(M;\FF_2)
\otimes 
H^1(M;\FF_2)\to H^2(M;\FF_2)$.
Then:
\begin{enumerate}
\item if $r\ge5$ and $t\le1$, we have $\vol(M)>3.57$; and
\item if $r\ge6$ and $t\le3$, or $r\ge7$ and $t\le5$, we have $\vol(M)>\voct$.
\end{enumerate}
}

\newtheorem*{cup stuff theorem}{Theorem \ref{cup stuff}}
\begin{cup stuff theorem}\CupStuffThrm\end{cup stuff theorem}

The proofs of Theorems \ref{main ingredient} and \ref{cup stuff} use
the following result, Proposition \ref{who you}. Recall that a group $\Pi$ is said to be {\it $k$-free} for a given positive integer $k$ if every subgroup of $\Pi$ whose rank is at most $k$ is free.

\newcommand\WhoYouProp{
Let $M$ be a closed, orientable, hyperbolic $3$-manifold, let
 $k\ge3$
be an integer, and suppose that $\dim
H_1(M;\FF_2)\ge \max(3k-4,6)$. Then either $\pi_1(M)$ is $k$-free,
or $\vol
M>2\voct$.
}

\newtheorem*{who you prop}{Proposition \ref{who you}}

\begin{who you prop}\WhoYouProp\end{who you prop}

The proofs of Theorems \ref{main ingredient} and \ref{cup stuff} and
Proposition \ref{who you} involve combining a rich variety of
topological and geometric techniques.
One key ingredient 
in the proof of Proposition \ref{who you} is the following result, 
which pertains to manifolds with boundary and features a different
measure of topological complexity, the \textit{Heegaard genus}, 
which in this paper will be  denoted $\mathrm{Hg}(N)$. For a compact, connected, orientable $3$-manifold $N$, with or without boundary, $\Hg(N)$ is 
the smallest genus of a closed surface in $N$ that divides it into two compression bodies.

\newcommand\ThrmBoundingMain{Let $N$ be a compact, orientable hyperbolic $3$-manifold with $\partial N$ connected, totally geodesic, and of genus $2$. If $\mathrm{Hg}(N)\ge 5$ then $\mathrm{vol}(N) > 2\voct$.
}
\newtheorem*{bounding main theorem}{Theorem \ref{bounding main}}
\begin{bounding main theorem}\ThrmBoundingMain\end{bounding main theorem}

This directly strengthens Theorem 1.1 of \cite{DeSh}, which gives a lower bound of $6.89$ on the volume of $N$ with the same topological hypothesis. Like the results above for closed manifolds, it is still likely not sharp: in a census compiled by Frigerio-Martelli-Petronio \cite{FMP}, all examples with volume at most $2\voct$ have Heegaard genus $3$. This census collected all manifolds with totally geodesic boundary that decompose into at most four truncated tetrahedra. After Kojima--Miyamoto's minimal-volume examples (of volume $6.45...$), the census of \cite{FMP} contains six manifolds with volume $7.10...$, and the next-smallest have volume $7.33... ~>2\voct$.

We now describe  the structure of the body of the paper, 
which 
breaks naturally into three parts. 
The first of these, consisting of Sections \ref{bound found} through \ref{trimonic}, 
primarily addresses hyperbolic $3$-manifolds with totally geodesic boundary. 
The results of these sections build on the ``geometric part'' of the 
proof scheme 
of \cite[Th.~1.1]{DeSh}, which itself built on vocabulary and results established by Kojima--Miyamoto in their work \cite{KM} that identified the minimum-volume compact hyperbolic $3$-manifolds with totally geodesic boundary. Key tools of analysis include \textit{return paths} and \textit{$(i,j,k)$-hexagons}, which are respectively associated to pairs and triples of boundary components of the universal cover $\widetilde{N}$ (definitions in Subsection \ref{orthoarc}).

Section \ref{bound found} reviews 
the methods introduced in \cite{KM} and further developed in \cite{DeSh} for describing how the lengths of return paths of 
a manifold $N$ satisfying the hypotheses of Theorem \ref{bounding main} 
 are controlled by 
the
geometry of $\partial N$, moderated by $(i,j,k)$-hexagons. 
In places, we incrementally improve these methods.
In particular, Proposition \ref{KM ell_2} simplifies the lower bound on the second-shortest return path length $\ell_2$ as a function of $\ell_1$ given in \cite[Lemma~2.9]{DeSh} (which itself built on \cite[\S 4]{KM}). And Proposition \ref{DeSh ell_2} gives a stronger lower bound than \cite[Prop.~3.9]{DeSh} on $\cosh\ell_1$ in the absence of a $(1,1,1)$-hexagon: $1.23$ here versus $1.215$ there. (The minimum possible 
value for $\cosh\ell_1$ is $1.183..$, 
proved in \cite{KM}.)

Section \ref{bound muffin} 
introduces a
significant new tool 
for bounding volume below: 
a second muffin. Here, ``muffin'' is Kojima-Miyamoto's term for a hyperbolic solid of rotation generated by a certain pentagon with four right angles. They show in \cite[Lemma 3.2]{KM} that such a muffin embeds in $N$ with its rotation center along the shortest return path $\lambda_1$. We called this $\mathrm{Muf}_{\ell_1}$ in \cite{DeSh} and continue to do so here. This is the ``first'' muffin in $N$.

In Subsection \ref{emb muf} we define a broader class of muffins 
and lay out criteria for embedding a second muffin in $N$, centered 
on the second-shortest return path, without overlapping $\mathrm{Muf}_{\ell_1}$.
Each muffin that we use intersects $\partial N$ in 
the union of two  disjoint 
disks, 
its ``caps''. In Subsection \ref{emb col} we give a sufficient condition to 
ensure that a collar of the region of $\partial N$ outside the muffin caps 
is embedded in $N$ and does not overlap the muffins.

Section \ref{bound vol} uses embedded muffins and collars to give lower 
bounds on volume.
The fundamental volume bound given in \cite{KM} and used in \cite{DeSh}, recorded here in (\ref{KM vol ineq}), is a function of $x = \cosh\ell_1$ that records the sum of the volumes of $\mathrm{Muf}_{\ell_1}$ and of an embedded collar of 
the complement in $\partial N$ of its caps. 
For manifolds satisfying certain conditions on $\ell_1$ and $\ell_2$, 
we bring a second muffin into play here, allowing us to recover additional volume.
We prove:

\newcommand\CoreOLarry{Let $N$ be an orientable hyperbolic $3$-manifold with $\partial N$ compact, connected, totally geodesic, and of genus $2$. If the universal cover $\widetilde{N}$ of $N$ contains no $(1,1,1)$-hexagon then $\mathrm{vol}(N) \ge 7.4$.}
\newtheorem*{larry core}{Corollary \ref{larry}}
\begin{larry core}\CoreOLarry\end{larry core}

This is the main result of Section \ref{bound vol}. 
It strengthens the 
lower bound of $6.89$ for $\mathrm{vol}(N)$ which,
under the same hypotheses, follows from Propositions 3.7 and 3.9 
of \cite{DeSh}.

In Section \ref{trimonic} we review and slightly upgrade 
certain 
 results of \cite{DeSh} that address the 
 other case, in which there is a $(1,1,1)$-hexagon  
in $\widetilde{N}$. 
Our upgrades here remove or relax restrictions on $\ell_1$ in the hypotheses 
of their antecedents. Notably,  
Lemma \ref{short cut vs 111 hex} shows unconditionally that $(1,1,1)$-hexagons interact well with shortest return paths, removing a hypothesis of \cite[Lemma 6.6]{DeSh}. 
The thrust of this section follows that of \cite[\S 6]{DeSh}, using the $(1,1,1)$-hexagon to construct 
a  submanifold $X$ of $N$ which is a``non-degenerate trimonic''
submanifold in the sense defined in \cite[\S 5]{DeSh}). 
Results from the ``topological part'' of the proof of \cite[Th.~1.1]{DeSh}, 
which are laid out in Sections 4 and 5 of that paper, will then be applied to $X$ 
without requiring further adaptation.

Sections \ref{bop tack} and \ref{whataterribletitle} constitute
the second of the present paper's three parts. 
Section \ref{bop tack} first introduces topological notation used in the rest of 
the paper, in Subsection \ref{top back}.
In Subsection \ref{geom back} we 
 refine methods of Agol-Storm-Thurston \cite{ASTD} 
for bounding the volume of a hyperbolic Haken $3$-manifold $M$ below 
in terms of the topology of the manifold obtained by cutting $M$ along an 
incompressible surface. The subsection's main result, 
Theorem \ref{from ast} strengthens the conclusion of Theorem 9.1 of \cite{ASTD} 
for compact such $M$ 
by replacing a non-strict inequality with a strict one; it also allows $M$ to 
have connected, totally geodesic boundary.

In Section \ref{whataterribletitle} we first prove  
Theorem \ref{7.4 upgrade}, which improves Theorem 7.4 of \cite{DeSh}. 
The improvement comes from applying Theorem \ref{from ast} 
(replacing the previous result's appeal to \cite[Th.~9.1]{ASTD}) to the frontier in $N$ of the 
trimonic submanifold $X$ constructed in Section \ref{trimonic}, 
in a certain case of the proof of this result. 
The proof of Theorem \ref{bounding main} completes the section, with 
complementary cases supplied there by Theorem \ref{7.4 upgrade} and  Corollary \ref{larry}.

In the paper's third part, 
beginning with Section \ref{closed background}, we shift 
our focus to closed manifolds. 
Section \ref{closed background} gives background necessary to prove our results in this setting. 
The proof of Proposition  \ref{who you} begins by using topological results due
to Culler and Shalen
\cite{CS_vol}, about desingularization of $\pi_1$-injective singular surfaces in $3$-manifolds, to show that if $M$ satisfies the homological hypothesis of the
proposition and $\pi_1(M)$ is not $k$-free, then $M$ contains a closed
incompressible surface $S_0$
of some genus $g$ with $2\le g\le k-1$. The homological hypothesis
implies that the Heegaard genus of $M$ is 
strictly greater than $2g+1$.

If $S\subset M$ is any
closed
incompressible surface, and  
$M'=M\setminus\setminus S$ denotes the manifold   obtained
from $M$ by splitting it along $S$, we denote by 
$\kish(M')$ (sometimes called the
``kishkes'' or ``guts'' of $M'$) the union of those components of
$\overline{M'-\Sigma}$ that have strictly negative Euler
characteristic, where $\Sigma$ denotes the union of the characteristic submanifolds of the
components of $M'$.

The existence of an incompressible surface $S_0$ of genus $g$,
together with the strict lower bound $2g+1$ for the Heegaard genus of
$M$, is used---via a result proved in \cite{kfree-volume}
using topological ideas developed by Culler, DeBlois and Shalen in \cite{CDS}---to produce a closed
incompressible surface $S\subset M$ such that either (1) the
Euler characteristic $\chi(\kish(M\setminus\setminus S))$ is at most
$-2$, or (2) $S$ separates $M$, and
 $M\setminus\setminus S$ has a component which is acylindrical,
 i.e. contains no essential annulus. If (1) holds, the geometric 
methods developed 
 by Agol, Storm and Thurston in
 \cite{ASTD} 
(via Theorem \ref{from
  ast} of this paper) 
give a 
strict 
 lower bound of $2\voct$ for $\vol M$. 

If (2) holds but (1) does not, and if we fix an acylindrical component
$A$ of $M\setminus\setminus S$, then $A$ is \diffeo morphic to a hyperbolic $3$-manifold $N$ with totally
geodesic boundary, and the methods of \cite{ASTD} show that $\vol N$
is a lower bound for $\vol M$. 
Furthermore, in this case a
Mayer-Vietoris calculation shows that $\dim H_1(N,\FF_2)$, and hence
the Heegaard genus of $N$, is at least $5$. If $S$ has genus $2$,
Theorem \ref{bounding main} now gives a 
strict 
lower bound of $2\voct$ for $\vol
N$. If $S$ has genus greater than $2$, the geometric results
established by Miyamoto in \cite{Miy}
give a stronger lower bound.

To prove Theorem  \ref{main ingredient} one must show that if
that $\dim H_1(M;\FF_2) \ge5$ then $\vol M>\voct/2$. If $\pi_1(M)$ is
$3$-free, then results proved by Anderson, Canary, Culler and Shalen in \cite{ACCS}, and improved in
\cite{ACS} by using the celebrated tameness theorem proved by Agol in
of \cite{agol-tameness} and by Calegari-Gabai in
\cite{cg}, give a lower bound for $\vol M$ of $3.08$, which is
considerably bigger than $\voct/2$. If $\dim H_1(M;\FF_2) \ge5$ but
$\pi_1(M)$ is not
$3$-free, a novel but simple application of one of the topological
results established by Shalen and Wagreich in
\cite{SW} provides a $(\ZZ/2\ZZ\times\ZZ/2\ZZ)$-covering space $\tM$ of
$M$ such that $\pi_1(\tM)$ is not $3$-free, and $\dim H_1(\tM;\FF_2) \ge7$. One
can then use  Proposition  \ref{who you} to show that
$\vol\tM>2\voct,$ which implies  $\vol M>\voct/2$.

The proof of  Theorem  \ref{cup stuff} follows the same basic outline
as the proof of Theorem 1.2 of \cite{DeSh}. As in the latter proof, we
distinguish the cases in which $\pi_1(M)$ is or is not $4$-free. 
If
$\pi_1(M)$ is  $4$-free, one of the main results proved by Guzman and
Shalen in \cite{kfree-volume} gives a lower bound of  $3.57$ for $\vol
M$, which is a surprising improvement over the lower bound of $3.44$
established in \cite{CS_vol3.44} and quoted in \cite{DeSh}, and is
enough to prove Assertion (1) of Theorem \ref{cup stuff} in this
case. If we combine the assumption of $4$-freeness with a  lower bound
of $6$ for $\dim H_1(M;\FF_2)$, then arguments given in
\cite{kfree-volume}, based on Agol and Dunfield's results on the
change of volume under Dehn drilling \cite{ASTD} and results due to
Culler and Shalen about volumes and homology of one-cusped manifolds
\cite{CS_onecusp}, allow one to obtain a lower bound of $3.69$ for
$\vol M$, which establishes Assertion (2) of Theorem \ref{cup stuff} in this
case.

If $M$ satisfies the hypothesis of Assertion (1) of (2) of Theorem \ref{cup stuff} and
$\pi_1(M)$ is not $4$-free, then one uses the homological hypotheses to find a two-sheeted
covering $\tM$ of $M$ such that $\pi_1(\tM)$ is not $4$-free, and
$\dim H_1(M;\FF_2)\ge8$. Proposition \ref{who you} then provides a 
strict 
lower bound of $2\voct$ for $\vol\tM$, and hence a 
strict 
lower bound of $\voct$ for $\vol M$. (This is similar to the argument
used to prove  \cite[Theorem
1.2]{DeSh} in the non-free case, but we obtain a stronger lower bound
in this context thanks to Theorem \ref{bounding main}. 
In our proof of  Theorem \ref{cup stuff}, in addition to improving the
estimates given by Theorem
1.2 of \cite{DeSh}, we have taken the opportunity to provide
more detail than was given in the proof of the latter result, and to correct a citation 
that appeared in that proof.)

We are grateful to Nathan Dunfield for explaining the example ${\rm
  m}009(5,1)$ that was referred to above, and to Joel Hass for explaining 
 material related to Proposition \ref{FHS plus epsilon} 
to us.

\section{Existing foundations}\label{bound found}

This section reviews an approach developed by Kojima--Miyamoto \cite{KM}, and further exploited 
in our 
earlier 
paper 
\cite{DeSh}, to controlling the geometry of a hyperbolic $3$-manifold $N$ with totally geodesic boundary via the geometry of $\partial N$.  We focus on the \textit{orthospectrum} of $N$, the sequence of lengths of properly immersed arcs in $N$ that meet $\partial N$ perpendicularly.  

Section \ref{orthoarc} relates the orthospectrum of $N$ to a certain spectrum of arc lengths on $\partial N$ using hyperbolic trigonometry.  In Section \ref{KM bounds} we describe Kojima--Miyamoto's packing arguments for bounding $\ell_2$, the second-smallest ortholength, below in terms of $\ell_1$.  We conclude in Section \ref{DeSh bounds} with an observation from \cite{DeSh} that bounds $\ell_2$ above in terms of $\ell_1$ in the absence of ``$(1,1,1)$-hexagons'' (see below).

In fact we make incremental improvements below to the existing bounds we describe.  These bounds are for the most part still not strong enough to be directly useful to our volume estimates, but we use them to delimit a search space for a procedure to find sharper ones.

\subsection{Ortholengths vs arclengths}\label{orthoarc}  If $N$ is a complete hyperbolic $3$-manifold with totally geodesic boundary, its universal cover $\widetilde{N}$ may be identified with a convex subset of $\mathbb{H}^3$ bounded by a collection of geodesic hyperplanes.  We will do so, and we will also continue to use the following terminology which originated in \cite{Ko} and \cite{KM} and was used in \cite{DeSh}.

\begin{definition}  Let $N$ be a hyperbolic $3$-manifold with compact totally geodesic boundary, and let $\widetilde{N} \subset \mathbb{H}^3$ be its universal cover.  A \textit{short cut} in $\widetilde{N}$ is a geodesic arc joining the closest points of two distinct components of $\partial\widetilde{N}$.  A \textit{return path}, or \textit{orthogeodesic}, in $N$ is the projection of a short cut under the universal covering map.  \end{definition}

Each return path is a homotopically non-trivial geodesic arc properly immersed in $N$, perpendicular to $\partial N$ at each of its endpoints.  Corollary 3.3 of \cite{Ko} asserts that for a fixed $K \in \mathbb{R}$ and finite-volume hyperbolic manifold $N$ with totally geodesic boundary, there are only finitely many return paths in $N$ with length less than $K$.  This makes possible the following:

\begin{definition}\label{orthospectrum}  Let $N$ be a finite-volume hyperbolic $3$-manifold with compact totally geodesic boundary.  Upon enumerating the collection of return paths as $\{\lambda_1,\lambda_2,\hdots\}$, where for each $i\in\mathbb{N}$ the length of $\lambda_{i+1}$ is at least the length of $\lambda_i$, let $\ell_i$ denote the length of $\lambda_i$.  The \textit{orthospectrum} of $N$ is the sequence $(\ell_1,\ell_2,\hdots)$.  Its elements are \textit{ortholengths}.\end{definition}

For $N$ as above, hyperbolic trigonometry relates the orthospectrum to arc lengths on $\partial N$ by means of a class of totally geodesic hexagons in $\widetilde{N}$ that have short cuts as some edges.  Below we reproduce two lemmas from \cite{DeSh} that describe the hexagons in question.

\begin{lemma}[\cite{DeSh}, Lemma 2.3]\label{mutual perp}  Suppose that $\Pi_1$, $\Pi_2$, and $\Pi_3$ are mutually disjoint geodesic planes in $\mathbb{H}^3$.  For each two-element subset $\{i,j\}$ of $\{1,2,3\}$, let $\lambda_{ij}$ denote the common perpendicular to $\Pi_i$ and $\Pi_j$.  Then $\lambda_{12}$, $\lambda_{13}$, and $\lambda_{23}$ lie in a common plane $\Pi$.  \end{lemma}

\begin{lemma}[\cite{DeSh}, Lemma 2.4]\label{rt ang hex}  Let $N$ be a finite-volume hyperbolic $3$-manifold with compact totally geodesic boundary, and suppose $\Pi_1$, $\Pi_2$, and $\Pi_3$ are distinct components of $\partial\widetilde{N}$.  Let $\Pi$ be the plane, produced by Lemma \ref{mutual perp}, which contains the short cuts $\lambda_{12}$, $\lambda_{13}$, and $\lambda_{23}$.  Let $C$ be the right-angled hexagon in $\Pi$ with edges the $\lambda_{ij}$ for $1\leq i<j\leq 3$, together with the set of geodesic arcs in the $\Pi_i$ joining their endpoints.  Then $C\subset\widetilde{N}$, and $C\cap\partial\widetilde{N} = \bigcup_i(C\cap\Pi_i)$.  \end{lemma}

In \cite{DeSh} the result above is stated only for compact hyperbolic $3$-manifolds with totally geodesic boundary, but its proof carries over to the current context without revision.

\begin{definition}\label{internal external} For a finite-volume hyperbolic $3$-manifold $N$ with compact totally geodesic boundary and three components $\Pi_1$, $\Pi_2$, and $\Pi_3$ of $\widetilde{N}$, let $C$ be the right-angled hexagon supplied by Lemma \ref{rt ang hex}.  We call the edges of $C$ which are short cuts \textit{internal}, and the edges in $\partial\widetilde{N}$ \textit{external}.  If the internal edges lift $\lambda_i$, $\lambda_j$, and $\lambda_k$, we call $C$ an \textit{$(i,j,k)$ hexagon}.  \end{definition}

We will say that the \textit{feet} of a return path $\lambda$ are the points $\lambda \cap \partial N$, and similarly for the feet of a short cut.  The orthospectrum of $N$ is related to the set of lengths of arcs in $\partial N$ joining feet of return paths.

\begin{definition}\label{arcspectrum}  Let $N$ be a finite-volume hyperbolic $3$-manifold with compact totally geodesic boundary.  For $i,j\in\mathbb{N}$ let $d_{ij}$ be the length of the shortest non-constant geodesic arc joining a foot of $\lambda_i$ to one of $\lambda_j$, or $\infty$ if no such arc exists.  For any $k\in\mathbb{N}$ let $d_{ij}^{(k)}$ be the length of the $k$th-shortest such arc, or $\infty$ as appropriate.  \end{definition}

\begin{lemma}\label{boundary arc lengths}For $i$, $j$, and $k$ in $\mathbb{N}$, let $X_{ij}^k$ be determined by \begin{align}\label{Xijk}
  \cosh X_{ij}^k = \frac{\cosh\ell_i\cosh\ell_j+\cosh\ell_k}{\sinh\ell_i\sinh\ell_j}  \end{align}
For a finite-volume hyperbolic $3$-manifold $N$ with compact totally geodesic boundary and any fixed $i, j \in \mathbb{N}$, if $k_1 \leq k_2 \leq k_3 \hdots$ is the set of $k\in\mathbb{N}$ such that there exists an $(i,j,k)$ hexagon in $\widetilde{N}$ then $d_{ij} = X_{ij}^{k_1}$, and $d_{ij}^{(n)} = X_{ij}^{k_n}$ for $n>1$.  In particular, $d_{ij}^k \geq X_{ij}^k$ for all $i,j,k\in\mathbb{N}$.\end{lemma}

\begin{proof} Any given geodesic arc $\gamma$ on $\partial N$ that joins a foot of $\lambda_i$ to one of $\lambda_j$ lifts to a geodesic arc $\tilde{\gamma}$ on a component $\Pi$ of $\partial\widetilde{N}$ joining the foot of a lift $\tilde{\lambda}_i$ of $\lambda_i$ to that of a lift $\tilde{\lambda}_j$ of $\lambda_j$.  The feet of $\tilde{\lambda}_i$ and $\tilde{\lambda}_j$ opposite their intersections with $\tilde{\gamma}$ lie in components of $\partial\widetilde{N}$ joined by a short cut $\tilde{\lambda}_k$ for some $k$.  Lemma \ref{rt ang hex} then implies that $\tilde{\gamma}$ is the external edge of an $(i,j,k)$ hexagon opposite $\tilde{\lambda}_k$.  The ``right-angled hexagon rule'' \cite[Theorem 3.5.13]{Ratcliffe} implies that the length of $\tilde\gamma$, hence of $\gamma$, is $X_{ij}^k$.

It is a quick consequence of the definition that for fixed $i$ and $j$, if $k < k'$ then $X_{ij}^k < X_{ij}^{k'}$.  The lemma follows.\end{proof}

We close this subsection with some basic observations on the monotonicity of the $X_{ij}^k$ for $i,j,k\in\{1,2\}$.

\begin{lemma}\label{ordering external lengths} The function $X_{11}^1$ defined Lemma \ref{boundary arc lengths} is decreasing in $\ell_1$.  $X_{11}^2$ decreases in $\ell_1$ and increases in $\ell_2$, and $X_{12}^1$ decreases in both $\ell_1$ and $\ell_2$.  Moreover, $X_{11}^1 \geq X_{12}^1 \geq X_{22}^1$.\end{lemma}

\begin{proof}  Each of these follows from (\ref{Xijk}) with a little manipulation.  We have for instance:\begin{align}\label{X111}
  \cosh X_{11}^1 = \frac{\cosh^2\ell_1 +\cosh \ell_1}{\sinh^2\ell_1} = \frac{\cosh\ell_1}{\cosh\ell_1-1} = 1 + \frac{1}{\cosh\ell_1-1}, \end{align}
and the first assertion is clear.  Along similar lines:\begin{align}\label{X112}
  \cosh X_{11}^2 = \frac{\cosh^2\ell_1+\cosh\ell_2}{\sinh^2\ell_1} = 1 + \frac{1+\cosh\ell_2}{\sinh^2\ell_1}, \end{align}
and the second assertion holds.  Finally:\begin{align}\label{X121}
  \cosh X_{12}^1 = \frac{\cosh\ell_1}{\sinh\ell_1}\sqrt{\frac{\cosh\ell_2+1}{\cosh\ell_2-1}} = \coth\ell_1\sqrt{1+\frac{2}{\cosh\ell_2-1}}\end{align}
That $X_{12}^1$ decreases in both $\ell_1$ and $\ell_2$ now follows from the fact that $\coth x$ decreases in $x$.  Moreover, both $\cosh X_{11}^1$ and $\cosh X_{12}^1$ are of the form
$$ \frac{\cosh x \cosh \ell_1 + \cosh \ell_1}{\sinh x\sinh \ell_1} = \coth x \coth\ell_1 + \frac{\cosh \ell_1}{\sinh x\sinh\ell_1},  $$
where one substitutes $\ell_1$ for $x$ to produce $\cosh X_{11}^1$ and $\ell_2$ for $x$ to produce $\cosh X_{12}^1$.  Since $x \mapsto \coth x$ is decreasing and $x \mapsto \sinh x$ is increasing, it follows that $\cosh X_{11}^1 \geq \cosh X_{12}^1$.  This implies the left-hand inequality above; the right-hand inequality follows similarly.\end{proof}

\subsection{Kojima--Miyamoto's lower bound}\label{KM bounds}  Here we will reproduce an argument originally from \cite{KM} which was slightly improved in \cite{DeSh}, and improve it slightly further.  It gives a lower bound on $\ell_2$ as a function of $\ell_1$ among finite-volume hyperbolic $3$-manifolds with compact totally geodesic boundary.  This bound is rarely near sharp, but it is easily computable. 
The functions $R$, $R'$, and $E$ of $\ell_1$ below are as in \cite{KM} and \cite{DeSh}; $R''$ and $M$ match \cite{DeSh}.

The basic idea here is that for a hyperbolic $3$-manifold $N$ with compact totally geodesic boundary, the topology of $\partial N$ determines its area by the Gauss--Bonnet theorem, and this bounds the areas of disks in a packing of $\partial N$.  The radii of such a packing are determined by the orthospectrum of $N$.

\begin{lemma}\label{disk radii}  For $X_{ij}^k$ as defined in (\ref{Xijk}), let $R = X_{11}^1/2$, satisfying\begin{align}\label{R}
\cosh R = \sqrt{\frac{2\cosh \ell_1 - 1}{2\cosh \ell_1 - 2}} = \sqrt{1 + \frac{1}{2\cosh\ell_1-2}}, \end{align}
and let $S = X_{12}^1 - R$.  For a finite-volume hyperbolic $3$-manifold $N$ with compact, connected totally geodesic boundary such that $S>0$ there are four disks embedded in $\partial N$ without overlapping: two of radius $R$, centered at feet of $\lambda_1$, and two of radius $S$, centered at feet of $\lambda_2$.\end{lemma}

\begin{proof} By Definition \ref{arcspectrum}, disks $U$ and $U'$ of radius $r$, centered at the feet of $\lambda_1$, are embedded in $\partial N$ without overlapping each other if and only $r\leq d_{11}/2$.  Similarly, disks $V$ and $V'$ of radius $s$, centered at the feet of $\lambda_2$, are embedded in $\partial N$ without overlapping if and only if $s\leq d_{22}/2$.  Finally, if $r+s\leq d_{12}$ then $U\cup U'$ does not overlap $V\cup V'$.

Lemma \ref{boundary arc lengths} and the definitions above imply that $R \leq d_{11}/2$ and $R+S\leq d_{12}$.  We will show below that $S\leq X_{22}^1/2 \leq d_{22}/2$, hence by the paragraph above that the lemma holds.  Applying (\ref{Xijk}) and the ``angle addition formula'' for hyperbolic sine yields:\begin{align}
  & \sinh S = \frac{\sqrt{(2\cosh^2\ell_1+\cosh\ell_2-1)(2\cosh\ell_1-1)}-\cosh\ell_1\sqrt{\cosh\ell_2+1}}{(\cosh\ell_1-1)\sqrt{2(\cosh\ell_1+1)(\cosh\ell_2-1)}}  \label{sinch ess}\\
  & \sinh \left(X_{22}^1/2\right) = \frac{\sqrt{\cosh\ell_1+1}}{\sqrt{2}\sinh\ell_2} \nonumber \end{align}
Subtracting $\sinh S$ from $\sinh \left(X_{22}^1/2\right)$ and using the common denominator $\sqrt{2}\sinh\ell_2(\cosh\ell_1-1)\sqrt{\cosh\ell_1+1}$, we find that the numerator of $\sinh \left(X_{22}^1/2\right) - \sinh S$ is as below: \begin{align*} 
    \cosh^2\ell_1-1 + \cosh\ell_1(\cosh\ell_2+1) - 2\sqrt{\left(\cosh^2\ell_1+\frac{\cosh\ell_2-1}{2}\right)\left(\cosh\ell_1-\frac{1}{2}\right)(\cosh\ell_2+1)}  \\
   =\ \ \left(\sqrt{\cosh^2\ell_1+\frac{\cosh\ell_2-1}{2}} - \sqrt{\left(\cosh\ell_1-\frac{1}{2}\right)(\cosh\ell_2+1)}\right)^2  \end{align*}
Therefore $\sinh \left(X_{22}^1/2\right) - \sinh S \geq 0$.  Setting the above equal to zero and solving the resulting equation, we obtain $\cosh\ell_1=\cosh\ell_2$.\end{proof}

The main result of this subsection is the following improvement on Lemma 2.9 of \cite{DeSh}.  In introducing it we recall that Kojima--Miyamoto proved that $\cosh\ell_1\geq \frac{3+\sqrt{3}}{4}$ for every compact hyperbolic $3$-manifold with connected, totally geodesic boundary of genus two \cite[Corollary 3.5]{KM}.  Their proof again carries through to the current setting.

\begin{proposition}\label{KM ell_2} For $R$ as in Lemma \ref{disk radii} let $R'$ satisfy $\cosh R' = 3-\cosh R$ and define a function $E$ of $\ell_1$ by:\begin{align}
  & \cosh E = 1+\frac{2}{\cosh^2(R+R')\cdot\tanh^2\ell_1-1}   \label{E}
\end{align}
$E$ is decreasing for $\frac{3+\sqrt{3}}{4}\leq \cosh \ell_1 \leq 1.4$.  For $R''$ determined by $\cosh R'' = \frac{1}{2\sin (\pi/9)}=1.4619...$, define a quantity $M$ that depends on $\ell_1$ by \begin{align}
  & \cosh M = \sqrt{1+\frac{\cosh \ell_1+1}{\cosh(2R'') - 1}}  \label{M}  \end{align}
For a finite-volume hyperbolic $3$-manifold $N$ with compact, connected totally geodesic boundary of genus $2$, $\ell_2\geq\max\{\ell_1,E,M\}$.
\end{proposition}

\begin{remark}  This combines and improves Lemmas 2.8 and 2.9 of \cite{DeSh}, of which Lemma 2.8 re-recorded arguments in \cite{KM} (in and around Lemmas 4.2 and 4.3 there), and Lemma 2.9 gave a new, improved bound in a subinterval.  Of the functions $E$, $F$, $L$, and $M$ there, the work in Lemma \ref{disk radii} above implies that $F\geq E$, and we argue directly below that $L\geq M$.\end{remark}

\begin{proof}  If disks $U$ and $U'$, of radius $R$, and disks $V$ and $V'$ of radius $S$ are all embedded in $\partial N$ without overlapping then the sum of their areas is less than the area of $\partial N$, which is $4\pi$ by the Gauss--Bonnet theorem.  We obtain the following inequality:
$$ 4\pi(\cosh R -1) + 4\pi(\cosh S -1) \leq 4\pi, \quad\Rightarrow\quad \cosh S \leq 3 - \cosh R $$
Thus for $R'$ as defined above, $S \leq R'$. Now taking $R$ and $S = X_{12}^1- R$ as defined in Lemma \ref{disk radii}, and applying that result's conclusion, we find that $X_{12}^1 \leq R + R'$.  Since $X_{12}^1$ is decreasing in $\ell_2$ (recall Lemma \ref{ordering external lengths}), the upper bound on $X_{12}^1$ determines a lower bound on $\cosh\ell_2$.  Setting $X_{12}^1$ equal to $R+R'$ and solving for $\cosh\ell_2$ yields formula (\ref{E}) for $\cosh E$.  It was proved in \cite[Lemma 3.4]{DeSh} that $E$, so defined, decreases for $\frac{3+\sqrt{3}}{4}\leq \cosh \ell_1 \leq 1.4$.

It was proved in \cite[Lemma 2.9]{DeSh} that $\ell_2$ is also bounded below by $\min\{L,M\}$, for $\ell_1$ satisfying:\begin{align*}
  & \cosh \ell_1 \leq \frac{\cos (2\pi/9)}{2\cos (2\pi/9) - 1} = 1.43969... \end{align*}
where $L$ and $M$ are respectively defined as functions of $\ell_1$ in formulas (2.9.2) and (2.9.3) there.  (The formula for $M$ is reproduced above in (\ref{M}).)  We need only observe that $L$ is decreasing and $M$ is increasing as functions of $\ell_1$, and they agree at the right endpoint $\cos(2\pi/9)/(2\cos(2\pi/9)-1)$ of the relevant interval, to conclude that $\min\{L,M\} = M$ here.

Substituting $\cosh\ell_1 = \cos(2\pi/9)/(2\cos(2\pi/9)-1)$ in (\ref{M}), then simplifying, shows that $M = \ell_1$ here. For $\ell_1$ larger than this value, $M < \ell_1$, and it is true by definition that $\ell_2\geq\ell_1$, so on this interval we also have $\ell_2\geq\max\{\ell_1,E,M\}$.\end{proof}

\begin{remark}  It is worth noting just how far from sharp the lower bound on $\ell_2$ given by $E$ is at $\cosh\ell_1 = \frac{3+\sqrt{3}}{4}$.  For $N$ with this value of $\ell_1$, $\partial N$ decomposes into equilateral triangles of side length $\cosh^{-1}(3+2\sqrt{3})$, and an explicit calculation gives $\cosh\ell_2 = \frac{13+9\sqrt{3}}{4} \simeq 7.147$.  On the other hand, $\cosh E \simeq 2.893$ here.\end{remark}

\subsection{A bound for $\ell_1$ in the absence of a $(1,1,1)$-hexagon}\label{DeSh bounds}  
Proposition 3.9 of \cite{DeSh} bounds the first ortholength $\ell_1$ of a finite-volume hyperbolic $3$-manifold $N$ with compact totally geodesic boundary below by $\cosh\ell_1\geq 1.215$, assuming $\widetilde{N}$ has no $(1,1,1)$-hexagon.  This is significantly better than the sharp universal lower bound of $\frac{3+\sqrt{3}}{4}\simeq 1.183$ for $\cosh\ell_1$ proved in \cite{KM}.  (For comparison, all manifolds with volume less than $7.63$ in the census of Petronio et.~al.~have $\cosh\ell_1<1.213$.)

The main observation behind this result is simply that in the absence
of a $(1,1,1)$-hexagon, the arc length $d_{11}$ on $\partial N$ is
bounded below by $X_{11}^2$ instead of $X_{11}^1$;
this follows from the first assertion of Lemma \ref{boundary arc lengths}.
Here we will recast this observation to give an upper bound on $\ell_2$ in terms of $\ell_1$, and also improve the absolute lower bound of $1.215$ for $\cosh\ell_1$, in this setting.

\begin{proposition}\label{DeSh ell_2}  Let $N$ be a finite-volume hyperbolic $3$-manifold with compact, connected, totally geodesic boundary of genus two.  If $\widetilde{N}$ has no $(1,1,1)$-hexagon then
$$ \cosh\ell_2 \leq (2+2\sqrt{3})\sinh^2\ell_1 -1, $$
and $\cosh\ell_1 > 1.23$.\end{proposition}

\begin{proof}  It is a consequence of Bor\"oczky's theorem recorded in Corollary 3.5 of \cite{KM} and Lemma 2.7 of \cite{DeSh} that $\cosh d_{11}\leq 3+2\sqrt{3}$ with the hypotheses above.  (As usual, this is stated in \cite{DeSh} for $N$ compact, but it extends without revision to the finite-volume case.)  Using the fact that $d_{11}$ is at least $X_{11}^2$ in the absence of $(1,1,1)$-hexagons, substituting the right-hand side of formula (\ref{X112}) for $\cosh d_{11}$ in the inequality above, and solving for $\cosh\ell_2$ gives the upper bound on $\ell_2$ in terms of $\ell_1$.

For the absolute lower bound on $\cosh\ell_1$, we note that the upper bound we have just proved for $\ell_2$ is an increasing function of $\ell_1$, whereas the lower bound $E$ of Proposition \ref{KM ell_2} is decreasing.  Direct computation shows that $E$ takes the value $1.200...$ and the upper bound the value $1.194...$ when $\cosh\ell_1 = 1.23$ (and their values coincide when $\cosh\ell_1 \approx 1.2304$). Thus if there is no $(1,1,1)$-hexagon and $\cosh\ell_1\leq 1.23$, the lower bound on $\ell_2$ exceeds the upper bound, a contradiction.\end{proof}

\section{More Better Muffins}\label{bound muffin}

A ``muffin'', 
so named by Kojima and Miyamoto in \cite{KM} (see the discussion below Proposition 3.1 of
that paper), 
is a member of a certain class of hyperbolic solids of rotation, shown
in \cite{KM} to embed in a compact hyperbolic three-manifold with
totally geodesic boundary and used in the main volume bound 
of \cite{KM}. 
We used the same class of muffins from \cite{KM} in \cite{DeSh},
denoting them as ``$\mathrm{Muf}_{\ell_1}$'' 
in Definition 3.1 there. As this notation suggests, members of this class are determined up to isometry by a single parameter, which in their application to volume bounds is the length $\ell_1$ of the shortest return path. Here we will continue to use the notation of \cite{DeSh} for this class of muffins 
(see Definitions and Remarks \ref{M ell one} below).

This section introduces a more general class of muffins, still
hyperbolic solids of rotation, but now depending on two parameters
$\ell$ and $R$ which are side lengths of certain reflectively
symmetric hyperbolic pentagons. 
In Section \ref{bound vol} we will use muffins together with collars to give lower bounds on volume for a hyperbolic $3$-manifold $N$ with totally geodesic boundary satisfying certain bounds on the lengths of $\ell_1$ and $\ell_2$. Here we 
formally define these objects and establish technical results --- for muffins in Section \ref{emb muf} and collars in Section \ref{emb col} --- that will allow us to show that they are disjointly embedded in $N$.

\subsection{Muffins}\label{emb muf} Given
 $\ell, R >0$, let $Q$ be a
hyperbolic \textit{Lambert quadrilateral}---one with three right
angles---such that its sides having right angles at both endpoints are
of lengths $\ell/2$ and $R$. Doubling $Q$ across its edge $\omega$
opposite the one $\rho$ with length $R$ yields a reflectively
symmetric hyperbolic pentagon $P$ with four right angles. The
\textit{base} $\lambda$ of $P$---the side opposite the non-right
vertex---has length $\ell$, and each of the two sides intersecting it
have length $R$. 
A
 \textit{muffin} is the solid in $\mathbb{H}^3$ obtained by rotating $P$ around $\lambda$.  We illustrate this construction in Figure \ref{muffin} and formalize its definition below. 

\begin{figure}
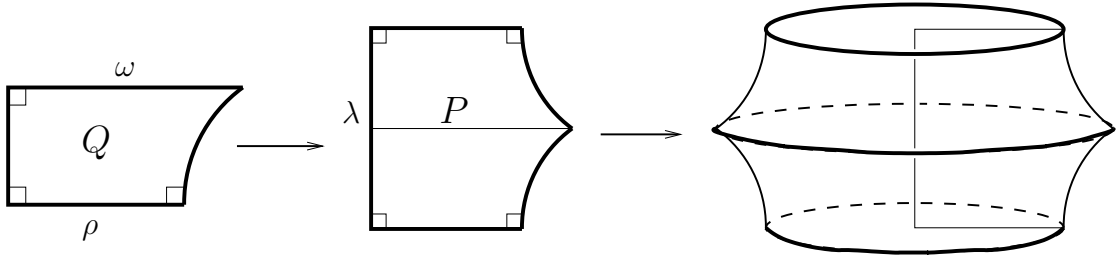
\caption{Making a muffin.}
\label{muffin}
\end{figure}

\begin{definition}\label{M ell R} The \textit{muffin} $M(\ell,R)$ with \textit{height} $\ell$ and \textit{cap radius} $R$ is the solid that results from rotating a reflectively symmetric pentagon $P$ with four right angles about its base $\lambda$ (the side opposite its not-necessarily-right angle, labeled in Figure \ref{muffin}), where $\ell$ and $R$ are respectively the lengths of $\lambda$ and of the sides of $P$ adjacent to $\lambda$---the side of $\rho$ of $Q$ in Figure \ref{muffin}, and its mirror image. We say that $M(\ell,R)$ is \textit{centered at $\lambda$}. Its \textit{caps} are the disks of radius $R$ obtained by rotating $\rho$ around its vertex $\rho\cap\lambda$ and likewise for its mirror image.
\end{definition}

We further say that the muffin has \textit{waist radius} $W$, the length of the side $\omega$ of $Q$ in Figure \ref{muffin}, and \textit{side altitude} $A$, where $A$ is the length of the final side of $Q$.  These lengths are related by the ``quadrilateral rule'' of hyperbolic trigonometry as follows:  \begin{align}\label{muffin relations}
  & \tanh A = \cosh R \tanh(\ell/2)  && \tanh W = \cosh(\ell/2)\tanh R  \end{align}
It follows that $M(\ell,R)$ is determined up to isometry by $\ell$ and $R$, or by $\ell$ and $W$.

\begin{definitionsremarks}\label{thog} For any $n\geq 2$ and any
  totally geodesic subspace $\Pi\subset\mathbb{H}^n$, there is a
  retraction 
$\pi\co \mathbb{H}^n\to\Pi$ 
that sends each point of $\mathbb{H}^n$ to its unique nearest point in $\Pi$. (See \cite{BO'Ne}, particularly Lemma 3.2 there.) We call this map the \textit{orthogonal projection} to $\Pi$. It has the property that for each $x\in \Pi$, $\pi^{-1}(x)\subset\mathbb{H}^n$ is a totally geodesic subspace of complementary dimension to $\Pi$, intersecting $\Pi$ orthogonally in $\{x\}$.
\end{definitionsremarks}

We give an explicit formula for the orthogonal projection to a certain plane $\Pi\subset\mathbb{H}^3$ in the proof of Proposition \ref{thog proj}.

Each assertion about muffins recorded 
in the Fact 
below follows from their construction as solids of rotation by an exercise in hyperbolic geometry.

\begin{numberedfact}\label{basic muffin} For any disjoint pair of totally
  geodesic planes $\Pi_1,\Pi_2\subset \mathbb{H}^3$ that attain a
  minimum distance $\ell >0$, and any $R>0$, the copy $M$ of
  $M(\ell,R)$ centered on the shortest geodesic arc joining $\Pi_1$ to
  $\Pi_2$ is contained in the convex set bounded by $\Pi_1$, and
  $\Pi_2$ and intersects $\Pi_1\cup\Pi_2$ 
precisely 
in the union of its caps. For such a muffin $M$:\begin{itemize}
	\item Every point of $M$ is 
at a distance at most $W$ from 
$\lambda$, where $\lambda$ and $W$ are respectively the center and waist radius of $M$. The interior of $M$ is contained in the open $W$-neighborhood of $\lambda$.
	\item For $i = 1$ or $2$, taking $C_i$ to be the cap of $M$ contained in $\Pi_i$ and $\pi_i\co\mathbb{H}^3\to\mathbb{H}^3$ to be the orthogonal projection to $\Pi_i$, we have $M\subset\pi_i^{-1}(C_i)$.
\end{itemize}
\end{numberedfact}

The proof of the following fact about right-angled hexagons uses 
the case $n=2$ of  orthogonal projection: for any hyperbolic geodesic
$ \lambda\subset \mathbb{H}^2$ we have an orthogonal projection
$\pi:\HH^2\to\lambda$, and $\pi^{-1}(x)\subset\mathbb{H}^2$ is also a
geodesic for any $x\in \lambda$.

\begin{proposition}\label{RAH} For any right-angled hexagon $C$ in
  $\mathbb{H}^2$, and 
any 
pair of opposite sides $\lambda$, $\lambda'$ of $C$, there is a geodesic arc $\delta\subset C$ that meets each of $\lambda$ and $\lambda'$ at right angles at a point in its interior.\end{proposition}

\begin{proof}
Let $\pi$ 
denote 
the orthogonal projection of $\mathbb{H}^2$ to the geodesic containing $\lambda$.

Because 
$C$ is the intersection of  half-planes bounded by the 
geodesics 
containing its sides, and
the edges of $C$ that share endpoints with $\lambda$ are orthogonal to
$\lambda$, we have
$\pi(C) = \lambda$. 
Hence 
$\mu\doteq\rho(\lambda')$ is a subsegment of $\lambda$. For each point 
 $P$ of $\mu$, the hyperbolic 
 geodesic 
 $\pi^{-1}(P)$ 
contains a 
segment  $s_P\subset C$ which meets $\lambda$
perpendicularly at $P$ and has its other endpoint in $\lambda'$.
Let us now choose
an endpoint $P_0$ of $\lambda'$, let $c$ denote the length of $\mu$,
and for each $x\in(0,c]$ let $P_x$ denote the point of $\mu$ whose
distance from $P_0$ is $x$. For each $x\in[0,c]$ we set
$s_x=s_{P_x}$. For every $x\in(0,c]$ there is a hyperbolic
quadrilateral $Q_x\subset C$ whose sides are $s_0$, $s_x$, and
subsegments $\mu_x$ and $\lambda'_x$ of $\mu$ and $\lambda'$
respectively. The two  interior angles of
$Q_x$  incident to $\mu_x$ are right
angles. The  interior angle between $\lambda'_x$  and $s_0$ is independent of $x$, and will be denoted $\beta$,
while the   interior angle between $\lambda'_x$  and $s_x$ will be denoted
$\alpha(x)$. The angles $\beta$ and $\alpha(c)$ are each less than $\pi/2$, since they are sub-angles of vertex angles of $C$.

The side $\mu_x$ of $Q_x$ has length $x$; we will denote the length of
its side $s_0$ by $a$. 
The hyperbolic law of cosines for quadrilaterals with  two right
angles, which is stated in \cite[VI.3.3]{Fenchel}, gives:
\[ \cos\alpha(x) = - \cos \beta\cosh x + \sin\beta\sinh x\sinh a. \]

Noting that $\cos\alpha(x)\to -\cos\beta$ as $x \to 0$, so that
$\alpha(x)\to\pi-\beta >\pi/2$, and 
recalling that $\alpha(c) < \pi/2$, we conclude from the intermediate
value theorem that 
$\alpha(x_0) = \pi/2$ for some 
$x_0\in (0,c')$.
Hence the segment $\delta\doteq s_{x_0}\subset C$ is perpendicular to
both $\lambda$ and $\lambda'$.
\end{proof}

For a given hyperbolic manifold $N$ with totally geodesic boundary, we will embed a muffin in $N$---for carefully chosen $\ell$ and $R$---by checking that a copy of $M(\ell,R)$ carefully placed in the universal cover $\widetilde{N}$ does not intersect its translates under the action of $\pi_1(N)$ by covering transformations. Here as in the prior works \cite{KM} and \cite{DeSh}, ``carefully placed'' will always mean centered at a lift $\tilde\lambda$ of a return path $\lambda$, 
 with height $\ell$ equal to the length of $\lambda$ 
 so that each of its caps lies in a component of $\partial\widetilde{N}$. We use the following result to ensure that $M(\ell,R)$, with choices carefully made, lies \textit{in $\widetilde{N}$}.

\begin{lemma}\label{You 'kay?} 
Let  $N$ be a hyperbolic manifold with geodesic boundary, and let
$\widetilde{N}$ denote its universal cover. If  $\tilde\lambda_k$ is a
lift 
of the $k$th-shortest return path to $\widetilde{N}$, the distance from $\tilde\lambda_k$ to any component of $\partial\widetilde{N}$ that does not contain either of its endpoints is at least $U_k$ defined by $\sinh U_k = \cosh \ell_1/\sinh(\ell_k/2)$. Equality is attained if $\tilde\lambda_k$ is an edge of a $(1,1,k)$-hexagon in $\widetilde{N}$.\end{lemma}

\begin{proof} For $\tilde\lambda_k$ as above, let $\Pi_1$ and $\Pi_2$
  be the components of $\partial\widetilde{N}$ containing its
  endpoints. For a third component $\Pi_3$ of $\widetilde{N}$, let $C$
  be the planar
right-angled hexagon supplied by Lemma \ref{rt ang hex}, 
containing $\tilde\lambda_k$ and intersecting each $\Pi_i$ perpendicularly in a side, for $i=1$, $2$, $3$. 
 Then $C\cap \Pi_3$ is the side of $C$ opposite $\tilde\lambda_k$. 
 
We now apply Proposition \ref{RAH}, with $\tilde\lambda_k$ and
$C\cap\Pi_3$ 
playing 
the roles of $\lambda$ and $\lambda'$ 
in that proposition. 
Let $\delta\subset C$ be the arc 
supplied by Proposition \ref{RAH}. 
The hexagon $C$ is divided into two right-angled pentagons by $\delta$. Let us take the length of $\delta$ to be $y$ and $\ell_i$, $\ell_j$, $\ell_k$ the lengths of $\tilde\lambda_i$, $\tilde\lambda_j$, and $\tilde\lambda_k$, respectively, and $x_i$ and $x_j$ those of the sub-arcs of $\tilde\lambda_k$ in the pentagons containing $\tilde\lambda_i$ and $\tilde\lambda_j$. So $x_i + x_j = \ell_k$, and the ``law of sines'' for mostly-right-angled pentagons recorded in \cite[VI.3.2]{Fenchel} gives:
\[ \cosh \ell_i = \sinh x_i\sinh y\quad\mbox{and}\quad \cosh\ell_j = \sinh x_j\sinh y. \]
Supposing without loss of generality that $\ell_j\ge \ell_i$, the equations above combine to imply that $x_j \ge x_i$ and hence that  $x_i \leq \ell_k/2$. Using the first equation we now obtain:
\[ \sinh y = \frac{\cosh \ell_i}{\sinh x_i} \ge \frac{\cosh\ell_1}{\sinh(\ell_k/2)} \]
The inequality above comes from the fact above that $x_i \leq \ell_k/2$, using the default bound $\ell_i \ge \ell_1$ in the numerator. Noting that if $C$ is a $(1,1,k)$ hexagon then $\delta$ is its axis of reflective symmetry and hence $x_i = x_j = \ell_k/2$, we obtain the Lemma's final assertion about equality.
\end{proof}

We will use 
Lemma \ref{You 'kay?} 
to give a 
sufficient condition 
in Lemma \ref{no x bdry} below, 
for a copy of $M(\ell,r)$ to lie 
entirely in $\widetilde{N}$. 
First we formally establish the link between the muffin $\mathrm{Muf}_{\ell_1}$ 
originally defined in \cite{KM} and \cite[Dfn.~3.1]{DeSh}, and $M(\ell,R)$ defined 
in this paper.

\begin{definitionsremarks}\label{M ell one}
In the notation of Definition \ref{M ell R}, the muffin $\mathrm{Muf}_{\ell_1}$ 
originally defined in \cite{KM} and \cite[Dfn.~3.1]{DeSh} is 
$M(\ell_1,R(\ell_1))$, where $R = R(\ell_1)$ is given by the formula (\ref{R}). 
From (\ref{muffin relations}) we therefore obtain the following formulas determining its side altitude $A = A(\ell_1)$ and waist radius $W = W(\ell_1)$ as functions of $\ell_1$ only:\begin{equation}\label{waist ell1}\begin{split}
	\tanh A & = \cosh R \tanh(\ell_1/2) =  \sqrt{\frac{2\cosh\ell_1 - 1}{2\cosh\ell_1+2}} \\
	\tanh W & = \cosh (\ell_1/2)\tanh R = \sqrt{\frac{\cosh\ell_1+1}{4\cosh\ell_1-2}}  \end{split}\end{equation}
\end{definitionsremarks}

The geometric motivation for the choice of cap radius for 
$\mathrm{Muf}_{\ell_1}$ is captured in Figure 3.1 of \cite{KM}. 
The left side of that Figure pictures a $(1,1,1)$-hexagon in our notation, 
with ``$\ell$'' there equal to $\ell_1$ and ``$R$'' and ``$A$'' given by (\ref{R}) 
and (\ref{waist ell1}). The Fact below records a feature of that picture.

\begin{numberedfact}\label{UAW} For $U_1$ as in Lemma \ref{You 'kay?} 
and $A = A(\ell_1)$ and $W = W(\ell_1)$ as in \ref{waist ell1}, 
we have
$U_1 = A + W$.
\end{numberedfact}

\begin{proof} 
Manipulating the formula that defines $U_1$ in Lemma \ref{You 'kay?} gives 
\[ \tanh U_1 = \frac{\cosh\ell_1}{\sqrt{\cosh^2\ell_1 + \sinh^2(\ell_1/2)}} = \frac{\sqrt{2}\cosh\ell_1}{\sqrt{(2\cosh\ell_1 - 1)(\cosh\ell_1+1)}} \]
The result now follows from 
the identity $\tanh(x+y) = \frac{\tanh x + \tanh y}{1+\tanh x\tanh y}$, using the formulas for the side altitude $A(\ell_1)$ and waist radius $W(\ell_1)$ 
of $\mathrm{Muf}_{\ell_1}$ from 
(\ref{waist ell1}).
\end{proof}

We use Fact \ref{UAW} in the proof of Lemma \ref{no x bdry} immediately below, 
then in a stronger way in the proof of Proposition \ref{embark}.

\begin{lemma}\label{no x bdry} For a hyperbolic $3$-manifold $N$ and a
  lift $\tilde\lambda_k$ of the $k$th-shortest return path of $N$ to
  the universal cover $\widetilde{N}$, a copy of $M(\ell_k,R)$
  centered at $\tilde\lambda_k$ is contained in $\widetilde{N}$ if its
  waist radius is less than 
the quantity $U_k$ defined in 
Lemma \ref{You 'kay?}. If so, it intersects $\partial\widetilde{N}$ 
precisely 
in its caps.

In particular, a copy of $\mathrm{Muf}_{\ell_1}$ 
centered at a lift $\tilde\lambda_1$ 
of $\lambda_1$ is contained in $\widetilde{N}$.
\end{lemma}

\begin{proof}[Proof of Lemma \ref{no x bdry}] It is useful here to recall that $\widetilde{N}$ is an intersection of hyperbolic half-spaces in $\mathbb{H}^3$, each bounded by a totally geodesic plane that is a component of the preimage of $\partial N$. By construction of $M(\ell_k,R)$, its caps are totally geodesic disks that each intersects its central geodesic arc at right angles at an endpoint. Since $\tilde\lambda_k$ intersects a component of $\partial\widetilde{N}$ at right angles at each of its endpoints, upon embedding $M(\ell_k,R)$ in $\mathbb{H}^3$ so that its center coincides with $\tilde\lambda_k$, each cap  is contained in one of these components, and $M(\ell_k,R)$ itself is contained in the intersection of the two half-spaces that they bound containing $\widetilde{N}$.

If $M(\ell_k,R)$ 
were 
not entirely contained in $\widetilde{N}$, it would thus intersect a third component of $\partial\widetilde{N}$, \ie one not containing either cap. By Lemma \ref{You 'kay?}, any such component has distance at least $U_k$ from $\tilde\lambda_k$. 
However, as observed in 
Fact \ref{basic muffin}, 
$M(\ell_k,R)$ is contained in the $W$-neighborhood of its central geodesic arc, where $W$ is its waist radius. This implies that if $W < U_k$ then $M(\ell_k,R)$, so embedded in $\mathbb{H}^3$, is entirely contained in $\widetilde{N}$ and intersects $\partial\widetilde{N}$ only in its caps.

By Fact \ref{UAW} we have 
$U_1 = A+W$. Therefore $W < U_1$, so by the above, a copy of $\mathrm{Muf}_{\ell_1}$ centered at a lift of $\lambda_1$ is entirely contained in $\widetilde{N}$.
\end{proof}

Kojima--Miyamoto proved in \cite[Lemma 3.2]{KM}, by an \textit{ad hoc} argument, that a copy of $\mathrm{Muf}_{\ell_1}$ centered at a lift of $\lambda_1$ embeds in $N$.  Here we develop a more systematic approach with the goal of identifying when \textit{two} muffins, centered at lifts of $\lambda_1$ and $\lambda_2$, embed disjointly in $N$. 

Here is the key observation. Let $N$ be a compact hyperbolic
$3$-manifold with totally geodesic boundary, and let $M$ and $M'$ be
muffins centered at lifts $\tilde\lambda$ and $\tilde\lambda'$ to
$\widetilde{N}$ of return paths of $N$.  Let $\Pi_1$ and $\Pi_2$ be
the components of $\partial \widetilde{N}$ containing the caps of $M$,
and let $\Pi_1'$ and $\Pi_2'$ play the same role for $M'$.  If $\Pi_i
\neq \Pi_j'$ for $i,j\in\{1,2\}$, then 
the four planes $\Pi_1,\Pi_2,\Pi_1',\Pi_2'$ determine a \textit{truncated tetrahedron} 
$\Delta$, for which the distance from $\tilde\lambda$ to $\tilde\lambda'$ is a 
\textit{transversal length}. These notions are defined carefully in the preprint 
\cite{DeB_TrigTruncTriTet}, which proves lower bounds on transversal length 
that we will use here to ensure that muffins do not overlap. Below we review 
their definitions in a form adapted to the current context.

Let $\mathcal{B} = \{\Pi_1,\Pi_2,\Pi_1',\Pi_2'\}$, and for $i = 1,2$, 
let $H_i$ (respectively, $H_i'$) be the half-space bounded by 
$\Pi_i$ (resp.~$\Pi_i'$) that contains $\widetilde{N}$ and hence the other 
$\Pi_j$ and $\Pi_j'$. Recall from Lemma \ref{mutual perp} that each 
three-element subset of $\mathcal{B}$ determines a 
geodesic plane that is perpendicular to all three of its members. It 
follows from \OneSide\ of \cite{DeB_TrigTruncTriTet} that either all 
four planes above have a common perpendicular plane $\Pi$, or that 
the common perpendicular to each three-plane subcollection bounds 
a single half-space that contains all three lifted return paths joining its members 
to the fourth plane. 
The \textit{truncated tetrahedron} $\Delta$ determined by $\mathcal{B}$ 
is described in \TrunchTeth. 
In the latter case, 
it 
is the intersection of the 
$H_i$ and $H_i'$ with the four half-spaces determined in this way by 
the three-element subcollections of $\mathcal{B}$. In the 
former, we take 
$\Delta$ 
to be the intersection of 
the common perpendicular plane $\Pi$ with the $H_i$ and $H_i'$, 
and say it is \textit{degenerate}.

The \textit{internal edges} of the truncated tetrahedron 
$\Delta$ 
defined 
as above are the lifted return paths joining each pair of distinct elements 
of $\mathcal{B}$. We say that two internal edges are 
\textit{opposite} if no single member of $\mathcal{B}$ contains an 
endpoint of each. (Thus each internal edge is opposite a unique 
other internal edge.) The minimum distance 
between a pair of opposite internal edges is a \textit{transversal length} 
of 
$\Delta$. 
 (In general, its value depends on the choice of edges).

In the present setting, the lifts $\tilde\lambda$ and
$\tilde\lambda'$ are opposite edges of 
$\Delta$. 
Using results of \cite{DeB_TrigTruncTriTet}, one can then bound the distance between $\tilde\lambda$ and
$\tilde\lambda'$  
in terms of their lengths and the lengths of the other edges of 
$\Delta$. 
We illustrate the use of this philosophy below with a re-proof of \cite[Lemma 3.2]{KM}. 

\begin{lemma}\label{KM muffin embed}  Let $N$ be an orientable
  hyperbolic $3$-manifold with compact totally geodesic boundary of
  genus 
$2$. 
For $\mathrm{Muf}_{\ell_1}$ as in Definitions and Remarks \ref{M ell one}, a
copy of $\mathrm{Muf}_{\ell_1}$ embedded 
in $\widetilde{N}$ centered at a lift of $\lambda_1$ is embedded in $N$ by the universal covering.  \end{lemma}

\begin{proof} 
By Lemma \ref{no x bdry}, a copy $M$ of $\mathrm{Muf}_{\ell_1}$ centered at a lift of
$\lambda_1$ is contained in $\widetilde{N}$. The universal cover thus
embeds it in $N$ if it is disjoint from all of its translates by the
action of $\pi_1(N)$ on $\widetilde{N}$. These translates are copies
of $\mathrm{Muf}_{\ell_1}$ centered at other lifts of $\lambda_1$. If
a component $\Pi$ of $\partial \widetilde{N}$ contains a cap of $M$
and that of a translate $M'$, then $M\cap M' = \emptyset$.  This is
because each of $M$ and $M'$ is contained in the preimage under
orthogonal projection of its intersection with $\Pi$, and disks of
radius $R$ about feet of lifts of $\lambda_1$ are disjoint or equal 
(as follows from 
Lemma \ref{disk radii}).

Now suppose that 
$M$ and a translate $M'$ do not have caps on a common component of
$\partial \widetilde{N}$. Then the 
four 
 planes containing their caps
determine a truncated tetrahedron 
$\Delta$ as in \TrunchTeth. 
Let $\lambda$ and $\lambda'$
be the centers of $M$ and $M'$, respectively, each a lift of
$\lambda_1$, hence having 
length $\ell_1$.  
The distance from $\lambda$ to $\lambda'$ is a \textit{transversal length} 
of $\Delta$, again as in \TrunchTeth. It is therefore given by 
$T(x,x;a,b,c,d)$ 
as in \TetraTrans, where 
$x = \cosh\ell_1$, and 
$a$, $b$, $c$, and $d$ 
are the hyperbolic cosines of lengths of the distinct edges of $\Delta$ not equal to 
$\lambda$ or $\lambda'$. 
Since 
these 
other edges of 
$\Delta$ 
are also lifts of return
paths, each also has length at least $\ell_1$. 
Therefore by \ShrinkingTrans\ 
we have: 
$$  \cosh T(x,x;a,b,c,d) \ge \cosh D \doteq \frac{2x}{x -1}  $$
On the other hand, the waist radius $W$ of $M$ and $M'$ is determined by 
(\ref{waist ell1}). 
Since each of $M$ and $M'$ is contained in the
$W$-neighborhood 
of
its center, 
$M$ does not intersect $M'$ if $W < D/2$.  Using the formula 
above 
for $D$, we obtain:
$$  \tanh(D/2) = 
\sqrt{\frac{x+1}{3x-1}} = 
\sqrt{\frac{\cosh\ell_1+1}{3\cosh\ell_1-1}}  $$
Since $3x-1 < 4x-2$ for $x>1$ and $y \mapsto \tanh y$ is increasing, it follows that $D/2 > W$, hence that $M$ does not intersect $M'$.
\end{proof}

We now 
turn our attention to the problem of 
embedding two muffins, disjointly, centered at the feet of the shortest and second-shortest return paths. We will use the following helpful sign computation.

\begin{lemma}\label{S derivs} For $R$ as in (\ref{R}) and $X_{ij}^k$
  as in (\ref{Xijk}), regard $S = X_{12}^1 - R$ as a function of
  variables $x = \cosh\ell_1$ and $y = \cosh\ell_2$. 
On the region defined by
$x > 1$ and $y\le 3$, the function $S$ is strictly increasing in the
variable $x$, strictly decreasing in the
variable $y$, 
and positive-valued.
\end{lemma}

\begin{proof} We note that since $R$ depends only on $x$, 
it follows from Lemma \ref{ordering external lengths} that $S$ is
strictly decreasing in $y$.

We now claim that $\partial S/\partial x > 0$. From (\ref{R}) we obtain that $\sinh R = (2x - 2)^{-1/2}$, hence:
\[ \frac{\partial R}{\partial x} = \frac{-(2x-2)^{-3/2}}{\cosh R} = \frac{-1}{(2x-2)\sqrt{2x-1}} \]
Using the description of $X_{12}^1$ from (\ref{X121}), we obtain:\begin{align*}
	\frac{\partial X_{12}^1}{\partial x} & = \frac{1}{\sinh X_{12}^1}\cdot \frac{-1}{(x^2-1)^{3/2}}\sqrt{\frac{y+1}{y-1}} = \frac{-1}{x^2-1}\sqrt{\frac{y+1}{y+2x^2-1}} \end{align*}
Therefore\begin{align}\label{partial S}
	(x-1)\frac{\partial S}{\partial x} = \frac{1}{2\sqrt{2x-1}} - \frac{1}{x+1}\sqrt{\frac{y+1}{y+2x^2-1}} \end{align}
The right-hand side above vanishes at $x =1$, but we are interested only in $x>1$. Setting the right-hand side above equal to $0$ and solving for $y$ yields 
\[ y = \frac{2x^3+6x^2+7x-3}{5-x} \] 
This is an increasing function of $x$ with values greater than $3$ for $x>1$. Inserting the test values $x=2$ and $y=2$ into the right-hand side of (\ref{partial S}) gives a positive value. Thus for $x$ and $y$ in the range of interest, $\partial S/\partial x (x,y) > 0$, proving the 
claim. It therefore follows that $S$ is strictly increasing in $x$ for $x>1$ and $y\le 3$.

We note that by (\ref{R}) and (\ref{X121}), both $R$ and $X_{12}^1$
increase without bound for any fixed $y$ as $x$ approaches $1$ from
the right. However, the formula for $\sinh S$ obtained in (\ref{sinch
  ess}) 
shows that $S(x,3)$ has the limit $\sqrt2/4$ as $x$ approaches $1$
from above. In view of the monotonicity established above, it follows
that  $S$ is bounded below by $\sqrt2/4$, and is in
particular strictly positive, 
for $x>1$ and $y\le 3$.
\end{proof}

The first main result of this section is the muffin embedding criterion below. Based on numerical exploration, we expect that in general, muffins which are sized so that their caps embed in $\partial N$ without overlapping will themselves embed in $N$ without overlapping. But having not been able to prove this general expectation, we settle here for a criterion that can be numerically checked on a rectangle of possible values of $\ell_1$ and $\ell_2$. 

\begin{proposition}\label{outside Muf_ell1}  Regard $R$ as in
  (\ref{R}), $S = X_{12}^1 - R$ as in Lemma \ref{disk radii}, for
  $X_{ij}^k$ as in (\ref{Xijk}), and $W$ as in (\ref{waist ell1}), as
  functions of variables $x = \cosh\ell_1$ and $y =
  \cosh\ell_2$. Suppose $N$ is an orientable hyperbolic $3$-manifold
  with $\partial N$ compact, connected, totally geodesic and of genus
  $2$, such that the lengths $\ell_1$ and $\ell_2$ of its shortest and
  second-shortest return paths satisfy
\[ a\leq \cosh\ell_1 \leq b \quad \mbox{and}\quad c\le \cosh\ell_2 \le d , \]
for some given $a,b,c,d$ with $1<a<b\le2$ and $1<c<d\le3$. 
If $W_S^0 \le U_2^0$, $W_R^0 + W_S^0 \leq T_{12}^0$, and $2W_S^0 \le T_{22}^0$, where $\sinh U^0_2 = \frac{a\sqrt{2}}{\sqrt{d-1}}$,\begin{align}\label{waisting away}
	& 
W_R^0 = 
	\tanh^{-1}\left(\sqrt{\frac{a+1}{4a-2}}\right), & 
	& W_S^0 = \tanh^{-1}\left(\sqrt{\frac{d+1}{2}}\tanh S(b,c)\right), \nonumber\\
	& T_{12}^0 = \cosh^{-1}\left(\frac{2b}{\sqrt{(b-1)(d-1)}}\right), &
	& \mbox{and}\quad T_{22}^0 =
          \cosh^{-1}\left(\frac{2a}{d-1}\right), \end{align}
then the interior of a copy of $M(\ell_2, S)$, centered at a lift of
$\lambda_2$, is contained in $\widetilde{N}$ and embeds in $N$ without
overlapping that of $\mathrm{Muf}_{\ell_1}$ from Lemma \ref{KM muffin
  embed}. 
(Note that by Lemma \ref{S derivs}, the quantity $S(x,y)$ is
positive-valued for $a\le x\le b$ and $c\le y\le d$, so that $W^0_S$
is well-defined and positive, and $M(\ell_2,S)$ is defined.)
\end{proposition}

\begin{proof}
We claim first that $W_S^0$ 
is an upper bound for the waist radius $W_2$ of $M(\ell_2,S)$, taken as a function of $x$ and $y$, on the rectangle $[a,b]\times[c,d]$. By the formula (\ref{muffin relations}), $W_2$ satisfies $\tanh W_2 = \cosh(\ell_2/2)\tanh S$. Lemma \ref{S derivs} implies that as a function of $x$ and $y$, $S$ increases with $x$ and decreases with $y$ at all points of this rectangle. The values of $S$ are thus bounded above by $S(b,c)$ on $[a,b]\times[c,d]$, and the claim thus follows from the ``half-angle identity'' for hyperbolic cosine.

For $ a\leq \cosh\ell_1 \leq b $ and $c\le \cosh\ell_2 \le d $,
the constant $U_2^0$ defined
in the statement of the proposition 
is a lower bound for
the quantity $U_2$ defined in Lemma \ref{You 'kay?}. Combining this
observation with the hypothesis $W_S^0\leq U_2^0$ and the claim just
proved, we deduce that on $W_2\le U_2$.
Lemma \ref{no x bdry} then implies 
that a copy of $M(\ell_2,S)$ centered at a lift of $\lambda_2$ is contained in $\widetilde{N}$. 

To show that the interior of $M(\ell_2,S)$ embeds in $N$ is equivalent
to showing that  it does not overlap any of its translates under the 
action of $\pi_1(N)$ on the universal cover $\widetilde{N}$ by covering transformations. 
Likewise, showing that the image of $M(\ell_2,S)$ is disjoint in $N$ from that of $\mathrm{Muf}_{\ell_1}$ is equivalent to showing that $M(\ell_2,S)$ does not overlap any translate of the latter muffin. We take these tasks on below.

Suppose first that a single component $\Pi$ of $\partial\widetilde{N}$ contains caps of both $M(\ell_2,S)$ and either a translate of $\mathrm{Muf}_{\ell_1}$ or of itself. In the first case, these caps are discs of radius $S$ and $R$ in $\Pi$, respectively, centered at the feet of lifts of $\lambda_2$ and $\lambda_1$; in the second, both are of radius $S$ and centered at feet of lifts of $\lambda_2$. By Lemma \ref{disk radii}, disks of radii $S$ and $R$ are disjointly embedded in $\partial N$ around the feet of $\lambda_2$ and $\lambda_1$; hence in either case the disks in question here do not overlap. Therefore, since each muffin is contained in the preimage of its cap under the orthogonal projection to $\Pi$, $M(\ell_2,R)$ also does not overlap this translate of $\mathrm{Muf}_{\ell_1}$, or in the second case, of itself.

Now consider a translate of $\mathrm{Muf}_{\ell_1}$ or of $M(\ell_2,S)$ with the property that neither of its caps is contained in a component of $\partial\widetilde{N}$ that also contains a cap of $M(\ell_2,S)$. Let $\tilde\lambda_2$ be the lift of $\lambda_2$ at which $M(\ell_2,S)$ is centered. In the first case let $\tilde\lambda_1$ be the center of the copy of $\mathrm{Muf}_{\ell_1}$ in question, and in the second let $\tilde\lambda_2'$ be the center of the translate of $M(\ell_2,S)$.  There is a truncated tetrahedron $\Delta$ with $\tilde\lambda_2$ as one edge and either $\tilde\lambda_1$ or $\tilde\lambda_2'$ its opposite, depending on the case, whose other edges are also lifts of return paths of $N$.  Applying \TetraTrans\ and \ShrinkingTrans\ 
as in the proof of Lemma \ref{KM muffin embed} 
shows that the transversal length of $\Delta$ is 
bounded below by a function 
$T_{12}(x,y)$ or $T_{22}(x,y)$ in the respective cases, given by
\[ T_{12}(x,y) = \cosh^{-1}\left(\frac{2x}{\sqrt{(x-1)(y-1)}}\right)\quad\mbox{and}\quad T_{22}(x,y) = \cosh^{-1}\left(\frac{2x}{y-1}\right), \]
for $x = \cosh\ell_1$ and $y = \cosh\ell_2$. Both $T_{12}$ and $T_{22}$ plainly decrease with $y$, for fixed $x$, and a computation shows that $\partial T_{12}/\partial x (x,y) < 0$ for $x<2$ whereas $T_{22}$ increases with $x$. Therefore their values on the rectangle $a \le x \le b$, $c\le y \le d$ are respectively bounded below by $T^0_{12} = T_{12}(b,d)$ and $T_{22}^0 = T_{22}(a,d)$, as given above in (\ref{waisting away}). 

We have seen that 
the waist radius $W_2$ of $M(\ell_2,S)$ is bounded
above by $W^0_S$. We will similarly show below 
that the waist radius
of $\mathrm{Muf}_{\ell_1}$ is bounded above by 
$W^0_R$.
As recorded in 
Fact \ref{basic muffin}, 
a  muffin with waist radius $W$ is contained in the $W$-neighborhood of
its center. Thus 
the hypotheses $W^0_R + W^0_S \leq T_{12}^0$ and 
$2W^0_S \leq T_{22}^0$ will imply, respectively, that the interior of $M(\ell_2,S)$ is disjoint from
the translate of $\mathrm{Muf}_{\ell_1}$ and from
$M(\ell_2,S)$.

The waist radius of $\mathrm{Muf}_{\ell_1}$ is $W(x)$ as given by the
formula (\ref{waist ell1}). Manipulating that formula shows that $W$
is a decreasing function of $x$ 
for $x>1/2$, 
and hence is bounded above on the interval $a\le x\le b$ by $W(a) =
W^0_R$, 
as asserted above.
\end{proof}

\subsection{Collars}\label{emb col} We now turn our attention to embedding collars, and controlling their interaction with muffins. 
The 
definition 
of collar given below 
is implicitly used in  \cite{KM} and
\cite{DeSh}. 
To motivate it, we recall from Section \ref{orthoarc} that for a hyperbolic $3$-manifold $N$ 
with totally geodesic boundary, we take the universal cover $\widetilde{N}$ of $N$ to be a 
convex subset of $\mathbb{H}^3$ bounded by a collection of geodesic hyperplanes. For a 
component $\Pi$ of $\partial \widetilde{N}$, let $\pi$ be the orthogonal projection to $\Pi$ from 
Definitions and Remarks \ref{thog}. Then for $x\in\Pi$, $\pi^{-1}(x)$ is a geodesic intersecting 
$\Pi$ orthogonally at $x$, and $\pi^{-1}(x)\cap\widetilde{N}$ is either a ray with endpoint $x$ or 
a segment with one endpoint at $x$ and the other on a different component of $\partial\widetilde{N}$.

\begin{definition}\label{collar def}
Let $N$ be  a hyperbolic $3$-manifold with totally geodesic boundary. If
$x$ is a point of $\partial N$ such that there exists a (necessarily
unique) geodesic path 
in $N$ 
(parametrized by arclength) which begins at $x$,
is perpendicular to $\partial N$
at its initial point, and has terminal point in $\partial N$, we
shall denote the length of this path by $h_x$, the interval
$[0,h_x]$ by $J_x$, and the path itself by $\alpha_x:J_x\to N$. If $x$ is
a point of $\partial N$ for which no path of this type exists, then
there is a unique geodesic ray beginning at $x$; in this case, we
shall set $h_x=+\infty$ and $J_x=[0,\infty)$, and denote (the
arclength parametrization of) the ray by 
$\alpha_x:[0,\infty)\to N$. For a subset $S$ of $\partial N$, we
define the
\textit{\redcollar\ of $S$ in $N$ with height $h>0$} to be the set
$\bigcup_{x\in S}\alpha_x([0,\min(h,h_x)))$. We shall say that the collar is
\textit{embedded} if $h \le \inf\{h_x\,|\,x\in S\}$, and the  map from
$S\times[0,h]$ to $N$ defined by $(x,t)\mapsto\alpha_x(t)$ is one-to-one.
\end{definition}

We first state a result that was 
implicitly used, although not explicitly stated, in the main volume bound of \cite{KM} and \cite{DeSh}.

\begin{proposition}\label{embark} Suppose $N$ is an orientable hyperbolic $3$-manifold
  with $\partial N$ compact, connected, totally geodesic and of genus
  $2$. Let $M$ be the projection to $N$ of a copy of
  the muffin $\mathrm{Muf}_{\ell_1}$ from Definitions and Remarks \ref{M ell one}, 
centered at a lift of $\lambda_1$ to the universal cover $\widetilde{N}$, 
and let $C, C'\subset\partial N$ be the two caps of $M$. 
For $A$ defined as in (\ref{waist ell1}), an \redcollar\ of $\partial 
N-(C \cup C')$ with 
any given 
height $h\le\min\{A,\ell_2/2\}$  
is embedded in $N$, disjointly from the interior of $M$.
\end{proposition}

\begin{proof}
For a totally geodesic plane $\Pi\subset\mathbb{H}^3$, the
``Hadamard--Hermann theorem'' \cite{Hermann} implies that the
exponential map restricts to a diffeomorphism from the normal bundle
$\nu(\Pi)$ of $\Pi$ onto $\mathbb{H}^3$. 
Since an orientation of the bundle $\nu(\Pi)$ gives an
identification of its total space with $\Pi\times\mathbb{R}$, 
it follows that there is a diffeomorphism from $\Pi\times (-h,h)$ to a
metric open neighborhood of $\Pi$ in $\mathbb{H}^3$ taking $(x,t)$ to
$\alpha_x(t)$ for any $x\in\Pi$ and 
any 
$t\in (-h,h)$, where $\alpha_x\co\mathbb{R}\to\mathbb{H}^3$ is an arclength-parametrized geodesic intersecting $\Pi$ orthogonally at $\alpha_x(0) = x$ for a choice of $\alpha_x'(0)$ depending continuously on $x$.

Now 
if $N$ is a hyperbolic manifold with totally geodesic boundary as in
the statement, we let $\tN\subset\HH^3$ denote the universal covering
of $N$, we 
take $\Pi$ to be a component of $\partial \widetilde{N}$, 
we
 take $h\le \min\{A,\ell_2/2\}$, and 
we choose the orientation of the bundle $\nu(\Pi)$ in such a way that 
$\alpha_x'(0)$ points into $\widetilde{N}$ for each $x\in\Pi$. We
claim that for each $x\in\Pi$ outside the preimage of $C\cup C'$, and 
each 
$t\in[0,h)$, $\Pi$ is the unique closest component of
$\partial\widetilde{N}$ to $\alpha_x(t)\in\widetilde{N}$. In
particular, this implies that 
the quantity denoted
 $h_x$ in 
Definition \ref{collar def} is at least $h$.

Suppose the claim does not hold, fix some $x_0\in\Pi$ outside the
preimage of $C\cup C'$ for which it fails, and let $t_0$ be the
infimum of the set of $t\in[0,h)$ such that a component $\Pi'\ne \Pi$
of $\partial\widetilde{N}$ is at least as close to $\alpha_{x_0}(t)$
as $\Pi$. Since $t_0$ is infimal with this property,
$\alpha_{x_0}(t_0)$ has distance $t_0$ from $\Pi'$ as well as
$\Pi$. Let $\Pi^{\perp}$ be the totally geodesic plane containing
$\alpha_{x_0}(t_0)$ and the shortest geodesic 
segments 
$\tau$ and $\tau'$ joining it to each of $\Pi$ and $\Pi'$,
respectively. Then $\Pi^{\perp}$ intersects $\Pi$ and $\Pi'$
orthogonally in geodesics $\gamma$ and $\gamma'$, respectively, and
there is a pentagon $P$ in $\Pi^{\perp}$ with sides consisting of
$\tau$ and $\tau'$, segments of $\gamma$ and $\gamma'$, and the
shortest geodesic 
segment 
$\lambda$ joining $\gamma$ to $\gamma'$.

Note that $\lambda$ is necessarily orthogonal to each of $\gamma$ and
$\gamma'$, hence also to $\Pi$ and $\Pi'$, so it is a lift of a return
path of $N$. Its length is at most $2t_0<2h \le \ell_2$,
so $\lambda$ must have length 
$\ell_1$, which 
must be strictly
less than $\ell_2$ in this situation. 
Hence 
$\lambda$ is a lift of $\lambda_1$ and therefore is the center of a
lift of the muffin $M$. The pentagon $P$ is symmetric under a
reflection fixing $\alpha_{x_0}(t_0)$ and exchanging the sides
containing it. 
It is divided into two quadrilaterals by the axis of this reflection. Let $Q$ be the resulting Lambert quadrilateral that contains $x_0$, and let $r_0$ be the length of its side joining $x_0$ to an endpoint of $\lambda$. The sides abutting this one have lengths $\ell_1/2$ and $t_0$, so the quadrilateral rule gives
\[ \cosh r_0 = \frac{\tanh t_0}{\tanh(\ell_1/2)} < \frac{\tanh A}{\tanh(\ell_1/2)} = \cosh R. \]
The inequality above comes from the hypothesis that $t_0<h\le A$, and the quantity $R$ above is the one defined in  (\ref{R}). But this is the radius of the caps $C$ and $C'$ of $M$, so the inequality implies that $x_0\in C\cup C'$, contradicting our hypothesis. The claim follows.

The claim implies that the height-$h$ \redcollar\ of the preimage in $\Pi$ of $\partial N - (C\cup C')$ is contained in $\widetilde{N}$ and does not intersect the height-$h$ \redcollar\ of any other component of $\partial\widetilde{N}$. Therefore the universal cover induces an embedding of the height-$h$ \redcollar\ of $\partial N - (C\cup C')$ in $N$.

We now recall 
from Fact \ref{UAW} 
that $U_1 = A + W$. 
 Here $U_1$, 
defined in Lemma \ref{You 'kay?}, is 
a lower bound on the
 distance between a lift of $\lambda_1$ and any component of
 $\partial\widetilde{N}$ not containing 
either of
its endpoints. 
Since $\mathrm{Muf}_{\ell_1}$ is contained in the $W$-neighborhood of its center, and the collar height $h$ is less than $A$, the height-$h$ \redcollar\ of $\Pi$ in $\widetilde{N}$ thus does not intersect any lift of $M$ that has no cap on $\Pi$. This implies that the height-$h$ \redcollar\ of $\partial N - (C\cup C')$ in $N$ is disjoint from $M$.
\end{proof}

The second main result of this section is a counterpart to Proposition \ref{outside Muf_ell1} that gives a criterion for setting the height of a collar of $\partial N$ that interacts well with a copy of $M(\ell_2,S)$.

\begin{proposition}\label{Et tu?}
Regard $R$ as in (\ref{R}) and $S = X_{12}^1 - R$ as in Lemma
\ref{disk radii}, for $X_{ij}^k$ as in (\ref{Xijk}), as functions of variables $x = \cosh\ell_1$ and $y = \cosh\ell_2$.
Suppose $N$ is an orientable hyperbolic $3$-manifold with $\partial N$ compact, connected, totally geodesic and of genus $2$, such that the lengths $\ell_1$ and $\ell_2$ of its shortest and second-shortest return paths satisfy 
\[ a\leq \cosh\ell_1 \leq b\le 2\quad \mbox{and}\quad c\le \cosh\ell_2 \le d \le 3, \]
and let $M$ 
denote
the projection to $N$ of a copy of $M(\ell_2,S)$ in $\widetilde{N}$ centered at a lift of $\lambda_2$.
Let $U_2^0$ and $W_S^0$ be defined
as in Proposition \ref{outside Muf_ell1}, and set
$H^0 = U_2^0 - W^0_S$. If $H^0 > 0$, 
the intersection 
of
the interior of $M$ 
with an
\redcollar\ of $\partial N$ in $N$ is the projection to $N$ of a union of \redcollar s of the caps of $M(\ell_2,S)$. 
\end{proposition}

\begin{proof} Let $\widetilde{M}$ be a copy of $M(\ell_2,S)$ in $\widetilde{N}$
  centered at a lift $\tilde\lambda_2$ of $\lambda_2$, and let $\Pi_1$
  and $\Pi_2$ be the components of $\partial\widetilde{N}$ containing
  the caps of $\widetilde{M}$. For $i = 1$ or $2$, since $\widetilde{M}$ is contained in the
  preimage of its cap in $\Pi_i$ under the orthogonal projection
  $\mathbb{H}^3\to\Pi_i$ it intersects a height-$H^0_2$ collar
  neighborhood of $\Pi_i$ in a \redcollar\  of its cap in
  $\Pi_i$. Therefore the projection of $\widetilde{M}$ to $N$ intersected with a
  height $H^0$ collar of $\partial N$ contains the union of the collar
  neighborhoods of its caps in $\Pi_1$ and $\Pi_2$, and any other
  point in this intersection 
lies in 
the projection of the intersection of
  $\widetilde{M}$ with the height-$H^0$ collar of another component $\Pi_3$ of
  $\partial\widetilde{N}$. 
  We will show below that
the latter 
intersection is empty, for any given such $\Pi_3$.

Recall from the proof of Lemma \ref{no x bdry} that $\widetilde{M}$ is
contained in a 
radius-$W_2$ 
neighborhood of its center $\tilde\lambda_2$, where 
$W_2$ 
is its waist radius, and from Lemma \ref{You 'kay?} that 
the quantity $U_2$ defined in that lemma is a lower  bound for 
the distance from $\tilde\lambda_2$ to $\Pi_3$.
From the proof of Proposition \ref{outside Muf_ell1}, for $(x,y)$ in
the rectangle $[a,b]\times[c,d]$, 
$W_2$ 
is bounded above by $W^0_S$ and $U_2$ is bounded below by 
$U_2^0$. 
Therefore by the triangle inequality, no point lies in both $\widetilde{M}$ and a height-$H^0$ \redcollar\  of $\Pi_3$.
\end{proof}

\section{Volume bounds}\label{bound vol}

\numberwithin{para}{section}

It is recorded in Lemma 3.3 of \cite{KM} that the volume of $\mathrm{Muf}_{\ell_1}$ decreases with $\ell_1$. The Lemma below records a related derivative for later reference.

\begin{lemma}\label{Mufell1 volume change} For the muffin
  $\mathrm{Muf}_{\ell_1}$ from 
Definitions and Remarks \ref{M ell one}, 
  depending
  on a parameter $\ell_1$, 
let $\mathit{VM}(x)$ record its volume as a function of $x = \cosh\ell_1$. Then\begin{align}\label{dvol Mufell1}
	\mathit{VM}'(x) = 2\pi A \frac{d}{dx} (\cosh R) < 0
\end{align}
where $R$ as in (\ref{R}), defined by $\cosh R(x) = \sqrt{1+1/(2x-2)}$ as a function of $x$, is the cap radius of $\mathrm{Muf}_{\ell_1}$ and $A$ is its altitude length, defined by $\cosh (2A) = (4x+1)/3$.
\end{lemma}

\begin{proof}
The formula above for $R$ comes from (\ref{R}), 
the one for $A$ can be deduced from (\ref{muffin relations}),
 and each matches one given on the
  first page of \cite[\S 3]{KM}. 
According to 
Lemma 3.3
  of \cite{KM},
for $A$ and $R$ as above, 
we have:
\begin{align}\label{Muf_l1}  \mathrm{vol}(\mathrm{Muf}_{\ell_1}) = 2\pi \left( A\cosh R - \frac{\ell_1}{2} \right) = \pi\left( (2A)\cosh R - \cosh^{-1}(x) \right).  \end{align}
(As is pointed out in the proof 
of \cite[Lemma 3.3]{KM},
this equality is included in
the formula stated on p.~213 of \cite{Fenchel}, which itself records an old result in hyperbolic geometry.)

Taking a derivative with respect to $x$ we obtain:
$$ \mathit{VM}'(x) = \pi\left[(2A)[\cosh R]' + (2A)'\cosh R - \frac{1}{\sqrt{x^2-1}}\right] = 2\pi A[\cosh R]' $$
The second 
equality 
above uses a calculus computation showing that $(2A)'\cosh R = \frac{1}{\sqrt{x^2-1}}$.  
\end{proof}

\begin{proposition}\label{SglMfn} Let $N$ be an orientable hyperbolic $3$-manifold with $\partial N$ compact, connected, totally geodesic, and of genus $2$, and let $x = \cosh\ell_1$ and $y=\cosh\ell_2$, where $\ell_1$ and $\ell_2$ are the respective lengths of the shortest and second-shortest return paths of $N$. For $x\in[1.24,1.5]$, define a piecewise-constant function $Y$ of $x$ as follows:  \begin{itemize}
  \item  for $x \in [1.24,1.25)$, $Y(x) = 1.986$;
  \item  for $x\in [1.25,1.27)$, $Y(x) =  1.9$;
  \item for $x\in [1.27,1.3)$, $Y(x) = 1.8$;
  \item  for $x \in [1.3,1.35)$, $Y(x) = 1.68$;
  \item for $x \in [1.35,1.4)$, $Y(x) =  1.59$; 
  \item for $x\in [1.4,1.45)$, $Y(x) =  1.55$; and
  \item for $x\in [1.45,1.5)$, $Y(x) = 1.52$.
\end{itemize}
If $x\in [1.24,1.5)$ and $y\ge Y(x)$, or if $x\ge1.5$, then $\mathrm{vol}(N)\geq7.4$.\end{proposition}

\begin{proof}  To prove this we will use the fundamental volume bound of \cite{KM}, 
a result which implicitly relied on Proposition \ref{embark}, and 
which we recorded as Lemma 3.2 in our prior work \cite{DeSh}.  Taking $R$ as in (\ref{R}) as a function of $x = \cosh\ell_1$, this is: \begin{align}\label{KM vol ineq} 
\mathrm{vol}(N) \geq \mathit{VM}(x) + \pi(2-\cosh R)(2H + \sinh(2H)),  \end{align}
where $\mathit{VM}(x)$ records the volume of $\mathrm{Muf}_{\ell_1}$ as in Lemma \ref{Mufell1 volume change}, and for $A$ equal to the altitude length of $\mathrm{Muf}_{\ell_1}$ as described there, $H = \min\{A,\ell_2/2\}$.

Let us regard the right-hand side of (\ref{KM vol ineq}) as defining a
function $V(x,H)$ of two independent variables. For fixed $x$, this
function clearly increases with $H$. Moreover, the formula for $A$ in
Lemma \ref{Mufell1 volume change} defines it as an increasing function
of $x$, and a computation shows that when $x = 1.24$,
$\cosh(2A)=1.98\bar{6}$. Therefore if $N$ is as in the Proposition's
hypotheses, with $x = \cosh\ell_1 \in [1.24,1.5]$ and $y =
\cosh(\ell_2)\geq Y(x)$, for $Y(x)$ defined in the Proposition, then 
$H\ge\cosh^{-1}(Y(x))/2$ and therefore
$\mathrm{vol}(N) \geq V(x,\cosh^{-1}(Y(x))/2)$. 

We claim that for any $x\in [1.24,1.5]$, the values of $V(x,\cosh^{-1}(Y(x))/2)$ on any interval of $x$ where $Y(x)$ is constant are minimized at the left endpoint of that interval. This follows from a computation. The second equality below uses  the conclusion of Lemma \ref{Mufell1 volume change}.
\[ \frac{\partial V}{\partial x}(x,H) = \mathit{VM}'(x) - \pi[\cosh R]'(2H+\sinh(2H)) =  \pi[\cosh R]' \left(2A -2H-\sinh(2H)\right). \]
Since $\cosh R$ decreases with $x$, this quantity is positive as long as $2A < 2H+\sinh(2H)$. For $x \le 1.5$ we have $2A < \cosh^{-1}(2.34) < 1.5$ and for $H \ge \cosh^{-1}(1.5)/2$,
\[ 2H + \sinh(2H) 
\ge
 \cosh^{-1}(1.5) + \sqrt{1.5^2-1} > 2.08.
 \]
Therefore since $A = A(x)$ is increasing and $Y(x)$ is nonincreasing on the interval $[1.24,1.5]$, $\partial V/\partial x(x,Y(x)) > 0$ for all $x$ here, and the claim follows.

\begin{table}[ht] \begin{center} \begin{tabular}{lllll}
  $x = \cosh \ell_1$\ \  & $Y(x)$\ \  & $V(x,H)$  \\ \hline
  $1.24$ & $1.986$ & $7.406$  \\
  $1.25$ & $1.9$ & $7.406$ \\
  $1.27$ & $1.8$ & $7.438$ \\
  $1.3$ & $1.68$ & $7.407$ \\
  $1.35$ & $1.59$ & $7.407$ \\
  $1.4$ & $1.55$ & $7.433$ \\
  $1.45$ & $1.52$ & $7.431$
\end{tabular} \end{center}
\caption{Values of $V(x,H)$, truncated after $3$ decimal places, for $H = \cosh^{-1}(Y(x))/2$.} 
\label{Vh values}
\end{table}

Table \ref{Vh values} collects the values $V(x,H)$, truncated after
the third decimal place, for $H = \cosh^{-1}(Y(x))/2$ at left
endpoints of the intervals where $Y(x)$ is constant.  Since these
values are all 
greater than 
$7.4$, by the claim and the observations above we have proven the Proposition's assertions when $x = \cosh\ell_1\in [1.24,1.5)$.

It remains to address the case $x\ge 1.5$. Here we recall that by definition, $\ell_2 \geq \ell_1$, and since $\ell_1/2 < A$, it follows from (\ref{KM vol ineq}) that $\mathrm{vol}(N) \ge V(x,\ell_1/2)$. This is exactly the function called $V$ in the proof of \cite[Proposition 3.7]{DeSh} where it is shown to be increasing for $\cosh \ell_1 \geq 1.439$. As its value at $\cosh \ell_1 = 1.5$ is $7.429$, truncated after three decimal places, we have that $\mathrm{vol}(N) \ge 7.4$ whenever $\cosh\ell_1\ge 1.5$.\end{proof}

We lack a volume bound for a manifold $N$ with totally geodesic boundary of genus $2$ if $x = \cosh\ell_1 < 1.24$ or if $x\in[1.24,1.5)$ and 
$\max\{\cosh E,\cosh M,\cosh \ell_1\} \le \cosh\ell_2 \le Y(x)$, 
where $Y$ is defined as above and $E$ and $M$ are defined as in 
Proposition \ref{KM ell_2}. 
The graphs of
$\cosh E$, $\cosh M$, $\cosh\ell_1$ and $Y(x)$ as functions of
$x=\cosh\ell_1$ are 
pictured in Figure \ref{graphs}.

\begin{figure}[ht]
\begin{tikzpicture}[xscale=20, yscale=5]
	\draw [help lines, xstep=0.05, ystep=0.1] (1.22,1.31) grid (1.495,1.89);
	\draw [thick] plot file {l2Ebound.txt};
	\node at (1.22,1.64) {$\cosh E$};
	\draw [thick] plot file {l2Mbound.txt};
	\node at (1.24,1.46) {$\cosh M$};
	\draw [thick, domain=1.33:1.5] plot (\x,\x);
	\node [left] at (1.335,1.35) {\small $\cosh\ell_1$};
	
	\draw [thick] (1.25,1.9) -- (1.27,1.9);
	\draw [thick] (1.27,1.8) -- (1.3,1.8);
	\draw [thick] (1.3,1.68) -- (1.35,1.68);
	\draw [thick] (1.35,1.59) -- (1.4,1.59);
	\draw [thick] (1.4,1.55) -- (1.45,1.55);
	\draw [thick] (1.45,1.52) -- (1.5,1.52);

	\draw[->] (1.21,1.28) -- (1.5,1.28) node [right] {$\cosh\ell_1$};
	\foreach \x in {1.25,1.3,1.35,1.4,1.45}
		\draw (\x,1.27) -- (\x,1.28) node [below] {$\x$};
	\draw[->] (1.18,1.35) -- (1.18,1.85) node [above left] {$\cosh\ell_2$};
	\foreach \y in {1.4,1.5,1.6,1.7,1.8}
		\draw (1.177,\y) -- (1.18,\y) node [left] {$\y$};
\end{tikzpicture}

\caption{$Y(x)$ versus $\max\{\cosh E,\cosh M,\cosh \ell_1\}$}
\label{graphs}
\end{figure}
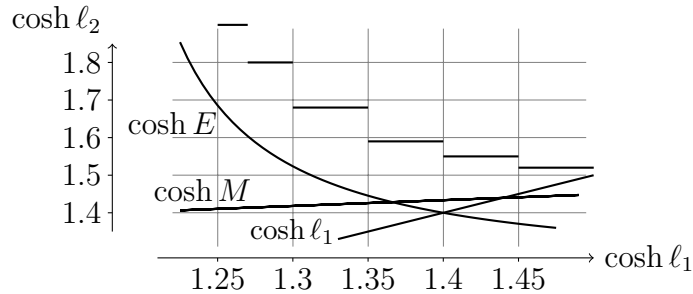

We address these regions with double-muffin volume bounds. First we record bounds on the volume of $M(\ell_2,S)$ over a rectangle of possible $\ell_1$- and $\ell_2$-values.

\begin{lemma}\label{second muffin volume} Regard $R$ as in (\ref{R}) and $S = X_{12}^1 - R$ as in Lemma \ref{disk radii}, for $X_{ij}^k$ as in (\ref{Xijk}), all as functions of $x = \cosh\ell_1$ and $y = \cosh\ell_2$. If $\ell_1$ and $\ell_2$ satisfy
\[ a\leq \cosh\ell_1 \leq b\le 2\quad \mbox{and}\quad c\le \cosh\ell_2 \le d \le 3, \]
then\begin{align}\label{VM two nought}
	\mathrm{vol}(M(\ell_2,S)) \ge \mathit{VM}_2^0 \doteq 2\pi(A^0\cosh S(a,d) - \cosh^{-1}(d)/2),   \end{align}
where $A^0 = \tanh^{-1} \left(\cosh\left( S(a,d)\right)\sqrt{\frac{c-1}{c+1}}\right)$.
\end{lemma}

\begin{proof} $M(\ell_2, S)$ has side altitude of length $A = \tanh^{-1}(\cosh S\tanh(\ell_2/2))$ by the formula (\ref{muffin relations}). By Lemma \ref{S derivs}, as a function $x$ and $y$, $S$ increases with $x$ and decreases with $y$ on the region in question. Hence it is bounded below by $S(a,d)$ there. It follows that $A$ is bounded below by $A^0$ as defined above on this region. 

The volume of $M(\ell_2,S)$ is 
given by the formula stated on p.~213 of \cite{Fenchel}, which we have
already quoted in connection with (\ref{Muf_l1}). In the present
context, the formula gives
\begin{align}\label{muffin volume}
  \mathrm{vol}(M(\ell_2,S)) = 2\pi(A\cosh S - \ell_2/2),  \end{align}

By combining (\ref{muffin volume}) with  the lower bounds on $A$ and
$S$ established  above, we now obtain
 the lower bound on $\mathrm{vol}(M(\ell_2,S))$ given in the Lemma's statement. 
\end{proof}

We now combine the volume of $M(\ell_2,S)$ with existing ingredients to give the main double-muffin volume bound:

\begin{proposition}\label{DblMfn} Suppose $N$ is an orientable hyperbolic $3$-manifold with $\partial N$ compact, connected, totally geodesic and of genus $2$, such that the lengths $\ell_1$ and $\ell_2$ of its shortest and second-shortest return paths satisfy
\[  a\leq \cosh\ell_1 \leq b\le 2\quad \mbox{and}\quad c\le \cosh\ell_2 \le d \le 3. \]
For $W_R^0$, $W^0_S$, $T_{12}^0$ and $T_{22}^0$ as in (\ref{waisting away}), if $W^0_R+W^0_S\le T^0_{12}$ and $2W_S^0\le T^0_{22}$, then $\mathrm{vol}(N) \ge \mathit{VM}(b) + \mathit{VM}_2^0$, where $\mathit{VM}(x)$ is as in Lemma \ref{Mufell1 volume change} and $\mathit{VM}_2^0$ is defined in (\ref{VM two nought}). 

Moreover, for $H^0$ as defined in Proposition \ref{Et tu?} and $A(x)$ as in Lemma \ref{Mufell1 volume change}, taking $H = \min\{A(a),\cosh^{-1}(c)/2, H^0\}$ we have:\begin{align}
\label{plus collar} \mathrm{vol}(N) \ge \mathit{VM}(b) + \mathit{VM}_2^0 + \pi(3-\cosh R(a)-\cosh S(b,c))\left(2H + \sinh(2H)\right), \end{align}
where $R = R(x)$ from (\ref{R}) and $S = X_{12}^1 - R$ from Lemma \ref{disk radii}.
\end{proposition}

\begin{proof} Given the inequalities relating $W_R^0$, $W^0_S$, $T_{12}^0$ and $T_{22}^0$, Proposition \ref{outside Muf_ell1} implies that the the universal cover $\widetilde{N}\to N$ embeds the interior of a copy of $M(\ell_2,S)$, centered at a lift of the second-shortest return path $\lambda_2$, disjointly from the interior of the copy $\mathrm{Muf}_{\ell_1}$ centered at $\lambda_1$. Therefore the volume of $N$ is at least the sum of these two muffins' volumes.

Recall from Lemma \ref{Mufell1 volume change} that the function $\mathit{VM}(x)$, which records the volume of $\mathrm{Muf}_{\ell_1}$ as a function of $x = \cosh\ell_1$, is decreasing. Thus since $x\le b$, $\mathit{VM}(x) \ge \mathit{VM}(b)$. By Lemma \ref{second muffin volume}, $\mathit{VM}_2^0$ bounds the volume of $M(\ell_2,S)$ below for these values of $\ell_1$ and $\ell_2$. Thus $\mathrm{vol}(N) \ge \mathit{VM}(b) + \mathit{VM}_2^0$.

We can augment this lower bound by adding the volume of a collar of the region in $\partial N$ outside the union of the caps of the embedded copies of $\mathrm{Muf}_{\ell_1}$ and $M(\ell_2,S)$, as long as the collar is not too high. If $B$ is the area of the region in $\partial N$ and $H$ is the collar height, then the collar volume $V$ satisfies $V = B \cdot (2H+\sinh(2H))/4$.  As $\partial N$ has area $4\pi$, by the Gauss-Bonnet theorem, and a hyperbolic disk of radius $r$ has area $2\pi(\cosh r - 1)$, we have:
\[ B = 4\pi - 4\pi(\cosh R -1) - 4\pi(\cosh S -1) = 4\pi \left(3 - \cosh R - \cosh S\right). \]
A collar height of $H = \min\{A,\ell_2/2\}$ was used in \cite{KM} and
\cite{DeSh}, where $A = A(x)$ 
is defined by 
$\cosh(2A) = (4x+1)/3$ as in
Lemma \ref{Mufell1 volume change}. 
For a copy of $\mathrm{Muf}_{\ell_1}$ in $N$ centered at $\lambda_1$, 
with caps $U$ and $U'$,
Proposition \ref{embark} asserts that a height-$H$ 
\redcollar\ of $\partial N- (U\cup U')$ 
is is embedded in
 $N$ disjointly from $\mathrm{Muf}_{\ell_1}$.
 Since $A$ is an increasing
 function of $x$, the quantity
$\min\{A(a),\cosh^{-1}(c)/2\}$ bounds $\min\{A,\ell_2/2\}$ below on the entire rectangle.

For $H^0$ as defined in Proposition \ref{Et tu?}, that result implies 
that 
a height-$H^0$ collar of $\partial N$ intersects the embedded copy of
$M(\ell_2,S)$ in the union of \redcollar s of its caps $V$ and $V'$ 
(noting that the bounds on $\ell_1$ and $\ell_2$ 
 in the hypothesis of the present proposition match 
 those in the hypothesis of
Proposition \ref{Et tu?}). 
Therefore a collar of $\partial N - (U\cup U'\cup V\cup V')$ with height $H = \min\{A(a),\cosh^{-1}(c)/2, H^0\}$ is embedded in $N$ without overlapping either the copy of $\mathrm{Muf}_{\ell_1}$ or of $M(\ell_2,S)$, so it contributes volume
\[ V = \pi(3-\cosh R - \cosh S)(2H + \sinh(2H)) \]
to the volume of $N$, independently of the muffins.

In order to bound $V$ below in terms of the given bounds on $x = \cosh\ell_1$ and $y = \cosh\ell_2$, we bound $R$ and $S$ above. Since $R$ decreases with $x$, its value is bounded above by $R(a)$. And by Lemma \ref{S derivs}, the value of $S$ is bounded above by $S(b,c)$. Together with the height bound $H$ described above, this gives the collar volume's contribution to the Proposition's lower bound on $\mathrm{vol}(N)$.
\end{proof}

We use Proposition \ref{DblMfn} and a numerical scheme, whose output is summarized in the picture below, to complete the proof of the first main result of this section. 

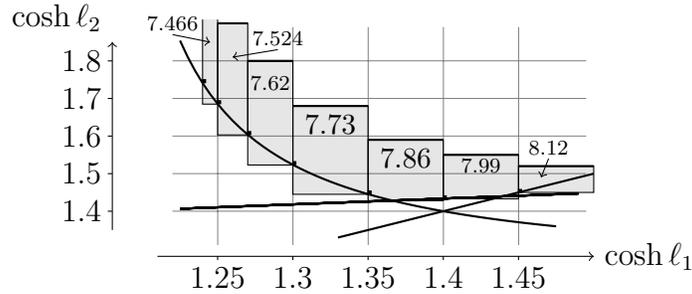
\begin{figure}[ht]
\begin{tikzpicture}[xscale=20, yscale=5]
	\draw [help lines, xstep=0.05, ystep=0.1] (1.22,1.31) grid (1.495,1.89);
	\draw [thick] plot file {l2Ebound.txt};
	\draw [thick] plot file {l2Mbound.txt};
	\draw [thick, domain=1.33:1.5] plot (\x,\x);
	
	\draw [thick] (1.25,1.9) -- (1.27,1.9);
	\draw [thick] (1.27,1.8) -- (1.3,1.8);
	\draw [thick] (1.3,1.68) -- (1.35,1.68);
	\draw [thick] (1.35,1.59) -- (1.4,1.59);
	\draw [thick] (1.4,1.55) -- (1.45,1.55);
	\draw [thick] (1.45,1.52) -- (1.5,1.52);

	\draw[->] (1.21,1.28) -- (1.5,1.28) node [right] {$\cosh\ell_1$};
	\foreach \x in {1.25,1.3,1.35,1.4,1.45}
		\draw (\x,1.27) -- (\x,1.28) node [below] {$\x$};
	\draw[->] (1.18,1.35) -- (1.18,1.85) node [above left] {$\cosh\ell_2$};
	\foreach \y in {1.4,1.5,1.6,1.7,1.8}
		\draw (1.177,\y) -- (1.18,\y) node [left] {$\y$};
		
	\fill [opacity=0.1] (1.24,1.685) -- (1.25,1.685) -- (1.25,1.91) -- (1.24,1.91) -- cycle;
	\draw [thin] (1.24,1.91) -- (1.24,1.685) -- (1.25,1.685) -- (1.25,1.91);
	\filldraw (1.24,1.742) rectangle ++(.002,.008);
	\node at (1.22,1.9) {\tiny $7.466$};
	\draw [->] (1.22,1.87) -- (1.245,1.85);
	
	\fill [opacity=0.1] (1.25,1.603) -- (1.27,1.603) -- (1.27,1.9) -- (1.25,1.9) -- cycle;
	\draw [thin] (1.25,1.603) -- (1.27,1.603) -- (1.27,1.9) -- (1.25,1.9) -- cycle;
	\node at (1.29,1.86) {\tiny $7.524$};
	\filldraw (1.25,1.686) rectangle ++(.002,.008);
	\draw [->] (1.29,1.83) -- (1.26,1.81);
	
	\fill [opacity=0.1] (1.27,1.523) -- (1.3,1.523) -- (1.3,1.8) -- (1.27,1.8) -- cycle;
	\draw [thin] (1.27,1.523) -- (1.3,1.523) -- (1.3,1.8) -- (1.27,1.8) -- cycle;
	\node at (1.285,1.74) {\tiny $7.62$};
	\filldraw (1.27,1.604) rectangle ++(.002,.008);

	\fill [opacity=0.1] (1.3,1.445) -- (1.35,1.445) -- (1.35,1.68) -- (1.3,1.68) -- cycle;
	\draw [thin] (1.3,1.445) -- (1.35,1.445) -- (1.35,1.68) -- (1.3,1.68) -- cycle;
	\node at (1.325,1.63) {\small $7.73$};
	\filldraw (1.3,1.524) rectangle ++(.002,.008);

	\fill [opacity=0.1] (1.35,1.428) -- (1.4,1.428) -- (1.4,1.59) -- (1.35,1.59) -- cycle;
	\draw [thin] (1.35,1.428) -- (1.4,1.428) -- (1.4,1.59) -- (1.35,1.59) -- cycle;
	\node at (1.375,1.54) {\small $7.86$};
	\filldraw (1.35,1.446) rectangle ++(.002,.008);

	\fill [opacity=0.1] (1.4,1.433) -- (1.45,1.433) -- (1.45,1.55) -- (1.4,1.55) -- cycle;
	\draw [thin] (1.4,1.433) -- (1.45,1.433) -- (1.45,1.55) -- (1.4,1.55) -- cycle;
	\node at (1.425,1.522) {\tiny $7.99$};
	\filldraw (1.4,1.433) rectangle ++(.002,.008);

	\fill [opacity=0.1] (1.45,1.45) -- (1.5,1.45) -- (1.5,1.52) -- (1.45,1.52) -- cycle;
	\draw [thin] (1.45,1.45) -- (1.5,1.45) -- (1.5,1.52) -- (1.45,1.52) -- cycle;
	\node at (1.47,1.57) {\tiny $8.12$};
	\filldraw (1.45,1.45) rectangle ++(.002,.008);
	\draw [->] (1.47,1.54) -- (1.465,1.49);
		
\end{tikzpicture}
\caption{Some volume bounds}
\label{volume bounds}
\end{figure}

\begin{theorem}\label{one two three five} Let $N$ be an orientable hyperbolic $3$-manifold with $\partial N$ compact, connected, totally geodesic, and of genus $2$. If the length $\ell_1$ of the shortest return path of $N$ satisfies $\cosh\ell_1 \ge 1.24$ then $N$ has volume at least $7.4$.\end{theorem}

\begin{proof} Given Proposition \ref{SglMfn}, to prove the Theorem we must bound the volume of a manifold $N$ whose shortest return path length $\ell_1$ satisfies $1.24\le \cosh \ell_1 \le 1.5$, and whose second-shortest return path length $\ell_2$ satisfies $\cosh\ell_2\le Y(x)$, for the piecewise-constant function $Y$ of the Proposition. For this we implement the volume bounds of Proposition \ref{DblMfn} numerically as follows.

On each maximal subinterval $[a_0,b_0)$ of $[1.24,1.5)$ where $Y$ is constant, we find the minimum value of the function $\max\{\ell_1,E,M\}$ from Lemma \ref{KM ell_2}. Let $c_0$ be the hyperbolic cosine of this value, truncated after three decimal places. By Lemma \ref{KM ell_2}, for $N$ with $\cosh\ell_1\in [a_0,b_0]$, $\cosh\ell_2 \ge c_0$. Let $d_0$ be the value of $Y$ on $[a_0,b_0)$. We have written Python scripts that divide the rectangle $[a_0,b_0]\times[c_0,d_0]$ into subrectangles of size $0.001\times0.001$, and on each subrectangle $[a,b]\times[c,d]$ performs the following tasks.\begin{enumerate}
	\item Check whether
$[a,b]\times[c,d]$ intersects the \textit{possible subregion}: the subset of $[a_0,b_0]\times[c_0,d_0]$ consisting of $(x,y)$-values that satisfy the criterion of Lemma \ref{KM ell_2}. For instance, $E(x) = \max\{\ell_1,E,M\}$ for $x\leq 1.366$, so since $E$ is a decreasing function of $x$ here (cf.~\cite[Lemma 3.4]{DeSh}), for $b\leq 1.366$ this is tantamount to checking that $d \ge \cosh E(b)$.

If $[a,b]\times[c,d]$ does not intersect the possible subregion then
we ignore it since in this case, by Lemma \ref{KM ell_2}, there does
not exist a manifold satisfying the hypotheses of the 
present 
result and with an $(x,y)$-value in $[a,b]\times[c,d]$.

	\item If $[a,b]\times[c,d]$ does intersect the possible subregion, check the muffin-embedding criteria of Proposition \ref{outside Muf_ell1}. These hold in each such case, confirming the expectation recorded above that result.
	\item On $[a,b]\times[c,d]$ intersecting the possible subregion, compute the two volume bounds supplied by Proposition \ref{DblMfn}, recording the larger. (The reason that the sum of muffin volumes $\mathit{VM}(b)+\mathit{VM}^0_2$  may be larger than the bound (\ref{plus collar}), which also incorporates the volume of a collar of a region in $\partial N$, is that in some cases the given lower bound $3-\cosh R(a) - \cosh S(b,c)$ for the area of that region may be negative.)
\end{enumerate}

The Python scripts performing this task are included in the ancillary files as \texttt{VolScript\_E.py}, \texttt{VolScript\_M.py}, and \texttt{VolScript\_x.py}, where the different names refer to the different lower bounds that the scripts call for $\max\{\ell_1,E,M\}$ (each choice being best for a particular range of $x$ values). They call functions from the script \texttt{formulas.py}, which collects relevant formulas from elsewhere in this paper. Their output is summarized in text files, also included in the ancillary materials, such as \texttt{124\_125.txt}: this one recording bounds for possible $(x,y)$-values with $x\in[1.24,1.25]$.

The resulting volume bounds are 
summarized in Figure \ref{volume bounds}. Each big rectangle $[a_0,b_0]\times [c_0,d_0]$ from above is shaded in the Figure, and the smallest of the volume bounds produced by the process above---taken over all subrectangles $[a,b]\times[c,d]$---is recorded in or directly above the big rectangle. (The subrectangle on which the minimum is attained is also blackened in the Figure.) Since each of these bounds is larger than $7.4$, the Theorem follows from Proposition \ref{SglMfn}.
\end{proof}

The next theorem provides a  volume
bound for manifolds satisfying  an extra condition.

\begin{theorem}\label{one two three} Let $N$ be an orientable hyperbolic $3$-manifold with $\partial N$ compact, connected, totally geodesic, and of genus $2$. If the length $\ell_1$ of the shortest return path of $N$ satisfies 
$\cosh\ell_1 \le 1.24$ and $\widetilde{N}$ contains no $(1,1,1)$-hexagon then $N$ has volume at least $7.409$.\end{theorem}

\begin{proof} 
Since $\widetilde{N}$ contains no $(1,1,1)$-hexagon, 
Proposition \ref{DeSh ell_2} asserts that $x \doteq \cosh\ell_1 > 1.23$.
The proof 
will follow 
the same scheme
as that of Theorem \ref{one two three five}, but 
we will use the interval $[1.23,1.24]$ of possible values of $x$, and will
appeal 
to Proposition \ref{DeSh ell_2} again to 
get 
a non-constant upper bound $Y_0(x)$ on $\cosh\ell_2$ in terms of $x$. 
From the formula of Proposition \ref{DeSh ell_2} we obtain $Y_0(x) = (2+2\sqrt{3})x^2 - (3+2\sqrt{3})$. 

We use a single large rectangle $[1.23,1.24]\times [1.74,1.938]$. The lower bound of $y = 1.74$ is the minimum of $\cosh E(x)$ on this interval, which is attained at $x=1.24$, truncated after three decimal places. The upper bound is the maximum of $Y_0(x)$ on the interval, also attained at $1.24$, rounded up after three decimal places. Because $Y_0(x)$ is not constant here but increasing, we alter step (1) above for a subrectangle $[a,b]\times [c,d]$ by checking that both $d \ge \cosh E(b)$ and $c \le Y_0(x)$. 

To achieve a slight improvement in our volume bounds, we use a subrectangle size of $0.0005 \times 0.001$---that is, we divide the previous subrectangles in half vertically. The smallest bound resulting from Theorem \ref{one two three five}'s revised process is $7.409$, truncated after three decimal places, attained on the subrectangle $[1.2305,1.231]\times [1.808,1.809]$, proving the result. (Using the previous subrectangle size would yield $7.399$ as a bound.) The relevant Python script is \texttt{VolScript\_EUl2.py}, and its output is \texttt{123\_124.txt}.
\end{proof}

Theorems \ref{one two three five} and \ref{one two three} together yield an unconditional lower bound on volume in the absence of a $(1,1,1)$-hexagon.

\begin{corollary}\label{larry}\CoreOLarry\end{corollary}

\begin{proof} For a manifold $N$ satisfying the
  hypotheses of the corollary, 
if the length $\ell_1$ of the shortest return path of $N$ satisfies $\cosh\ell_1\le 1.24$ then $\mathrm{vol}(N) \ge 7.409$ by Theorem \ref{one two three}. If $\cosh\ell_1\ge 1.24$ then $\mathrm{vol}(N)\ge 7.4$ by Theorem \ref{one two three five}. Thus $\mathrm{vol}(N)$ is at least $7.4$ in both cases.\end{proof}

\section{Trimonic Manifolds from a $(1,1,1)$-hexagon}\label{trimonic}

Here we strengthen some results from Section 6 of \cite{DeSh}. Our first result removes a restriction on the length of $\ell_1$ from the hypothesis of Lemma 6.6 there.

\begin{lemma}\label{short cut vs 111 hex}  Let $N$ be a hyperbolic $3$-manifold with compact totally geodesic boundary. If $C$ is a $(1,1,1)$-hexagon in $\widetilde{N}$ and $\tilde\lambda_1$ is a lift of the shortest return path of $N$, then $\tilde\lambda_1$ is an internal edge of $C$ or $\tilde\lambda_1\cap C = \emptyset$.\end{lemma}

We recall from Definition \ref{internal external} here that the three \textit{internal} edges of a $(1,1,1)$-hexagon are short cuts, \ie lifts of the shortest return path, and its other, \textit{external} edges lie in components of $\partial\widetilde{N}$. This matches the usage in \cite{DeSh}.

\begin{proof} We follow the proof of Lemma 6.6 of \cite{DeSh}: let
  $\Pi_1$ and $\Pi_2$ be the components of $\partial\widetilde{N}$
  containing the endpoints of $\tilde\lambda_1$ and $\Pi$ the geodesic
  plane containing $C$. As 
shown 
in the proof of \cite[L.~6.6]{DeSh}'s first paragraph, if $\tilde\lambda_1$ lies in $\Pi$ then it is an internal edge of $C$. (This argument does not require a condition on the length of $\tilde\lambda_1$.) We therefore suppose below that $\tilde\lambda_1$ does not lie in $\Pi$ and hence intersects it transversely in a point $x$.

As in the previous proof, we take $\Pi'$ to be the component of $\partial\widetilde{N}$ containing the external edge of $C$ closest to $x$. As observed there, the distance from $\Pi'$ to $x$ is at most $A$, defined in equation (3.1.1) of \cite{DeSh} by $\cosh A = \sqrt{\frac{2}{3}(\cosh\ell_1+1)}$ (the same ``$A$'' from Lemma \ref{Mufell1 volume change} here). This thus bounds the distance from $\Pi'$ to $\tilde\lambda_1$ above.

On the other hand, we repeat for emphasis that $\Pi'\ne \Pi_i$ for $i = 1$ or $2$ (this was also used in \cite[L.~6.6]{DeSh}), since $\tilde\lambda_1$ intersects $C$ transversely but intersects $\Pi_1$ and $\Pi_2$ at right angles. Therefore the distance from $\Pi'$ to $\tilde\lambda_1$ is bounded \textit{below} by the quantity $U_1$ from Lemma \ref{You 'kay?} of this paper, which is given strictly in terms of $\cosh\ell_1$ by:
\[ \cosh U_1 = \sqrt{\frac{2\cosh^2\ell_1}{\cosh\ell_1 -1} + 1} = \sqrt{\frac{2\cosh\ell_1-1}{\cosh\ell_1-1}(\cosh\ell_1 + 1)} \]
A little manipulation then shows that $\cosh U_1 > \sqrt{2(\cosh\ell_1+1)}$. It follows that $U_1 > A$, yielding a contradiction. Thus $\tilde\lambda_1$ can only intersect $C$ as one of its internal edges.\end{proof}

Lemma \ref{short cut vs 111 hex} 
in turn facilitates an analogous strengthening of \cite[Lemma 6.7]{DeSh}:

\begin{lemma}\label{two 111s} Let $N$ be a hyperbolic $3$-manifold with compact totally geodesic boundary. If $C$ and $C'$ are distinct $(1,1,1)$-hexagons in $\widetilde{N}$, then $C\cap C'$ is empty or a single internal edge of each.
\end{lemma}

\begin{proof} The proof of \cite[Lemma 6.7]{DeSh} still 
goes through 
here 
(i.e.~without the 
prior result's 
hypothesis that $\cosh\ell_1\le 1.215$): 
one need only replace the reference to
\cite[Lemma 6.6]{DeSh} in the proof of \cite[Lemma 6.7]{DeSh} by a reference to Lemma \ref{short cut vs 111 hex} above.
\end{proof}

The following result strengthens both Propositions 6.8 and 6.9  of
  \cite{DeSh} by considerably weakening the restriction on $\ell_1$. 
The 
term
\textit{trimonic manifold,} which appears in the following statement, is
defined in   \cite[Definition 5.7]{DeSh}, while the term
\textit{non-degenerate} is defined in  \cite[Definition 5.10]{DeSh}.

\begin{proposition}\label{heeere's trimonny!} Let $N$ be an orientable hyperbolic $3$-manifold with $\partial N$ compact, connected, totally geodesic, and of genus $2$, such that there is a $(1,1,1)$-hexagon in $\widetilde{N}$ and 
\[  \cosh \ell_1 < \frac{\cos (2\pi/9)}{2\cos (2\pi/9) - 1} = 1.43969... \]
Then there is a submanifold $X\subset N$ with $\partial N\subset X$, such that $X$ is a non-degenerate trimonic manifold relative to $\partial N$.\end{proposition}

\begin{proof} The trimonic manifold $X$ is constructed in the proof of
  \cite[Prop.~6.8]{DeSh} as a regular neighborhood of 
  the union of $\partial N$ with
  the 
  image in
  $N$ of a $(1,1,1)$-hexagon $C$ in $\widetilde{N}$ under the
  universal 
covering projection $\widetilde{N}\to N$.
  Following that proof, we
  denote by $f$
the restriction of the covering projection 
to $C$.
The upper bound on $\cosh\ell_1$ ensures that $N$ has a unique
shortest return path $\lambda_1$, by \cite[Lemma 6.3]{DeSh}, and hence
that $f$ projects every internal edge of $C$ to $\lambda_1$.

In order to show that $X$ is a trimonic manifold relative to
$\partial N$, we must verify that it has the properties (1)--(6)
stated in \cite[Dfn.~5.7]{DeSh}. 
Property (1), which in  the proof of
  \cite[Prop.~6.8]{DeSh} was verified by an appeal to Lemma 6.6 of \cite{DeSh}, follows here from Lemma \ref{short cut vs 111 hex}
above. 
Lemma 6.7 of \cite{DeSh}, which in the
earlier argument was used to establish Property (3), 
is replaced here by Lemma \ref{two 111s}. 
The verifications of Properties (2), (4), (5) and (6) go through
exactly as in the earlier argument. Furthermore, 
the argument for non-degeneracy given in Proposition 6.9 of \cite{DeSh} does not depend on any particular upper bound for $\cosh\ell_1$ and so carries through without alteration.
\end{proof}

\section{
Background on topology and least-area surfaces
}\label{bop tack}

\numberwithin{para}{subsection}

The proof of Theorem \ref{bounding main}, which was stated in the
Introduction,  combines the results of the preceding sections with 
concepts and results from three-manifold topology
and geometry 
 which will
also be important
  in later sections.
  We introduce these here. 
Subsection \ref{top back} establishes topological 
conventions, definitions and observations 
that 
will be used
throughout the rest of the paper. 
Subsection \ref{geom back} reviews  
a series of deep differential geometric results on minimal
surfaces in three-manifolds proved by other authors, 
which were used in the proof of \cite[Th.~7.4]{DeSh}, 
and adapts them
to
the present context. 
Their first application here is in the proof of Theorem \ref{7.4 upgrade}, 
which is itself an ingredient of the proof of Theorem \ref{bounding main}.

\subsection{Topological background}\label{top back}~
The material concerning manifolds in this subsection, and
elsewhere in the paper, is 
to be understood in the
smooth category. However, we shall often implicitly exploit the equivalence of
the smooth and piecewise-linear categories in dimension $3$ to go back
and forth between the two categories.

We stipulate, 
as part of the definition of connectedness, 
that a connected topological space is non-empty.

The Euler characteristic of a finitely triangulable space $Y$ will be denoted $\chi(Y)$, and we will set $\chibar(Y)=-\chi(Y)$.

If $A$ is a subset of a topological space $X$, we denote the frontier
of $A$ in $X$, defined to be $\overline{A}\cap\overline{X-A}$, by $\Fr_XA$.

We shall follow the conventions of \cite{Schar-Thomp} regarding
Heegaard splittings. In particular, each compact, connected, orientable
$3$-manifold-with-boundary $M$ has a well-defined Heegaard genus, which 
we denote by $\Hg(M)$.

A connected submanifold $Y$ of a connected manifold $X$ is said to be {\it
  $\pi_1$-injective} in $X$ if the inclusion homomorphism
$\pi_1(Y)\to\pi_1(X)$ is injective. More generally, a (possibly empty) submanifold $Y$ of a  manifold $X$ is said to be {\it
  $\pi_1$-injective} in $X$ if each component $C$ of $Y$ is
  $\pi_1$-injective in the component of $X$ containing $C$.

A $3$-manifold $M$ is said to be {\it irreducible} if 
$M$ is connected and
every (tame) $2$-sphere 
in $M$ bounds a ball. 
One says that 
a two-dimensional submanifold  
$\cals$ of 
an irreducible 
$3$-manifold $M$
is
 {\it incompressible} if (1)
$S$ is 
closed and orientable, and is 
contained in $\inter M$, (2) no
component of $S$ is a $2$-sphere, and (3) $\cals$ is $\pi_1$-injective.
The $3$-manifold $M$ is said to be \textit{boundary-irreducible} if $\partial M$ 
is 
$\pi_1$-injective, and \textit{boundary-reducible} otherwise.

In this paper, we say
that a connected $3$-manifold is {\it strongly atoral} if its fundamental group has no rank-$2$ free abelian subgroup. 
We shall say that a $3$-manifold is {\it \simple} if it
is compact, orientable, irreducible, boundary-irreducible, strongly
atoral,  and 
has an infinite fundamental group. Note that a simple $3$-manifold $M$ is
not homeomorphic to a ball; since $M$ is irreducible, it follows that no boundary
component of $M$ 
can be a $2$-sphere. 

Note also that the definition of ``\simple'' 
given 
here is similar but not
identical to 
the  definition of the same term in \cite{CDS}. It is easy to see that a $3$-manifold is \simple\ in the
sense defined in \cite{CDS} if and only if it is either \simple\ in
the sense defined here or is homeomorphic to a ball.

Note that 
every closed, orientable, hyperbolic $3$-manifold is 
\simple, 
and 
that 
in a
\simple\
$3$-manifold, every 
connected
incompressible surface
has genus at least $2$.

We 
now 
review some definitions from \cite{CDS}. 
If $S$ is a
closed
 surface
 in 
the interior of 
a  $3$-manifold $M$, 
we denote by
 $M\setminus\setminus S$ the manifold with boundary obtained by
 splitting $M$ along $S$: 
 it is the completion of the path metric on 
$Z\doteq M-S$ 
induced by the 
 restriction of a metric on $M$. The inclusion map
 $Z\hookrightarrow M$ 
extends to a map $M\cut S\to M$ 
 that restricts on 
$(M\cut S) - Z$ 
to a two-sheeted covering 
 map to $S$. The restriction is a disconnected cover if and only if 
 $S$ is two-sided in $M$.

If $M$ is \simple, and $S$ is incompressible and non-empty,
then each component  of
 $M\setminus\setminus S$ is \simple\
and has non-empty boundary.

Any 
\simple\ $3$-manifold $U$ with non-empty boundary 
has a well-defined relative characteristic submanifold $\Sigma_U$ in
the sense of \cite{Jo} and \cite{JS}. (In the notation of \cite{JS},
$(\Sigma_U,\Sigma_U\cap\partial U)$ is the characteristic pair of
$(U,\partial U)$. 
The assumptions that $U$ is \simple\ and has non-empty
boundary are enough to guarantee that $(U,\partial U)$ is a
``sufficiently large pair'' in the sense of \cite{JS}, so that the
Characteristic Pair Theorem \cite[p. 118]{JS} guarantees that the
characteristic pair is well defined. 
The arguments of \cite{JS}, and the corresponding arguments in
\cite{Jo}, are done in the piecewise-linear category; as we are
translating the results to the smooth category, we must regard
$\Sigma_U$ as a \textit{smooth manifold with corners} 
---we refer to \cite[Ch.~16]{LeeSmoothMflds} for definitions and basic facts---
such that $\Sigma_U\cap\partial
U$ and $\Fr_U\Sigma_U$ are smooth manifolds with boundary.) 
For each 
component $C$ of $\Sigma_Q$ either (i) $C$ may be given the structure of an
$I$-bundle over a compact 
(smooth) 
$2$-manifold-with-boundary $F_C$ with
$\chi(F_C)<0$, in such a way
that $\Fr_QC$ is the  preimage of
$\partial F_C$ under the bundle projection, or (ii)
$C$ is 
homeomorphic to a 
solid torus 
and the components 
of $\Fr_QC$ are 
(smooth) 
annuli in $\partial C$ that are homotopically non-trivial in
$C$.

Now let $Q$ be a  compact $3$-manifold, each of  whose
components is a \simple\ manifold with non-empty boundary.
We denote 
by $\Sigma_{Q}\subset Q$ the union of the submanifolds $\Sigma_U$,
where $U$ ranges over the components of $Q$. 
Since $\Sigma_Q$ is a manifold with corners whose frontier is a smooth
manifold with boundary, $\overline{Q-\Sigma_Q}$ is a manifold with
corners. 
Each component of 
$\overline{Q-\Sigma_Q}$ 
either has (strictly)
negative Euler characteristic, or may be identified 
(by a diffeomorphism of manifolds with corners) 
with $S^1\times[0,1]\times[0,1]$ in such a way that its frontier in
$Q$ is $S^1\times[0,1]\times\{0,1\}$. 
(To rule out components of $\overline{Q-\Sigma_Q}$  with strictly
positive Euler characteristic, we must show that no boundary component
$W$ of  $\overline{Q-\Sigma_Q}$ is a $2$-sphere. If
$W\cap\Sigma_Q=\emptyset$ 
this 
follows from the observation, made
above, that no boundary component of a \simple\ $3$-manifold $Q$ is a
$2$-sphere. 
If $W\cap\Sigma_Q\ne\emptyset$ then $W$ contains a component of
$\Fr_Q\Sigma_Q$, which is a $\pi_1$-injective annulus in $Q$ and
therefore cannot be contained in a $2$-sphere.) 
We
denote by $\kish(Q)$ the union of all components of
$\overline{Q-\Sigma_Q}$ that have negative Euler
characteristic, and 
set 
$\kish^0(Q) = \kish Q - \Fr_Q\kish Q$.
Thus
$\kish Q$ is a $3$-manifold-with-corners, while $\kish^0Q$ is a smooth
$3$-manifold-with-boundary; and, 
 by definition, for each component $K$ of
$\kish(Q)$, we have $\chi(K)<0$, or equivalently $\chibar(K)\ge1$.

To say that
$Q$ is {\it acylindrical} 
(where 
 $Q$ is still understood to be a compact 
$3$-manifold, each of  whose
components is a \simple\ manifold with non-empty boundary) 
means that $\Sigma_Q=\emptyset$; this is
equivalent to saying that 
$\kish(Q)=Q$.

Note that every 
compact hyperbolic $3$-manifold  
with non-empty totally geodesic boundary is (i) \simple\ and (ii)
acylindrical.

As in \cite{ACS}, we define a {\it book of $I$-bundles}  to
be a triple $\calw=(W,\calB,\calp)$, where (1) $W$ is a compact,
orientable
smooth $3$-manifold with boundary; 
(2) $\calb$ and $\calp$ are 
manifolds with corners such that 
$\calb\cup\calp=W$ and $\calb\cap\calp=\Fr_W\calb=\Fr_W\calp$, and
each of the sets 
$\calb\cap\calp$, 
$\calb\cap\partial W$, and $\calp\cap\partial W$ is a smooth manifold with boundary; 
(3) 
each component 
$B$ of $\calB$ is a solid torus whose frontier
components in $W$ are
(smooth)
annuli in $\partial B$ that are homotopically non-trivial in
$B$; 
and (4) each component $P$ of $\calp$ may be given the structure of an
$I$-bundle over a compact
smooth 
$2$-manifold with boundary $F_P$  in such a way
that $\Fr_QP$ is the  preimage of
$\partial F_P$ under the bundle projection. 
The components of $\calb$
and $\calp$ are called {\it bindings} and {\it pages} of $\calw$,
respectively. The manifold $W$ is called the underlying manifold of
$\calw$ and may be denoted $|\calw|$. We will say that a book of $I$-bundles $\calw$ is
connected if $|\calw|$ is connected.

It follows from the discussion above
  that if  $Q$ is 
a compact
$3$-manifold, each of  whose
components is a \simple\ manifold with non-empty boundary,
and if $\kish(Q)=\emptyset$, then $Q$ is the
underlying manifold of some book of $I$-bundles $\calw$. Indeed, we
may take $\calw=(Q,\calb,\calp)$, where $\calb$ is a regular
neighborhood of the union of all solid torus components of $\Sigma_{Q}$, and
$\calp$ 
is a regular
neighborhood of the union of all components of $\Sigma_{Q}$ having
negative Euler characteristic.


\subsection{
Background on least-area surfaces
}\label{geom back}~

Here we review   
a series of deep differential geometric results on minimal
surfaces in three-manifolds proved by other authors, 
which were used in the proof of \cite[Th.~7.4]{DeSh}. We 
further adapt 
these results here for use in the proof of Theorem \ref{7.4 upgrade}
of the present paper---which is an upgrade of \cite[Th.~7.4]{DeSh}---and 
for 
further applications later in this paper.

\begin{definitionsremarks}\label{refer later}
If $S$ is an oriented smooth manifold and $M$ is a Riemannian manifold, possibly with boundary, 
 with $\dim M\ge\dim S$, any
smooth immersion $f\co S\to M$ 
pulls the Riemannian metric $g$ on $M$ back to a Riemannian metric
$f^* g$ on $S$. For any 
(measurable)
 $A\subset S$, we 
define  the \textit{volume of $A$ under $f$,} denoted by $\vol_f(A)$, to be 
the integral of the volume form of $f^* g$ over $A$.  
This extends
naturally to a 
definition of $\vol_f(A)$ for  any smooth map $f:S\to M$, by taking the volume form to be $0$ at
points where the derivative of $f$ is singular. Thus if $f$ is not an
immersion, we may have $\vol_f(A)=0$ even if (say) $A$ has non-empty interior.

For $\dim S=2$
we will use the term ``area'' in place of ``volume,'' and write
$\area_f(A)$ in place of $\vol_f(A)$ for measurable $A\subset S$. 
We will say that a smooth map $f\co S\to M$ is \textit{least-area} if 
$\area_f(S)\le \area_g(S)$ for any map $g\co S\to M$ smoothly homotopic to $f$.

There is an equivalent definition of the volume of $A\subset S$ under
the smooth map $f$ in the special case where $\dim M=\dim S$ and where  $M$,
as well as $S$, is equipped with an orientation. In this situation, the
metric on $M$ defines a volume form $\alpha$, and the pulled back form
$f^*(\alpha)$ may be written as $h\cdot\omega$, where $\omega$ is a
non-vanishing form that determines the given orientation of $S$. We
then have $\vol_f(A)=\int_A|h|\omega$. 
In
particular, since 
$\left\vert\int_A|h|\omega\right\vert\ge |\int_Ah\omega|$,
it follows that  $\vol_f(A)\ge|\int_A f^*(\alpha)|$.

There is a still more general notion of volume 
of submanifolds 
from geometric measure theory---the Hausdorff measure---that applies even to non-smooth maps 
and gives the same result as the definition above for smooth embeddings. 
\end{definitionsremarks}

\begin{notationremarks}\label{after refer later}
By definition, a riemannian metric on a manifold $M$ gives an inner
product, and hence a norm, on the tangent space at any point of $M$. If $\kappa :M_1\to M_2$ is a
smooth map between riemannian manifolds, then for every $\bx\in M_1$ the
derivative $d\kappa _{\bx}:T_{\bx}M_1\to T_{\kappa (\bx)}M_2$ is a
linear map, and has an operator norm
$\|d\kappa _{\bx}\|=\max_{v\in T_{\bx}M_1,\|x\|=1}\|d\kappa _x(v)\|$.

Now suppose that $\|d\kappa _{\bx}\|\le1$ for every $\bx\in M_1$. Then
for any oriented smooth manifold $S$, any smooth map $f:S\to M_1$, and any point
$\bu\in S$ such that $df_{\bu}$ is non-singular, the norms defined on
$T_{\bu}$ by the pullbacks via $f$
and $\kappa \circ f$ of the metrics $g_1$ and $g_2$ on $M_1$ and $M_2$
satisfy $\|(\kappa\circ f)^*(g_2)\|\le\|f^*(g_2)\|$. This implies in
particular that if $\beta_1$ and $\beta_2$ denote the volume forms of
the respective
pullbacks, and if for $i=1,2$  we choose an $n$-form  $\omega$ 
defining the orientation of $S$ and
write $\beta_i=h_i\omega$ for some
function $h_i$ which is positive wherever $df_{\bu}$ is non-singular,
then $h_2\le h_1$. If as in \ref{refer later} we extend the $\beta_i$
to all of $S$ by defining them to be $0$ at points where $df_{\bu}$ is
singular, and if we define the $h_i$  to be $0$ at such points as
well, then the inequality $h_2\le h_1$ holds on all of $S$; upon
integrating we conclude that for any (measurable) set $A\subset S$ we
have $\vol_{\kappa\circ f}(A)\le\vol_f(A)$, in the notation of
\ref{refer later}.

If, in addition to the
assumption that $\|d\kappa _{\bx}\|\le1$ for every $\bx\in M_1$, we
assume that $\|d\kappa _{\bx}\|<1$ for some $\bx\in f(A)$, the same
argument shows that $\vol_{\kappa\circ f}(A)<\vol_f(A)$.
\end{notationremarks}

\begin{proposition}\label{thog proj} 
Let $\Pi$ be any  totally geodesic plane in $\HH^3$ and let
$\pi:\HH^3\to\Pi$ denote  orthogonal projection. Then if
$\bx$ is any point of $\HH^3$, and $D$ denotes the hyperbolic distance from $\bx$ to
$\Pi$, we have
$\|d\pi_{\bx}\|=1/\cosh D$, where $d\pi_{\bx}:T_{\bx}\HH^3\to T_{\pi{\bx}}\Pi$
is the derivative map, and
the operator norm $\|d\pi_{\bx}\|$ is defined as in \ref{after refer later}. 
In particular we have $\|d\pi_{\bx}\|\le1$ for every $\bx\in\HH^3$,
and $\|d\pi_{\bx}\|<1$ for every $\bx\in\HH^3-\Pi$. 
\end{proposition}

\begin{proof} Because the isometries of $\mathbb{H}^3$ act
  transitively on its collection of totally geodesic planes and
  conjugate orthogonal projections to orthogonal projections, we may
  fix a particular plane $\Pi$ on which to 
establish the assertions of the proposition. 
Using the upper half-space model for $\mathbb{H}^3$, we choose $\Pi$ to be the unit hemisphere centered at $\mathbf{0}$. The orthogonal projection $\pi$ to this plane is given 
in terms of the ambient coordinates on $\mathbb{R}^3$
by
\[ \pi(x,y,z) =
  \frac{1}{1+x^2+y^2+z^2}\,\left(2x,2y,\sqrt{(1+x^2+y^2+z^2)^2-4(x^2+y^2)}\right) \]

Let us prove the first assertion of the proposition, that for any
$\bx\in\HH^3$ we have $\|d\pi_{\bx}\|=1/\cosh\redD$, where $\redD$
denotes the distance from $\bx$ to $\Pi$.

Because the stabilizer of $\Pi$ in $\mathbb{H}^3$ acts transitively on
the  points of $\Pi$, we may assume without loss of generality that
$\pi(\bx)=(0,0,1)$; after possibly modifying $\bx$ by a reflection
about $\Pi$ we may further assume that $\bx=(0,0,z)$, where $z=e^{\redD}$.

Denoting the standard basis vectors for $\mathbb{R}^3$ as $\be_1$, 
$\be_2$, $\be_3$, we have that $\{z\be_1,z\be_2,z\be_3\}$  
and $\{\be_1,\be_2\}$ are orthonormal bases for $T_{\bx}\mathbb{H}^3$,  
and $T_{\pi\bx}\Pi$, respectively. In terms of these bases, 
the derivative $d\pi_{\bx}$ of $p$ at
$\bx$
is given by
\Equation\label{the matrix}
d\pi_{\bx} = \frac{1}{\cosh \redD }\begin{pmatrix} 1& 0 & 0 \\ 0 & 1 &
    0 
    \end{pmatrix} 
.
\EndEquation

The matrix in (\ref{the matrix}) 
can be 
obtained 
by computing 
the usual Jacobian matrix at 
$\bx = (0,0,z)$, of partial derivatives of the components of $\pi$, 
applying it to the basis vectors for $T_{\bx}\mathbb{H}^3$ given above, 
expressing their images in terms of the basis vectors for $T_{\pi\bx}\Pi$, 
then substituting $e^{\redD} $ for 
$z$ and simplifying.

The expression (\ref{the matrix}) for $d\pi_{\bx}$ immediately implies 
the first assertion of the proposition, that
$\|d\pi_{\bx}\|=1/\cosh\redD$. 
\end{proof}

\begin{corollary}\label{more o'larry} The inclusion map $S\hookrightarrow M$ of a totally geodesic surface $S$ in a closed hyperbolic $3$-manifold $M$ is least-area, in the sense of \ref{refer later}. 
\end{corollary}

\begin{proof} 
Since
$S$ is least-area if each of its components is, 
we may assume that $S$ is connected. 

Let $i:S\to M$ denote
the inclusion map, and let
$p\co\widetilde{M}\to M$
denote the  covering space determined by 
$i_\sharp(\pi_1 (S))$. 
Then $i$
admits a lift
$\ti:S\to \tM$; set $\tS=\ti(S)$. 
If we write $\tM=H^3/\Gamma_0$, where $\Gamma_0$ is a discrete,
torsion-free group of isometries of $\HH^3$, and let $q:\HH^3\to\tM$
denote the quotient map, then $\Gamma_0$ leaves the plane $\Pi\doteq
q^{-1}(\tS)$ invariant, 
preserves some
component of $p^{-1}(S)$, which is a plane
$\Pi$, 
and therefore commutes with 
the orthogonal projection from $\HH^3$ to $\Pi$. 
Hence 
this projection 
induces a projection $\pi\co\widetilde{M}\to\widetilde{S}$. 
It follows from  Lemma \ref{thog proj} that
$\|d\pi_{\bx}\|\le1$ for every $\bx\in\tM$.

If $f:S\to M$ is a map homotopic to $i$, and we choose a smooth
homotopy $F\co S\times I\to M$  from $i$ to $f$, then $F$ admits a
lift $\tF:S\to\tM$. Now 
$H \doteq p\circ\pi\circ \tF$ is a homotopy from the identity map of $S$ to
$j\doteq p\circ\pi\circ\tf:S\to S$, where $\tf$ is a lift of $f$ to
$\tM$. Since $\|d\pi_{\bx}\|\le1$ for every $\bx\in\tM$, it follows from
\ref{after refer later} that $\area_{\pi\circ \tf}(S)\le\area_{\tf}(S)$.
Since $p$ is a local isometry, this inequality may be rewritten as 
$\area_j(S)\le \area_f(S)$.

Now fix an orientation of $S$. By an observation made in \ref{refer later}
 we have
$\area_j(S)\ge|\int_S j^*(\alpha)|$, where $\alpha$ is the area form
determined by the orientation of $S$. But it is a standard
consequence of Stokes's Theorem that the pull-backs of an $n$-form under
homotopic maps between closed $n$-manifolds have the
same integral. Since $j:S\to S$ is homotopic to the identity, it
follows that  $\int_S j^*(\alpha)=\int_S {\rm id}_S^*(\alpha)=\area(S)\ge0$, 
and hence that $\area_j(S)\ge\area(S)$. 
Noting that since $S$ is totally geodesic, $\area_i(S) = \area(S)$, we conclude
 that $\area_f(S)\ge\area_j(S)\ge\area_{i}(S)$.
\end{proof}

The lemma below is a variation on a standard consequence of deep results 
of Thurston \cite{Th1}, Agol-Storm-Thurston \cite{ASTD}, and Miyamoto \cite{Miy}; compare it with 
Theorems 7.2 and 9.1 of \cite{ASTD}.

\begin{lemma}\label{ASTD plus epsilon} 
Let $N$ be a  compact $3$-manifold, each of  whose components is
\simple\ and has non-empty boundary. Suppose that $N$ is equipped with a hyperbolic metric 
such that $\partial N$ is a 
minimal surface. Then
we have 
$\vol N>\voct\,\chibar(\kish N)$.
\end{lemma}

In the 
proof  of Lemma \ref{ASTD plus epsilon}
and below it,
the \textit{double} of a manifold $N$ with boundary is the manifold $\mathit{DN}$ obtained from $N\sqcup\overline{N}$, where $\overline{N}$ is a second copy of $N$, by identifying $\partial N\to\partial\overline{N}$ via the identity map. The following facts are standard and will be taken for granted. First, if $N$ is oriented then so is $\mathit{DN}$, by equipping $\overline{N}$ with the opposite orientation from that of $N$. Second, a hyperbolic structure with totally geodesic boundary on $N$ embeds isometrically into a boundaryless hyperbolic structure on $DN$ in which $\partial N$ hence sits as totally geodesic surface separating $N$ from $\overline{N}$.

\begin{proof}[Proof of Lemma \ref{ASTD plus epsilon}]
Theorem 7.2 of \cite{ASTD} asserts that
$\vol N \ge \frac{1}{2}V_3\|\mathit{DN}\|$, 
where $V_3$ is the volume of a regular ideal tetrahedron 
and 
$\|\cdot\|$ denotes the Gromov norm.

As used in the proof of \cite[Th.~9.1]{ASTD}, $\kish^0 N$ admits a hyperbolic 
structure with totally geodesic boundary, which
 by Miyamoto's universal lower bound on the volumes of manifolds with totally geodesic boundary \cite[Th.~4.2]{Miy} 
 has volume at least $\voct\,\chibar(\kish N)$. Its double  
$D\kish^0 N$ therefore has a finite-volume hyperbolic structure with volume 
at least twice this. And $D\kish^0 N$ is the interior of the submanifold 
$D\kish N$ of $\mathit{DN}$ whose frontier is a 
disjoint union of incompressible tori. It thus follows that 
$V_3\|\mathit{DN}\| \ge \vol (D\kish^0N)$
from Theorem 6.5.5 of Thurston's notes \cite{Th1}. Combining these inequalities, 
we obtain
\[ \vol N \ge \frac{1}{2}V_3\|\mathit{DN}\| 
	\ge \frac{1}{2}\vol (\textstyle{D\kish^0}N) \ge\voct\,\chibar(\kish N) \]
Note that the present result's conclusion asserts a strict inequality. If $\partial N$ is not totally geodesic 
 then by \cite[Th.~7.2]{ASTD} the leftmost inequality above is strict, and our desired 
conclusion holds. But if $\partial N$ is totally geodesic, then we claim that the rightmost inequality above 
is strict and again the desired conclusion holds; so in fact it holds unconditionally.

To prove the claim we note that since  $\partial N$ is totally
geodesic, the manifold $N$ is acylindrical. 
Hence, as pointed out in 
Subsection \ref{top back}, we have $\kish N = N$; we therefore have $\kish^0 N = N$.
Theorem 4.2 of \cite{Miy},  applied to $N$, 
 asserts 
for the rightmost inequality above 
that ``equality holds only if $N$ is 
decomposed into $T^n(0)$'s'', where in this case (for $n=3$), the truncated regular simplex $T^3(0)$ 
of edgelength $0$ is a regular ideal octahedron (compare \cite[Example 5.1]{Miy}). 
But any complete manifold that decomposes into copies of $T^3(0)$ has cusps, 
and the present $N$ is compact by hypothesis. Therefore the 
volume inequality is strict as claimed.
\end{proof}

The following 
result, Proposition \ref{from FHS}, will be seen to be a
direct consequence of deep results of Freedman--Hass--Scott \cite{FHS} 
and Schoen--Yau \cite{SchoenYau}. This result,
and its variant Proposition \ref{FHS plus epsilon}, will facilitate 
the 
applications 
of Lemma \ref{ASTD plus epsilon}.

\begin{proposition}\label{from FHS} Let $S$ be a connected incompressible surface 
in a closed, orientable hyperbolic $3$-manifold $M$. The inclusion map $S\hookrightarrow M$ 
is homotopic to a least-area immersion $f\co S\to M$ that is either (i) an embedding, or 
(ii) a two-sheeted covering map to a one-sided surface $K$ isotopic to the core of 
a twisted $I$-bundle in $M$ bounded by $S$.
\end{proposition}

\begin{proof} Note that $M$, being hyperbolic, is 
$\mathbb{P}^2$-irreducible and
  aspherical. 
  (These are topological hypotheses of the results of \cite{SchoenYau} and \cite{FHS}.)
 Since $S$ is incompressible, the inclusion map $S\hookrightarrow M$ is homotopic 
in $M$ to a smooth least-area immersion $f$, by the main 
result of \cite{SchoenYau}. Theorem 5.1 of
\cite{FHS}
(which implicitly assumes connectedness of the domain $S$) 
then asserts 
that $f$ satisfies one of the alternatives (i), (ii).
\end{proof}

In Section 7 of \cite{FHS}, generalizations of the results of the kind
that we
have summarized in Proposition \ref{from FHS} are considered. The
first four paragraphs of that section outline a proof that Proposition
\ref{from FHS} remains true if the hypothesis that $M$ is closed is
weakened. In particular, the proposition appears to remain true if $M$ is a
compact hyperbolic $3$-manifold with smooth boundary, and the mean
curvature of $\partial M$ with respect to the inward normal is
everywhere non-negative. The arguments that are indicated depend on methods developed in
\cite{MeeksYauBoundary}. Joel Hass has explained to us how these
methods can be adapted to this purpose; the details appear to be rather involved. In this paper we
need only the very special case in which $\partial M$ is connected and
totally geodesic, and we prefer to provide a complete proof of this
special result, stated as Proposition \ref{FHS plus epsilon} below, that quotes only results that are proved in detail in
the literature.

\newcommand\redS{S}

\begin{proposition}\label{FHS plus epsilon} Let $\redS$ be a connected,
incompressible 
(closed)
 surface 
in 
(the interior of)
a compact, orientable hyperbolic $3$-manifold $N$ with connected totally geodesic 
boundary. 
Then one of the following alternatives holds: 
\begin{itemize}
\item[(i-a)] $\redS$ and $\partial N$ 
cobound a submanifold of
$N$ which is a trivial $I$-bundle over a closed, connected surface;
\item[(i-b)]
the inclusion map $\redS\hookrightarrow N$ is homotopic in $N$ to a least-area embedding $f\co \redS\to N$ 
whose image is contained in  $\inter N$; or  
\item[(ii)]
the inclusion map $\redS\hookrightarrow N$ is homotopic in $N$ to a
least-area immersion which is a two-sheeted 
covering map to a one-sided surface $K\subset\inter N$, isotopic to the core of a twisted 
$I$-bundle in 
$N$ bounded by $\redS$.
\end{itemize}
\end{proposition}

\begin{proof}[Proof of Proposition \ref{FHS plus epsilon}] 
We first consider the case in which
the inclusion map $\redS\hookrightarrow DN$ 
is 
homotopic in $DN$ to a map whose image is contained in $\partial
N$. 
In this case, 
since $\redS$ 
and $\partial N$ are disjoint, it  follows from \cite[Lemma
5.3]{Waldhausen} that 
$\redS$ and $\partial N$ 
cobound a submanifold of
$DN$ which is a trivial $I$-bundle over a closed, connected surface. 
Since $\redS$ is in the interior of $N$, this submanifold is contained in $N$. 
Thus 
Alternative (i-a) of
the conclusion of the proposition holds in this case.

For the rest of the proof, we
shall
assume  that:
\begin{plainclaim}\label{assumptishment}
the inclusion map $\redS\hookrightarrow DN$ 
is not
homotopic in $DN$ to a map whose image is contained in $\partial
N$.
\end{plainclaim}

Proposition \ref{from FHS}, applied with the double
$\mathit{DN}\supset N$ playing the role of $M$, implies 
that 
the inclusion $\redS\hookrightarrow \mathit{DN}$ is homotopic in $\mathit{DN}$ 
to a least-area immersion $f\co \redS\to\mathit{DN}$ 
which is
either (i$'$) an embedding, or 
(ii$'$) a two-sheeted covering map to a one-sided surface $K\subset
DN$.  (The second alternative of the conclusion of Proposition
\ref{from FHS} is stronger than (ii$'$), but the stronger information
will not be used.)

The 
inclusion map $\partial N\hookrightarrow\mathit{DN}$ of the 
totally geodesic surface $\partial N$ 
is least-area, 
by Corollary \ref{more o'larry}, 
so 
 since
$\redS$
and $\partial N$ are disjoint,
Theorem 6.2 of \cite{FHS} further asserts that 
$f(\redS)$ is also disjoint from or
identical to $\partial N$. 
In view of \ref{assumptishment} we cannot have $f(\redS)=\partial
N$. Hence:
\Equation\label{display it}
f(\redS)\cap\partial N = \emptyset.
\EndEquation

Let us now fix 
a base point in $\partial N\subset N\subset DN$, 
and consider the covering space $p\co\widetilde{\mathit{DN}}\to\mathit{DN}$  determined by 
the image of the inclusion homomorphism $\pi_1(N)\to\pi_1(DN)$. There
is a submanifold $N_0$ of $\tDN$ which is mapped \diffeo morphically onto
$N$ by $p$; and since $\partial N$ is connected, the submanifold
$Z\doteq\tDN-\inter N_0$ is also connected. 
We claim:
\begin{plainclaim}\label{this one first} 
The inclusion $\partial N_0\hookrightarrow Z$ is a homotopy equivalence.
\end{plainclaim}

To prove \ref{this one first}, note that since 
$N$ is boundary-irreducible, the surface $\partial N_0$ is
incompressible; a priori this implies that, using a base point in  $\partial N_0$
lying over the chosen base point of $\partial N$, the group $\pi_1(\tDN)$ is
canonically identified with a free product with amalgamation
$\pi_1(N_0)\star_{\pi_1(\partial N_0)}\pi_1(Z)$. In particular, the inclusion
homomorphisms from $\pi_1(N_0)$, $\pi_1(Z)$ and $\pi_1(\partial N_0)$ to
$\pi_1(\tDN)$ are
injective, and if $A$, $B$ and $C$ denote the respective images of
these injections we have $A\cap B=C$. But by the construction of
$\tDN$ we have $A=\pi_1(\tDN)$, and hence $B=C$; that is, the inclusion
homomorphism $\pi_1(\partial N_0)\to\pi_1(Z)$ is an isomorphism. But $\tDN$ is
aspherical since $DN$ is, and since $\partial N_0$ is incompressible, $Z$ is
also aspherical. The genus-$2$ surface $\partial N_0$ is also
aspherical. 
This implies \ref{this one first}.

In particular it follows from \ref{this one first} that 
the inclusion $N_0\hookrightarrow\tDN$
is a homotopy equivalence (which could be seen more directly). 

Next, we claim: 
\begin{plainclaim}\label{stake a claim}
 $f(\redS)\subset \mathrm{int}\,N$, 
and  $f$ is homotopic 
in $\mathrm{int}\,N$
to the inclusion $\redS\hookrightarrow N$. 
\end{plainclaim}

To prove \ref{stake a claim}, we first observe that 
the inclusion map $\redS\hookrightarrow N$ lifts to an embedding $j$ of $\redS$ in $\tDN$. By the covering homotopy property of covering
spaces, $j$ is homotopic in $\tDN$ to some lift $\tf$ of
$f$. 
It follows from (\ref{display it}) that 
either
$\tf(\redS)\subset\inter N_0$ or $\tf(\redS)\subset\inter Z$. But if
$\tf(\redS)$ were contained in $\inter Z$, then 
by \ref{this one first}, 
$\tf$ would be homotopic in $\tDN$ to a map with image
contained in $\partial N_0$, and hence  $\redS\hookrightarrow \mathit{DN}$  
would be homotopic in $\mathit{DN}$ to a map with image
contained in $\partial N$; 
this contradicts \ref{assumptishment}.
It follows
that $\tf(\redS)\subset\inter N_0$. Now since $j$ and $\tf$ are homotopic
in $\tDN$ and both map $\redS$
into $\inter N_0$, and since 
we have observed that 
$N_0\hookrightarrow\tDN$
is a homotopy equivalence, 
the maps $j$ and $\tf$ are in fact homotopic in $\inter N_0$. 
This immediately
implies 
\ref{stake a claim}.

Since $f$ is a least-area map from $\redS$ to $DN$, and since $f(\redS)\subset
N$ by \ref{stake a claim}, it follows from the definition that $f$
is a 
least-area map from $\redS$ to $N$. We also know from \ref{stake a
  claim} that the least-area map $f:\redS\to N$ and the inclusion
$\redS\hookrightarrow N$ are homotopic maps from $\redS$ to $N$.

If (i$'$) holds, i.e. if $f$ is an embedding, then it follows
from \ref{stake a claim} that Alternative 
(i-b) 
of the present
proposition holds.

Now suppose that (ii$'$) holds. Thus $f$ is a two-sheeted covering map to a one-sided surface $K\subset
DN$.  According to \ref{stake a claim} we have $K\subset\inter N$.
Let $J$ be a tubular neighborhood of $K$ in  $\inter N$. Then $J$ is a twisted $I$-bundle with core $K$.
Now since $\redS$ is orientable and $K$ is not, the covering map $f:\redS\to
K$ must be equivalent to the orientation covering of $K$. It follows
that $f$ is homotopic in $J\subset\inter N$ to a diffeomorphism
$g:\redS\to\partial J$. But by \ref{stake a claim}, $f$ is homotopic 
in $\mathrm{int}\,N$
to the inclusion $\redS\hookrightarrow N$. Hence the maps $g$ and
$\redS\hookrightarrow N$, which may both be regarded as embeddings of $\redS$
in $\inter N$, are homotopic in $\inter N$. It now follows from
Corollary 5.5 of \cite{Waldhausen} that the surfaces $\redS$ and $\partial
J=g(\redS)$ are isotopic in $N$.

Thus we may fix  a self-diffeomorphism $h$ of $N$, isotopic to the
identity, such that $h(\redS)=\partial J$. If we now set $J_1=h^{-1}(J)$
and $K_1=h^{-1}(K)$, then $J_1$ is a twisted $I$-bundle whose boundary
is $\redS$, while $K_1$ is a core of $J_1$ and is isotopic to  $K$. This
gives Alternative (ii) of the conclusion of  the present proposition.
\end{proof}

The next result is a variation on [Theorem 9.1] of \cite{ASTD} in which 
a strengthened hypotheses (that $M$ is 
compact 
rather than just having
 finite volume) yields 
the strengthened conclusion of a strict inequality. 

\Theorem\label{from ast}
Let $S$ be a 
connected
incompressible surface in a 
compact, 
orientable hyperbolic
$3$-manifold $M$ 
with empty or connected, totally geodesic boundary
(so that each component of $M\cut S$ has non-empty boundary and is
\simple\ by the discussion in \ref{top back}, and hence $\kish(M\cut S)$
is defined).
Then 
we have
\Equation\label{efficacious}
\vol M>\voct\chibar(\kish(M\setminus\setminus S)).
\EndEquation
\EndTheorem

\Proof
The incompressibility of $S$, together with the hypotheses concerning
the manifold $M$, implies that $S$ has genus at least $2$.

Since $S$ is incompressible, and $M$ has at most one boundary component, 
the hypotheses of  
either 
Proposition \ref{from FHS} 
or Proposition \ref{FHS plus epsilon} hold. 
Hence either $\partial M=\emptyset$ and one of the alternatives (i),
(ii) of the conclusion of Proposition
\ref{from FHS} holds, or $\partial M$ is connected and one of the alternatives (i-a),
(i-b) or (ii) of the conclusion of Proposition
\ref{FHS plus epsilon} holds.

First consider the case in which  $\partial M$ is connected and Alternative (i-a) of Proposition
\ref{FHS plus epsilon} holds. 
Let $J$ denote
the submanifold of
$M$ which is a trivial $I$-bundle over a closed, connected surface and
has boundary $S\cup\partial M$. 
Let $N$ denote the component of $M\setminus\setminus S$ that is
distinct from $J$. 
Then since $\kish(J) = \emptyset$, and $N$ is \diffeo morphic to $M$,
we have
\[ \chibar(\kish(M\setminus\setminus S)) = \chibar(\kish N) = \chibar(\kish M). \]
Since $\partial M$ is by totally geodesic by hypothesis, it is 
least-area and hence minimal by Corollary \ref{more o'larry}. 
The conclusion thus follows from Lemma \ref{ASTD plus epsilon}.

In the rest of the proof we shall assume that either $\partial M=\emptyset$ and one of the alternatives (i),
(ii) of the conclusion of Proposition
\ref{from FHS} holds, or $\partial M$ is connected and one of the alternatives 
(i-b) or (ii) of the conclusion of Proposition
\ref{FHS plus epsilon} holds.  
Thus 
the inclusion map $S\hookrightarrow M$ is homotopic 
in $M$ to a smooth least-area immersion $f$ (which 
is therefore 
minimal) 
and has image contained in $\inter M$, 
and  $f$ is either 
(\redI) 
an
embedding, or 
(\redII)
a 
two-sheeted covering map 
to a one-sided surface $K$
isotopic to the core of a twisted $I$-bundle in $M$ bounded by $S$. 
In 
Case (\redI), 
Corollary 5.5 of \cite{Waldhausen} implies that 
$S$ is isotopic to the image of $f$, which 
in this case we denote by $K$. 
In either case,
$N \doteq M\setminus\setminus K$ 
 is 
a hyperbolic $3$-manifold with
minimal surface boundary. Furthermore, $N$ is \diffeo morphic to $M\cut S$ in Case
(\redI), while in Case (\redII) $M\cut S$ is \diffeo morphic to the disjoint
union $N\discup J$, where $J\subset M$ is the twisted $I$-bundle
bounded by $S$. 
Note also that in Case (\redII), since $J$ is an
$I$-bundle over a closed surface
and $\partial J=S$ has genus at least $2$,
the manifold
$J$ is \simple\ and has
non-empty boundary, 
and $\kish
J=\emptyset$. Hence
in either case, the components of $N$ are \simple\ $3$-manifolds with
non-empty boundary, and $\kish (M\cut S)$ is 
diffeomorphic (as a manifold with corners) 
to $\kish
N$. In particular we have
\[ \chibar(\kish N) = \chibar(\kish(M\setminus\setminus S)) .\]
The desired conclusion now follows from Lemma \ref{ASTD plus epsilon}, 
upon noting that $\vol N = \vol M$.
\EndProof

\section{The capstone volume-topology dichotomy in the geodesic boundary case}\label{whataterribletitle}

\numberwithin{para}{section}

In this section 
we prove this paper's main result for hyperbolic $3$-manifolds
with totally geodesic boundary, Theorem \ref{bounding main}, which was
stated in the Introduction. 
First, we 
strengthen Theorem 7.4 of \cite{DeSh} 
 using Proposition \ref{heeere's trimonny!} and Theorem \ref{from ast}.

\begin{theorem}\label{7.4 upgrade} Let $N$ be a compact, orientable hyperbolic $3$-manifold with $\partial N$ connected, totally geodesic, and of genus $2$, such that there is a $(1,1,1)$-hexagon in $\widetilde{N}$ and 
\[  \cosh \ell_1 < \frac{\cos (2\pi/9)}{2\cos (2\pi/9) - 1} = 1.43969... \]
Then 
either
the Heegaard genus 
$\mathrm{Hg}(N)$
is at most 
$4$ or $\mathrm{vol}(N) > 2\voct$.\end{theorem}

Recall from the Introduction that the quantity $\voct = 3.6638...$ referred to above is the volume of a regular ideal octahedron in $\mathbb{H}^3$, four times Catalan's constant.

\begin{proof} We follow the proof of \cite[Th.~7.4]{DeSh}, replacing
  its first paragraph's appeal to Propositions 6.8 and 6.9 of that
  paper by 
an appeal to Proposition \ref{heeere's trimonny!} of the present paper 
in order to produce the non-degenerate trimonic submanifold $X$ of $N$
under the weaker bound on $\cosh\ell_1$. 
In the cases that $V = \overline{N-X}$ is boundary-reducible or of the form $|\calw|$
for a book of $I$-bundles $\calw$ (see \ref{top back}), the 
previous proof's topological arguments apply verbatim to show 
that $\mathrm{Hg}(N)\le 4$.

The rest of the argument will be devoted to the case in which we have:
\begin{plainclaim}\label{this case}
$V 
\doteq
\overline{N-X}$ is boundary-irreducible and is not of the form $|\calw|$
for a book of $I$-bundles $\calw$.
\end{plainclaim}
In this case,
we will slightly strengthen 
 the geometric half of the conclusion's
dichotomy, from $\mathrm{vol}(N) > 7.32$ 
in \cite[Th.~7.4]{DeSh}
to
$\mathrm{vol}(N) > 2\voct$ here
. The proof of \cite[Th.~7.4]{DeSh} appealed to
Theorem 9.1 of \cite{ASTD}. 
 To secure the improvement, 
we will appeal to Theorem \ref{from ast} instead.

As recorded in the proof of
\cite[Th.~7.4]{DeSh}, 
the condition \ref{this case} 
implies that $T = \partial V$ is
incompressible in $N$.
According to the discussion in Subsection \ref{top back} of the
present paper, this implies that the
components of $N\cut T$, which are canonically identified with $X$
and $V$, are \simple; they obviously have non-empty boundaries, and so
$\kish V$ and $\kish X$ are defined.
It is also recorded in the proof of
\cite[Th.~7.4]{DeSh} that
$\kish V$ and $\kish X$
are non-empty. 
Therefore $\chibar(\kish(N\cut T))\ge 2$, and 
applying 
Theorem \ref{from ast} 
with $T$ playing the role of $S$ there 
gives the desired volume bound.
\end{proof}

We conclude this section with the proof of the following result, which
upgrades Theorem 1.1 of \cite{DeSh} and was stated in the Introduction.

\begin{theorem}\label{bounding main}
\ThrmBoundingMain\end{theorem}

\begin{proof} For $N$ satisfying the Theorem's hypotheses, if its universal cover $\widetilde{N}$ has no $(1,1,1)$-hexagon then by Corollary \ref{larry}, $\mathrm{vol}(N) \ge 7.4 > 2\voct$. If $\widetilde{N}$ has a $(1,1,1)$-hexagon and the length $\ell_1$ of its shortest return path satisfies $\cosh\ell_1 > 1.439$, Theorem \ref{one two three five} gives the same lower bound on $\mathrm{vol}(N)$. We are left with the case that $\widetilde{N}$ does have a $(1,1,1)$-hexagon and $\cosh\ell_1 \le 1.439$. Here Theorem \ref{7.4 upgrade} implies that since $\mathrm{Hg}(N)\ge 5$, $\mathrm{vol}(N) > 2\voct$.\end{proof}

\section{
Background for the results on closed manifolds
}\label{closed background}

\begin{definitionsconventionsremarks}\label{general}
\textnormal{
The definitions, conventions and remarks given in Subsection \ref{top
  back} will be freely used in the rest of the paper, as will the ones
given below.
}

\textnormal{
We recall that the {\it rank} of a group $\Pi$ is the minimum
cardinality of a generating set for $\Pi$. 
}

\textnormal{
As was mentioned in the Introduction, a group $\Pi$ is said to be {\it $k$-free} for a given positive integer $k$ if every subgroup of $\Pi$ whose rank is at most $k$ is free.
}

\textnormal{
A group is said to be {\it freely indecomposable} if it is not trivial
or infinite cyclic, and is not a free product of two non-trivial
subgroups.
}

\textnormal{
If $M$ is a compact, connected, orientable $3$-manifold, its Heegaard
genus $\Hg(M)$ (see \ref{top back})
is bounded below by $\rank\pi_1(M)$, which in turn is bounded below by $\dim
H_1(M;\FF_p)$ for each prime $p$. 
}

\end{definitionsconventionsremarks}

\Proposition\label{live or die}
Let $M$ be a compact, orientable $3$-manifold-with-boundary, and let
$F$ be a field. Then the dimension of the image of the inclusion
homomorphism $H_1(\partial M;F)\to H_1(M;F)$ is equal to the sum of
the genera of the components of $\partial M$.
\EndProposition

\Proof
According to \cite[Lemma 3.5]{hatcher-basic}, the dimension of the image of the boundary
homomorphism $H_2(M,\partial M;\QQ)\to H_1(\partial M;\QQ)$ is
one-half the dimension of $H_1(\partial M;\QQ)$. An examination of the
proof of [Lemma 3.5] of \cite{hatcher-basic} reveals that the proof
goes through without change if $\QQ$ is replaced by any field. Thus if
$\partial:H_2(M,\partial M;F)\to H_1(\partial M;F)$ denotes the
boundary homomorphism, the dimension of the image of $\partial$ is
$(\dim H_1(\partial M;F))/2$. But we have $\dim H_1(\partial M;F)=2G$,
where $G$ denotes the sum of
the genera of the components of $\partial M$; furthermore, by the
homology exact sequence of the pair $(M,\partial M)$, the image of
$\partial$ is the kernel of the inclusion
homomorphism $\iota:H_1(\partial M;F)\to H_1(M;F)$. Hence the kernel
of $\iota$ has dimension $G$, and since the domain of $\iota$ is a
vector space of dimension $2G$, the image of $\iota$ has dimension $G$
as well.
\EndProof

\Proposition\label{non-sing thing}
Let $k\ge3$ be an integer, and let $M$ be a closed,
orientable, hyperbolic $3$-manifold. Suppose that $\dim
H_1(M;\FF_2)\ge  \max(3k-4, 6)$, and that
$\pi_1(M)$ is not $k$-free. Then $M$ contains a closed incompressible surface
of some genus $g$ with $2\le g\le k-1$.
\EndProposition

\Proof
Proposition 8.1 of  \cite{CS_vol} includes the fact that if
 $k\ge  3$ is an integer and $M$ is a closed, orientable hyperbolic
$3$-manifold with $H_1(M;\FF_2)\ge  \max(3k-4, 6)$, then
either $\pi_1(M)$ is $k$-free, or M contains a closed incompressible
surface of genus at most $k-1$. Since the genus of an incompressible
surface in $M$ is at least $2$ by \ref{general}, the present
proposition follows.
\EndProof

\Proposition\label{3-free case}
If $M$ is a closed, orientable, hyperbolic $3$-manifold and $\pi_1(M)$ is $3$-free, then $\vol M>3.08$.
\EndProposition

\Proof
This is included in Corollary 9.3 of \cite{ACS}.
\EndProof

\Lemma\label{spex}
Let $M$ be a closed $3$-manifold. Set $V=H_1(M;\FF_2)$, and suppose that $P$ is a codimension-$2$ subspace of $V$. Set $k=dim P=(\dim V)-2$. Let $\tM$ denote the regular covering of $M$, with covering group $\ZZ/2\ZZ\times\ZZ/2\ZZ$, that is determined by $P$ (so that the normal subgroup of $\pi_1(M)$ corresponding to $\tM$ is the preimage of $P$ under the Hurewicz homomorphism $\pi_1(M)\to V$). Then $\dim H_1(\tM;\FF_2)\ge2k+1$.
\EndLemma

\Proof
Let $\Gamma_1$ denote the normal subgroup of $\pi_1(M)$ generated by
all 
commutators 
and squares. Thus $\Gamma_1$ is the kernel of the Hurewicz homomorphism $\eta:\pi_1(M)\to H_1(M;\FF_2)$.
According to the case $p=2$ of \cite[Lemma 1.5]{SW}, if $n$ is any integer less than or equal to $(\dim V)-2$,  if $E$ is any subgroup generated by $n$ elements of $\pi_1(M)$, and if $D$ denotes the subgroup $D\doteq E\Gamma_1$ of $\pi_1(M)$, we have $\dim H_1(D;\FF_2) \ge 2n+1$. To apply this, we take $n=k=(\dim V)-2$, we choose elements $x_1,\ldots,x_k$ of $\pi_1(M)$ whose images under $\eta$ form a basis of $P$, and we take $E$ to be the subgroup of $\pi_1(M)$ generated by $x_1,\ldots,x_k$. Then $D\doteq E\Gamma_1=\eta^{-1}(P)$, and hence $H_1(\tM;\FF_2)$ is isomorphic to $H_1(D;\FF_2)$. The result now follows.
\EndProof

\section{Incompressible surfaces, homology rank, and
  volume}\label{incompressible section}

\Lemma\label{teenage werewolf}
If  $\calw$ is a connected book of $I$-bundles 
(see \ref{top back}), 
each of whose pages has negative
Euler characteristic, then $\dim
H_1(|\calw|)\le2\chibar(|\calw|)+1$.
\EndLemma

\Proof
This is Lemma 2.11 of \cite{ACS}. (The
  connectedness hypothesis is missing from the statement of \cite[Lemma
  2.11]{ACS}, but it
  is used in the proof, and holds in the context of the applications given in \cite{ACS}. 
\EndProof

\Lemma\label{a beard}
Let $g$ be a positive integer, and $M$ be a closed, orientable
hyperbolic $3$-manifold that contains a closed, connected
incompressible surface of genus $g$. Suppose that 
$\Hg(M)>2g+1$. 
Then there exist a closed, connected  incompressible surface $S\subset M$ such that either 
\begin{enumerate}
\item   $\chibar(\kish(M\setminus\setminus S))\ge2$, or
\item  the surface $S$ separates $M$, and $M\setminus\setminus S$ has an acylindrical component.
\end{enumerate}
\EndLemma

\Proof
According to \cite[Proposition 13.2]{kfree-volume}, the hypotheses
imply that there is a connected incompressible surface $S\subset M$
such that either Alternative (2) of the statement of the present lemma holds, or one of
the following alternatives holds:
\begin{itemize}
\item[(1a)] the surface $S$ separates $M$, and for each component $B$ of $M\setminus\setminus S$ we have $\kish(B)\ne\emptyset$; or
\item[(1b)]  the surface $S$ does not separate $M$, and
  $\chibar(\kish(M\setminus\setminus S))\ge2h-2$, where $h$ denotes
  the genus of $S$.
\end{itemize}
(Proposition 13.2 of \cite{kfree-volume} also gives information about
the genus of $S$, which will not be needed here.)

If Alternative (1a) holds then $\kish(M\setminus\setminus S)$ has at
least two components, and according to 
\redtopback\ 
we have $\chibar(K)\ge1$ for each component $K$ of
$\kish(M\setminus\setminus S)$. Hence Alternative (1) of the
conclusion of the present lemma holds in this case.
If Alternative (1b) holds, then since the genus $h$ of $S$ is at least
$2$ by \ref{general}, we have $\chibar(\kish(M\setminus\setminus
S)\ge2$, i.e.  Alternative (1) of the present lemma holds in this
case as well.
\EndProof

The following result, like Theorem \ref{from ast}, is proved using the
techniques of \cite{ASTD}.

\Proposition\label{more from CDS}
Let $S$ be a connected, incompressible surface in a closed, orientable
hyperbolic $3$-manifold $M$ and let $A$ be an acylindrical component
of $M\setminus\setminus S$. Then
$A$ is \diffeo morphic to a hyperbolic manifold $N$ with totally geodesic
boundary, and $\vol M\ge\vol N$.
\EndProposition

\Proof
This is a formal consequence of Propositions 6.1 and 6.2 of
\cite{CDS}. Note that if $M$, $S$ and $A$ satisfy the hypotheses of
Proposition \ref{more from CDS}, then $A$ is irreducible,
boundary-irreducible and acylindrical. In Section 6 of \cite{CDS}, a certain real-valued
invariant $\geodvol A$ is defined for a compact, connected
$3$-manifold $A$ with non-empty boundary. Proposition 6.1 of
\cite{CDS} asserts that if $A$ is irreducible, boundary-irreducible
and acylindrical, then $A$ is \diffeo morphic to a hyperbolic manifold $N$ with totally geodesic
boundary, and $\vol N=\geodvol A$. Proposition 6.2 of \cite{CDS}
asserts that if $A$ is an acylindrical component of
$M\setminus\setminus S$, where  $S$ is a connected, incompressible surface in a closed, orientable
hyperbolic $3$-manifold $M$, then $\vol M\ge\geodvol A$. Proposition
\ref{more from CDS} now follows immediately.

(The actual definition of $\geodvol A$ is that it is one-half the
Gromov volume of the double of $A$. This definition of course enters
into the proofs of Propositions 6.1 and 6.2 of \cite{CDS}.)
\EndProof

\Proposition\label{who you}\WhoYouProp
\EndProposition

\Proof
Set $r=\dim
H_1(M;\FF_2)$.

We shall assume that
$\pi_1(M)$ is not $k$-free, and show that $\vol
M>2\voct$. Since $\pi_1(M)$ is not $k$-free and $r\ge \max(3k-4,6)$, it follows from Proposition
\ref{non-sing thing} that $M$ contains a closed incompressible surface
of some genus $g$ with $2\le g\le k-1$.

Since $k\ge g+1$, we have $r\ge \max(3k-4,6)>2g+1$. In particular,
we have $\Hg(M)>2g+1$ (see \ref{general}).
Hence by 
Lemma \ref{a beard}, there is a closed, connected
incompressible surface $S\subset M$ such
that either 
\begin{enumerate}
\item   $\chibar(\kish(M\setminus\setminus S))\ge2$, or
\item  the surface $S$ separates $M$, and $M\setminus\setminus S$ has an acylindrical component.
\end{enumerate}

If (1) holds, 
Theorem \ref{from ast} gives
$\vol M >2\voct$, so that the conclusion of the lemma is true in
this case. 
For the rest of the proof, we shall assume that (2) holds
but that (1) does not.

Fix an acylindrical component
$A$ of $M\setminus\setminus S$, and let $B$ denote the other component
of $M\setminus\setminus S$. Since $A$ is acylindrical, it follows from
Proposition \ref{more from CDS}
that $A$ is \diffeo morphic to a hyperbolic $3$-manifold $N$ with totally
geodesic boundary, and that
\Equation\label{goombah}
\vol M\ge\vol N.
\EndEquation

According to 
\redtopback,
we have $\chibar(K)\ge1$ for each component $K$ of
$\kish(M\setminus\setminus S)$; thus
$\chibar(\kish(M\setminus\setminus S))$ is bounded below by the number
of components of $\kish(M\setminus\setminus S)$. Since $A$ is acylindrical, we have
$\kish(A)=A\ne\emptyset$ by \redtopback. If $\kish(B)$ were also
non-empty, it would follow that $\kish(M\setminus\setminus S)$ had at
least two components, and therefore that 
$\chibar(\kish(M\setminus\setminus S))\ge2$; this would mean that
Alternative (1) above holds, a contradiction. Hence $\kish(B)=\emptyset$.

Consider the subcase in which $S$ has genus at least $3$. If $h$
denotes the genus of $S$, it follows from Theorem 5.4 of
\cite{Miy} that $\vol N$ is bounded below by $h\vol T_{\pi/(3h)}$, where
$T_{\theta}$ denotes a truncated regular simplex of dihedral angle
$\theta$, in the sense defined in \cite{Miy}. 
Proposition 1.1 of
\cite{Miy}
gives a formula for the volume of $T_\theta$ 
which is visibly 
monotone decreasing in $\theta$. Since $h\ge3$ it follows that $\vol
N\ge h\vol T_{\pi/(3h)}\ge3\vol T_{\pi/9}=10.4\ldots>2\voct$, 
and
the lemma is established in this subcase.

There remains the subcase in which $S$ has genus $2$. In the following
argument, all homology groups
in this argument will be understood to have coefficients in $\FF_p$.

Since $\kish(B)=\emptyset$, it follows from an observation made in
\ref{top back} that 
the connected manifold $B$ is the underlying
manifold of a  book of $I$-bundles $\calw$, each of whose pages has negative
Euler characteristic. Hence Lemma \ref{teenage werewolf} gives
$\dim H_1(B)\le1+2\chibar(B)$. But since $S$ has genus $2$
we have $\chibar(S)=2$, and since $S=\partial B$ we have
$\chibar(B)=\chibar(S)/2=1$. Hence $\dim H_1(B)\le3$.

Consider the Mayer-Vietoris fragment
$$
\xymatrix@C=1em{
H_1(S) \ar[r]^-j& H_1(A)\oplus H_1(B)\ar[r] &H_1(M)\ar[r]
&H_0(S)\ar[r]^-\tau& H_0(A)\oplus H_0(B).
}
$$
The homomorphism $\tau$ is injective since $S$ is connected and
$A\ne\emptyset$. Hence if $J$ denotes the image of $j$, we have an
exact sequence
$$
\xymatrix@C=1em{
0\ar[r]&J \ar[r]^-j& H_1(A)\oplus H_1(B)\ar[r] &H_1(M)\ar[r]
&0.
}
$$
The exactness of the latter sequence implies that
\Equation\label{arachnid}
\dim H_1(A)=\dim J+\dim H_1(M)-\dim H_1(B).
\EndEquation

We have seen that $\dim H_1(B)\le3$, and by definition we have $\dim
H_1(M)=r$. To estimate $\dim J$, we recall that the homomorphism
$j:H_1(S)\to H_1(A)\oplus H_1(B)$ is defined by
$j(x)=(\iota_A(x),\iota_B(x))$, where $\iota_A:H_1(S)\to H_1(A)$ and
$\iota_B:H_1(S)\to H_1(B)$ are the inclusion homomorphisms. Hence
$\dim J$ is bounded below by the dimension of the image of
$\iota_A$. According to Proposition
\ref{live or die}, 
the dimension of the image of $\iota_A$ is equal to the genus of $S$,
which is $2$. Thus $\dim J\ge2$, and (\ref{arachnid}) implies
that
$\dim H_1(A)\ge2+r-3=r-1$, or equivalently that $\dim H_1(N)\ge r-1$.

In particular,
we have $\Hg(N)\ge
r-1$ (see \ref{general}). 
Since $r\ge6$ by hypothesis, 
$\Hg(N)$ 
is in particular at least $5$. According to Theorem \ref{bounding main},
this implies that $\vol N>2\voct$.
Since $\vol M\ge\vol N$ by (\ref{goombah}), the conclusion of the
lemma follows in this final subcase.
\EndProof

\section{
Homology of manifolds with volume at most $\voct/2$
}

The following result was stated in the introduction.

\Theorem\label{main ingredient}\MainIngredThrm\EndTheorem

\Proof
We set $\Pi=\pi_1(M)$, and $V=H_1(M;\FF_2)$. We  identify  $V$ with
$H_1(\Pi;\FF_2)$.

We shall assume that $\dim V\ge5$ and show that $\vol M>\voct/2$, thus
proving the theorem.

If $\Pi$ is $3$-free, it follows from Proposition \ref{3-free case}
that $\vol M>3.08$. Since $3.08>\voct/2\ge\mu$, the conclusion holds
in this case. For the rest of the proof we shall assume that
$\Pi$ is not $3$-free. 
We fix a subgroup $\Delta$ of $\Pi$ which has
rank at most $3$ and is not free. The image $J$ of the inclusion homomorphism
$H_1(\Delta;\FF_2)\to H_1(\Pi;\FF_2)=V$ has dimension
at most $3$ since $\rank\Delta\le3$.
Since $\dim V\ge5$, there is a codimension-$2$ subspace $P$ of  $V$
containing $J$. The subspace $P$ defines a regular covering space
$\tM$ of $M$ whose covering group is isomorphic to $\ZZ/2\ZZ\times
\ZZ/2\ZZ$. Since $P\supset J$, there is a subgroup  of $\pi_1(\tM)$
isomorphic to $\Delta $. Hence $\pi_1(\tM)$ is not $3$-free.

We have $\dim P=\dim V-2\ge 3$. Hence if we set 
$\tr=\dim H_1(\tM;\FF_p)$,
then the case $p=2$ of Lemma \ref{spex} gives
$\tr\ge 2\cdot 3+1=7$. Since  in particular we have $\tr\ge6$, and $\pi_1(\tM)$ is not
$3$-free, we may apply 
Proposition 
\ref{who you}, with $k=3$ and with $\tM$ playing the role of $M$,
to deduce
that $\vol\tM>2\voct$. 
Since $\tM$ is a four-fold covering of $M$, we have $\vol M=(\vol\tM)/4>\voct/2$, as required.
\EndProof

\section{
Volumes of manifolds with small cup product rank
}\label{cup section}

The next two lemmas are needed for the proof of Theorem \ref{cup
  stuff}, which was stated in the introduction.

\Lemma\label{from cusped}
 Suppose that $M$ is a closed, aspherical $3$–manifold.
Set $r=\dim H_1(M;\FF_2)$, and let $t$ denote
the dimension of  the image of the
cup product pairing
$H^1(M;\FF_2)
\otimes 
H^1(M;\FF_2)\to H^2(M;\FF_2)$.
Then for any two-sheeted covering $\tM$ of $M$, we have
$\dim H_1(M;\FF_2)\ge2r-t-1$.
\EndLemma

\Proof
This is the case $m=1$ of \cite[Proposition 3.5]{CS_onecusp}.
\EndProof

\Lemma\label{tongue or pen}
Let $M$ be a closed, orientable hyperbolic $3$-manifold. Suppose that
$\pi_1(M)$ is $4$-free and that $\dim H_1(M;\FF_2)\ge6$. Then
$\vol(M)>3.69$.
\EndLemma

\Proof
This follows from the proof of \cite[Proposition
14.5]{kfree-volume}. The latter proposition is equivalent to the
statement that if $M$ is a closed, orientable hyperbolic $3$-manifold
such that
$\dim H_1(M;\FF_2)\ge8$, then
$\vol(M)>3.69$. In the first two paragraphs of the proof of
\cite[Proposition
14.5]{kfree-volume}, it is shown that the hypothesis 
$\dim H_1(M;\FF_2)\ge8$ implies that either $M$ contains an
incompressible surface of genus at most $3$, or $\pi_1(M)$ is
$4$-free, and it is shown that in the former case one has a stronger
conclusion than $\vol(M)>3.69$. The remainder of the proof is devoted
to the case in which $M$ satisfies the homological hypothesis and
$\pi_1(M)$ is $4$-free. However, an examination of this part of the
proof reveals that while the assumption of $4$-freeness is used in an
essential way, the only homological information that is used is that
$\dim H_1(M;\FF_2)\ge6$. Thus this portion of the proof establishes
the present lemma.

More specifically, the argument given in \cite[Proposition
14.5]{kfree-volume} for the $4$-free case is divided into four
subcases, labeled (a)--(d). The proofs of the inequality $\vol(M)>3.69$
in Subcases (a) and (d) make strong use of $4$-freeness, but do not
depend on any homological information. The proofs in Subcases (b) and
(c) are direct applications of Lemma
14.4 of \cite{kfree-volume}. That lemma does not involve $4$-freeness,
but does have a homological hypothesis, namely  $\dim H_1(M;\FF_2)\ge6$.
\EndProof

As we mentioned in the Introduction, the following theorem improves on
Theorem 1.2 of \cite{DeSh}; and the proof given here, besides
strengthening the result, provides more details than the proof in
\cite{DeSh} and corrects a citation.

\Theorem\label{cup stuff}\CupStuffThrm\EndTheorem

\Proof
We shall first prove Assertions (1) and (2) in the case where
$\pi_1(M)$ is $4$-free. In this case Assertion (1) is an immediate
consequence of Proposition 12.12 of \cite{kfree-volume}, which asserts that every closed, orientable hyperbolic $3$-manifold with $4$-free fundamental
group has volume greater than $3.57$. To prove  Assertion (2) in
this case, we note that the hypothesis of   Assertion (2) implies that
$r\ge6$, which by Lemma \ref{tongue or pen} above and the $4$-freeness
of $\pi_1(M)$ implies $\vol(M)>3.69>\voct$.

We now prove both assertions in the case where $\pi_1(M)$ is not
$4$-free. In this case we fix a subgroup $\Delta$ of $\Pi$ which has
rank at most $4$ and is not free. The image $J$ of the inclusion homomorphism
$H_1(\Delta;\FF_2)\to H_1(\Pi;\FF_2)=V$ has dimension
at most $4$ since $\rank\Delta\le4$.
Since the hypothesis of either of the assertions (1) or (2) implies $\dim V\ge5$, there is a codimension-$1$ subspace $P$ of  $V$
containing $J$. The subspace $P$ defines a two-sheeted covering space
$\tM$ of $M$. Since $P\supset J$, there is a subgroup  of $\pi_1(\tM)$
isomorphic to $\Delta $. Hence $\pi_1(\tM)$ is not $4$-free.

According to Lemma \ref{from cusped}, we have $\dim
H_1(\tM;\FF_2)\ge2r-t-1$. But the hypothesis of either of the
assertions (1) or (2) implies that $2r-t-1\ge8$, and hence
$\dim
H_1(\tM;\FF_2)\ge8$. Since $\pi_1(\tM)$ is not $4$-free, we may now
apply 
Proposition 
 \ref{who you}, taking $k=4$, and letting
$\tM$ play the role of $M$, to deduce that 
$\vol\tM>2\voct$. 
Since $\tM$ is a two-sheeted covering
of $M$, we have $\vol\tM=(\vol M)/2>\voct$. Since $\voct>3.57$, this
establishes both Assertion (1) and Assertion (2) in this case.
\EndProof

\bibliographystyle{plain}
\bibliography{new-volumes}

\end{document}